\documentclass[10pt,reqno]{amsart}

\usepackage[utf8x]{inputenc}
\usepackage[english]{babel}
\usepackage{amsmath,amsfonts,amssymb,amsthm,shuffle}
\usepackage[T1]{fontenc}
\usepackage{lmodern}

\usepackage[top=3.5cm,bottom=3.5cm,left=3.6cm,right=3.6cm]{geometry}

\usepackage[dvipsnames]{xcolor}
\usepackage[hyperindex=true,frenchlinks=true,colorlinks=true,
citecolor=Mahogany,linkcolor=DarkOrchid,urlcolor=Tan,linktocpage,
pagebackref = true]{hyperref}
\usepackage{tikz}
\usetikzlibrary{shapes}
\usetikzlibrary{fit}
\usetikzlibrary{decorations.pathmorphing}

\usepackage{dsfont}
\usepackage{wasysym}
\usepackage{cite}
\usepackage{subfigure}
\usepackage{multirow}
\usepackage{enumitem}
\usepackage{multicol}

\title[Pluriassociative algebras II]{Pluriassociative algebras II: \\
The polydendriform operad and related operads}
\keywords{Tree; Rewrite rule; Associative algebra; Operad; Koszul operad;
Koszul duality; Diassociative operad; Dendriform operad.}
\subjclass[2010]{05E99, 05C05, 18D50.}
\date{\today}
\author{Samuele Giraudo}
\address{Laboratoire d'Informatique Gaspard-Monge, Université Paris-Est
    Marne-la-Vallée, 5 boulevard Descartes, Champs-sur-Marne,
    77454 Marne-la-Vallée cedex 2, France}
\email{samuele.giraudo@u-pem.fr}

\numberwithin{equation}{subsection}
\renewcommand{\leq}{\leqslant}
\renewcommand{\geq}{\geqslant}

\newtheorem{Theoreme}{Theorem}[subsection]
\newtheorem{Proposition}[Theoreme]{Proposition}
\newtheorem{Lemme}[Theoreme]{Lemma}

\newcommand{\Aca}{\mathcal{A}}
\newcommand{\Cca}{\mathcal{C}}
\newcommand{\Dca}{\mathcal{D}}
\newcommand{\Fca}{\mathcal{F}}
\newcommand{\Gca}{\mathcal{G}}

\newcommand{\Hca}{\mathcal{H}}
\newcommand{\Oca}{\mathcal{O}}
\newcommand{\Mca}{\mathcal{M}}

\newcommand{\Vca}{\mathcal{V}}
\newcommand{\Zca}{\mathcal{Z}}
\newcommand{\Cfr}{\mathfrak{c}}
\newcommand{\Efr}{\mathfrak{e}}
\newcommand{\Rfr}{\mathfrak{r}}
\newcommand{\Sfr}{\mathfrak{s}}
\newcommand{\Tfr}{\mathfrak{t}}
\newcommand{\Gfr}{\mathfrak{g}}

\newcommand{\Ksf}{{\mathsf{K}}}

\newcommand{\K}{\mathbb{K}}
\newcommand{\EnsNat}{\mathbb{N}}

\newcommand{\GenLibre}{\mathfrak{G}}
\newcommand{\RelLibre}{\mathfrak{R}}
\newcommand{\GenDias}{\GenLibre_{\Dias_\gamma}}
\newcommand{\GenDendr}{\GenLibre_{\Dendr_\gamma}}
\newcommand{\GenTDendr}{\GenLibre_{\TDendr_\gamma}}
\newcommand{\GenAs}{\GenLibre_{\As_\gamma}}
\newcommand{\GenDAs}{\GenLibre_{\DAs_\gamma}}
\newcommand{\GenDup}{\GenLibre_{\Dup_\gamma}}

\newcommand{\RelDias}{\RelLibre_{\Dias_\gamma}}
\newcommand{\RelDendr}{\RelLibre_{\Dendr_\gamma}}
\newcommand{\RelAs}{\RelLibre_{\As_\gamma}}
\newcommand{\RelDAs}{\RelLibre_{\DAs_\gamma}}
\newcommand{\RelDup}{\RelLibre_{\Dup_\gamma}}

\newcommand{\RelTDendr}{\RelLibre_{\TDendr_\gamma}}

\newcommand{\Alg}{\Aca}
\newcommand{\AlgLibre}{\Fca}
\newcommand{\Dias}{\mathsf{Dias}}
\newcommand{\Dendr}{\mathsf{Dendr}}
\newcommand{\As}{\mathsf{As}}
\newcommand{\DAs}{\mathsf{DAs}}
\newcommand{\Dup}{\mathsf{Dup}}
\newcommand{\DupDendr}{\mathsf{D}}
\newcommand{\DeuxAs}{\mathit{2as}}
\newcommand{\Trias}{\mathsf{Trias}}
\newcommand{\TDendr}{\mathsf{TDendr}}
\newcommand{\Com}{\mathsf{Com}}
\newcommand{\Zin}{\mathsf{Zin}}
\newcommand{\Leib}{\mathsf{Leib}}
\newcommand{\Lie}{\mathsf{Lie}}

\newcommand{\OpLibre}{\mathbf{Free}}

\newcommand{\GDias}{\dashv}
\newcommand{\DDias}{\vdash}

\newcommand{\GDendrA}{\leftharpoonup}
\newcommand{\DDendrA}{\rightharpoonup}
\newcommand{\GDendr}{\prec}
\newcommand{\DDendr}{\succ}
\newcommand{\MAs}{\star}
\newcommand{\MAsA}{\triangle}
\newcommand{\MDAsA}{\wasylozenge}
\newcommand{\MDAs}{\diamond}
\newcommand{\GDup}{\hookleftarrow}
\newcommand{\DDup}{\hookrightarrow}

\newcommand{\MTDendr}{\wedge}
\newcommand{\OpAsDendrA}{\bullet}
\newcommand{\OpAsDendr}{\odot}
\newcommand{\ProdZin}{\shuffle}
\newcommand{\Min}{\downarrow}
\newcommand{\Max}{\uparrow}

\newcommand{\Rel}{\leftrightarrow}

\newcommand{\Recr}{\to}

\newcommand{\Cat}{\mathrm{cat}}
\newcommand{\Nar}{\mathrm{nar}}
\newcommand{\Schr}{\mathrm{schr}}

\newcommand{\Cacher}[1]{}
\newcommand{\Sloane}[1]{\href{http://oeis.org/#1}{{\bf #1}}}

\newcommand{\Feuille}{%
\begin{tikzpicture}
    \node(1)[Feuille]at(0,0){};
    \draw[Arete](1)--(0,.2);
\end{tikzpicture}}

\newcommand{\Noeud}{%
\begin{split}
\begin{tikzpicture}[xscale=.5,yscale=.4]
    \node[Feuille](0)at(0.50,-.75){};
    \node[Feuille](2)at(1.5,-.75){};
    \node[Noeud](1)at(1.00,0.00){};
    \draw[Arete](0)--(1);
    \draw[Arete](2)--(1);
    \node(r)at(1.00,.85){};
    \draw[Arete](r)--(1);
\end{tikzpicture}
\end{split}}

\newcommand{\NoeudTexte}{%
\raisebox{-.22em}{\scalebox{.7}{%
\begin{tikzpicture}[xscale=.4,yscale=.35]
    \node[Feuille](0)at(0.50,-.75){};
    \node[Feuille](2)at(1.5,-.75){};
    \node[Noeud](1)at(1.00,0.00){};
    \draw[Arete](0)--(1);
    \draw[Arete](2)--(1);
    \node(r)at(1.00,.85){};
    \draw[Arete](r)--(1);
\end{tikzpicture}}}}

\newcommand{\ArbreBinGDeux}[1]{%
\begin{split}
\begin{tikzpicture}[yscale=.25,xscale=.25]
    \node[Feuille](0)at(0.00,-3.33){};
    \node[Feuille](2)at(2.00,-3.33){};
    \node[Feuille](4)at(4.00,-1.67){};
    \node[Noeud](1)at(1.00,-1.67){};
    \node[Noeud](3)at(3.00,0.00){};
    \draw[Arete](0)--(1);
    \draw[Arete](1)edge[]node[EtiqArete]
            {\begin{math}#1\end{math}}(3);
    \draw[Arete](2)--(1);
    \draw[Arete](4)--(3);
    \node(r)at(3.00,1.25){};
    \draw[Arete](r)--(3);
\end{tikzpicture}
\end{split}}

\newcommand{\ArbreBinDDeux}[1]{%
\begin{split}
\begin{tikzpicture}[yscale=.25,xscale=.25]
    \node[Feuille](0)at(0.00,-1.67){};
    \node[Feuille](2)at(2.00,-3.33){};
    \node[Feuille](4)at(4.00,-3.33){};
    \node[Noeud](1)at(1.00,0.00){};
    \node[Noeud](3)at(3.00,-1.67){};
    \draw[Arete](0)--(1);
    \draw[Arete](2)--(3);
    \draw[Arete](3)edge[]node[EtiqArete]
            {\begin{math}#1\end{math}}(1);
    \draw[Arete](4)--(3);
    \node(r)at(1.00,1.25){};
    \draw[Arete](r)--(1);
\end{tikzpicture}
\end{split}}

\newcommand{\ArbreBinValue}[4]{%
\begin{split}
\begin{tikzpicture}[xscale=.5,yscale=.4]
    \node(0)at(-.50,-1.50){\begin{math}#3\end{math}};
    \node(2)at(2.50,-1.50){\begin{math}#4\end{math}};
    \node[Noeud](1)at(1.00,.50){};
    \draw[Arete](0)edge[]node[EtiqArete]
        {\begin{math}#1\end{math}}(1);
    \draw[Arete](2)edge[]node[EtiqArete]
        {\begin{math}#2\end{math}}(1);
    \node(r)at(1.00,1.5){};
    \draw[Arete](r)--(1);
\end{tikzpicture}
\end{split}}

\newcommand{\ArbreBin}[2]{%
\begin{split}
\begin{tikzpicture}[xscale=.4,yscale=.3]
    \node(0)at(-.50,-1.50){\begin{math}#1\end{math}};
    \node(2)at(2.50,-1.50){\begin{math}#2\end{math}};
    \node[Noeud](1)at(1.00,.50){};
    \draw[Arete](0)--(1);
    \draw[Arete](2)--(1);
    \node(r)at(1.00,1.75){};
    \draw[Arete](r)--(1);
\end{tikzpicture}
\end{split}}

\definecolor{Noir}{RGB}{0,0,0}
\definecolor{Blanc}{RGB}{255,255,255}
\definecolor{Rouge}{RGB}{205,35,38}
\definecolor{Bleu}{RGB}{2,60,195}
\definecolor{Vert}{RGB}{23,163,1}
\definecolor{Violet}{RGB}{181,18,225}
\definecolor{Orange}{RGB}{255,113,15}
\definecolor{Marron}{RGB}{52,46,0}

\tikzstyle{Noeud}=[circle,draw=Bleu!80,fill=Bleu!8,inner sep=1pt,
minimum size=2.5mm,line width=1.25pt,font=\scriptsize]
\tikzstyle{Arete}=[Rouge!80,cap=round,line width=1.25pt]
\tikzstyle{Feuille}=[rectangle,draw=Noir!70,fill=Noir!16,
inner sep=0pt,minimum size=1mm,line width=1.25pt]
\tikzstyle{Clair}=[fill=Blanc]
\tikzstyle{Marque1}=[draw=Vert!100,fill=Vert!30]
\tikzstyle{Marque2}=[draw=Orange!100,fill=Orange!40]
\tikzstyle{Marque3}=[draw=Rouge!100,fill=Rouge!50]
\tikzstyle{EtiqArete}=[regular polygon,regular polygon sides=6,
draw=Vert!80,fill=Vert!2,inner sep=.2pt,line width=.75pt,
minimum size=2.8mm,midway,text=Noir,font=\tiny]
\tikzstyle{NoeudSchr}=[Noeud,draw=Vert!80,fill=Vert!8]
\tikzstyle{NoeudCor}=[Noeud,draw=Marron!80,fill=Marron!8]

\begin{document}

\maketitle

%%%%%%%%%%%%%%%%%%%%%%%%%%%%%%%%%%%%%%%%%%%%%%%%%%%%%%%%%%%%%%%%%%%%%%%%
%%%%%%%%%%%%%%%%%%%%%%%%%%%%%%%%%%%%%%%%%%%%%%%%%%%%%%%%%%%%%%%%%%%%%%%%
%%%%%%%%%%%%%%%%%%%%%%%%%%%%%%%%%%%%%%%%%%%%%%%%%%%%%%%%%%%%%%%%%%%%%%%%
\begin{abstract}
    Dendriform algebras form a category of algebras recently introduced
    by Loday. A dendriform algebra is a vector space endowed with two
    nonassociative binary operations satisfying some relations. Any
    dendriform algebra is an algebra over the dendriform operad, the
    Koszul dual of the diassociative operad. We introduce here, by
    adopting the point of view and the tools offered by the theory of
    operads, a generalization on a nonnegative integer parameter
    $\gamma$ of dendriform algebras, called $\gamma$-polydendriform
    algebras, so that $1$-polydendriform algebras are dendriform
    algebras. For that, we consider the operads obtained as the Koszul
    duals of the $\gamma$-pluriassociative operads introduced by the
    author in a previous work. In the same manner as dendriform algebras are suitable
    devices to split associative operations into two parts,
    $\gamma$-polydendriform algebras seem adapted structures to split
    associative operations into $2\gamma$ operation so that some partial
    sums of these operations are associative. We provide a complete study
    of the $\gamma$-polydendriform operads, the underlying operads of
    the category of $\gamma$-polydendriform algebras. We exhibit several
    presentations by generators and relations, compute their Hilbert
    series, and construct free objects in the corresponding categories.
    We also provide consistent generalizations on a nonnegative integer
    parameter of the duplicial, triassociative and tridendriform operads,
    and of some operads of the operadic butterfly.
\end{abstract}

%%%%%%%%%%%%%%%%%%%%%%%%%%%%%%%%%%%%%%%%%%%%%%%%%%%%%%%%%%%%%%%%%%%%%%%%
%%%%%%%%%%%%%%%%%%%%%%%%%%%%%%%%%%%%%%%%%%%%%%%%%%%%%%%%%%%%%%%%%%%%%%%%
%%%%%%%%%%%%%%%%%%%%%%%%%%%%%%%%%%%%%%%%%%%%%%%%%%%%%%%%%%%%%%%%%%%%%%%%
\begin{footnotesize}
    \tableofcontents
\end{footnotesize}

%%%%%%%%%%%%%%%%%%%%%%%%%%%%%%%%%%%%%%%%%%%%%%%%%%%%%%%%%%%%%%%%%%%%%%%%
%%%%%%%%%%%%%%%%%%%%%%%%%%%%%%%%%%%%%%%%%%%%%%%%%%%%%%%%%%%%%%%%%%%%%%%%
%%%%%%%%%%%%%%%%%%%%%%%%%%%%%%%%%%%%%%%%%%%%%%%%%%%%%%%%%%%%%%%%%%%%%%%%
\section*{Introduction}
Associative algebras play an obvious and primary role in algebraic
combinatorics. In recent years, the study of natural operations on
certain sets of combinatorial objects has given rise to more or less
complicated algebraic structures on the vector spaces spanned by these
sets. A primordial point to observe is that these structures maintain
furthermore many links with combinatorics, combinatorial Hopf algebra
theory, representation theory, and theoretical physics. Let us cite for
instance the algebra of symmetric functions~\cite{Mac95} involving
integer partitions, the algebra of noncommutative symmetric
functions~\cite{GKLLRT94} involving integer compositions, the
Malvenuto-Reutenauer algebra of free quasi-symmetric functions~\cite{MR95}
(see also~\cite{DHT02}) involving permutations, the Loday-Ronco Hopf
algebra of binary trees~\cite{LR98} (see also~\cite{HNT05}), and the
Connes-Kreimer Hopf algebra of forests of rooted trees~\cite{CK98}.
\medskip

There are several ways to understand and to gather information about such
structures. A very fruitful strategy consists in splitting their associative
products $\star$ into two separate operations $\GDendr$ and $\DDendr$ in
such a way that $\star$ turns to be the sum of $\GDendr$ and $\DDendr$.
To be more precise, if $\Vca$ is a vector space endowed with an associative
product $\star$, splitting $\star$ consists in providing two operations
$\GDendr$ and $\DDendr$ defined on $\Vca$ and such that for all elements
$x$ and $y$ of $\Vca$,
\begin{equation}
    x \star y = x \GDendr y + x \DDendr y.
\end{equation}
This splitting property is more concisely denoted by
\begin{equation}
    \star = \GDendr + \DDendr.
\end{equation}
One of the most obvious example occurs by considering the shuffle product
on words. Indeed, this product can be separated into two operations
according to the origin (first or second operand) of the last letter of
the words appearing in the result~\cite{Ree58}. Other main examples
include the split of the shifted shuffle product of permutations of the
Malvenuto-Reutenauer Hopf algebra and of the product of binary trees of
the Loday-Ronco Hopf algebra~\cite{Foi07}. The original formalization
and the germs of generalization of these notions, due to Loday~\cite{Lod01},
lead to the introduction of dendriform algebras. Dendriform algebras
are vector spaces endowed with two operations $\GDendr$ and $\DDendr$
so that $\GDendr + \DDendr$ is associative and satisfy few other relations.
Since any dendriform algebra is a quotient of a certain free dendriform
algebra, the study of free dendriform algebras is worthwhile. Besides,
the description of free dendriform algebras has a nice combinatorial
interpretation involving binary trees and shuffle of binary trees.
\medskip

In recent years, several generalizations of dendriform algebras were
introduced and studied. Among these, one can cite dendriform
trialgebras~\cite{LR04}, quadri-algebras~\cite{AL04},
ennea-algebras~\cite{Ler04}, $m$-dendriform algebras of Leroux~\cite{Ler07},
and $m$-dendriform algebras of Novelli~\cite{Nov14}, all providing new
ways to split associative products into more than two pieces. Besides,
free objects in the corresponding categories of these algebras can be
described by relatively complex combinatorial objects and more or less
tricky operations on these. For instance, free dendriform trialgebras
involve Schröder trees, free quadri-algebras involve noncrossing connected
graphs on a circle, and free $m$-dendriform algebras of Leroux and free
$m$-dendriform algebras of Novelli involves planar rooted trees where
internal nodes have a constant number of children.
\medskip

The theory of operads (see~\cite{LV12} for a complete exposition and
also~\cite{Cha08}) seems to be one of the best tools to put all these
algebraic structures under a same roof. Informally, an operad is a space
of abstract operators that can be composed. The main interest of this
theory is that any operad encodes a category of algebras and working with
an operad amounts to work with the algebras all together of this category.
Moreover, this theory gives a nice translation of connections that may
exist between {\em a priori} two very different sorts of algebras. Indeed,
any morphism between operads gives rise to a functor between the both
encoded categories. We have to point out that operads were first
introduced in the context of algebraic topology~\cite{May72,BV73} but
they are more and more present in combinatorics~\cite{Cha08}.
\medskip

The first goal of this work is to define and justify a new generalization
of dendriform algebras. Our long term primary objective is to develop
new implements to split associative products in smaller pieces. Our main
tool is the Koszul duality of operads, an important part of the theory
introduced by Ginzburg and Kapranov~\cite{GK94}. We use the approach
consisting in considering the diassociative operad $\Dias$~\cite{Lod01},
the Koszul dual of the dendriform operad $\Dendr$, rather that focusing
on $\Dendr$. For this, we rely on the definition of a generalization
$\Dias_\gamma$ on a nonnegative integer parameter $\gamma$ of the
diassociative operad introduced by the author in~\cite{GirI}. These
operads, called $\gamma$-pluriassociative operads, satisfy several
properties and are among other set-operads and Koszul operads. We
introduce in the present work the operads $\Dendr_\gamma$ as the Koszul
dual of the operads~$\Dias_\gamma$.
\medskip

The operads $\Dendr_\gamma$ are the underlying operads of the category
of $\gamma$-polydendriform algebras, that are algebras with $2\gamma$
operations $\GDendrA_a$, $\DDendrA_a$, $a \in [\gamma]$, satisfying some
relations. Free objects in these categories involve binary trees where
all edges connecting two internal nodes are labeled on $[\gamma]$ and the
computation of a product of two binary trees admits an inductive
description. Moreover, the introduction of $\gamma$-polydendriform
algebras offers to split an associative product~$\star$ by
\begin{equation}
    \star = \GDendrA_1 + \DDendrA_1 + \dots + \GDendrA_\gamma
    + \DDendrA_\gamma,
\end{equation}
with, among others, the stiffening conditions that all partial sums
\begin{equation}
    \GDendrA_1 + \DDendrA_1 + \dots + \GDendrA_a + \DDendrA_a
\end{equation}
are associative for all $a \in \{1, \dots, \gamma\}$. Moreovoer, this
work naturally leads to the consideration and the definition of numerous
new operads. Table~\ref{tab:operades_introduites_II} summarizes some
information about these.
\begin{table}
    \centering
    \begin{tabular}{c|c|c|c}
        Operad & Objects & Dimensions & Symm. \\ \hline \hline
        $\Dendr_\gamma$ & $\gamma$-edge valued binary trees
            & $\gamma^{n - 1} \frac{1}{n + 1} \binom{2n}{n}$ & No \\ \hline
        $\As_\gamma$ & $\gamma$-corollas & $\gamma$ & No \\ \hline
        $\DAs_\gamma$ & $\gamma$-alternating Schröder trees
            & $\sum\limits_{k = 0}^{n - 2} \gamma^{k + 1}
            (\gamma - 1)^{n - k - 2} \frac{1}{k + 1} \binom{n - 2}{k}
            \binom{n - 1}{k}$ & No \\ \hline
        $\Dup_\gamma$ & $\gamma$-edge valued binary trees
            & $\gamma^{n - 1} \frac{1}{n + 1} \binom{2n}{n}$ & No \\ \hline
        $\TDendr_\gamma$ & $\gamma$-edge valued Schröder trees
            & $\sum\limits_{k = 0}^{n - 1} (\gamma + 1)^k \gamma^{n - k - 1}
            \frac{1}{k + 1} \binom{n - 1}{k} \binom{n}{k}$ & No \\ \hline
        $\Com_\gamma$ & --- & --- & Yes \\ \hline
        $\Zin_\gamma$ & --- & --- & Yes
    \end{tabular}
    \bigskip
    \caption{The main operads defined in this paper. All these operads
    depend on a nonnegative integer parameter $\gamma$. The shown
    dimensions are the ones of the homogeneous components of
    arities $n \geq 2$ of the operads.}
    \label{tab:operades_introduites_II}
\end{table}
\medskip

This article is organized as follows.
Section~\ref{sec:preliminaires_Koszul_Dendr} contains the definition
of the Koszul duality for operads and gives some recalls about
the dendriform operad and dendriform algebras.
\medskip

Then, the operad $\Dendr_\gamma$ is introduced in
Section~\ref{sec:dendr_gamma} as the Koszul dual of $\Dias_\gamma$
(Theorem~\ref{thm:presentation_dendr_gamma}). Since $\Dias_\gamma$ is a
Koszul operad~\cite{GirI}, $\Dendr_\gamma$ also is, and then, by using
results of Ginzburg and Kapranov~\cite{GK94}, the alternating versions of
the Hilbert series of $\Dias_\gamma$ and $\Dendr_\gamma$ are the inverses
for each other for series composition. This, toghether with the expression
for the Hilbert series of $\Dias_\gamma$ established in~\cite{GirI},
leads to an expression for the Hilbert series of $\Dendr_\gamma$
(Proposition~\ref{prop:serie_hilbert_dendr_gamma}). Motivated by the
knowledge of the dimensions of $\Dendr_\gamma$, we consider binary trees
where internal edges are labelled on $\{1, \dots, \gamma\}$, called
$\gamma$-edge valued binary trees. These trees form a generalization of
the common binary trees indexing the bases of $\Dendr$, and index the
bases of $\Dendr_\gamma$. We continue the study of this operad by
providing a new presentation obtained by considering the Koszul dual of
$\Dias_\gamma$ over its $\Ksf$-basis, introduced in~\cite{GirI}
(Theorem~\ref{thm:autre_presentation_dendr_gamma}).
This presentation of $\Dendr_\gamma$ is very compact since its space of
relations can be expressed only by three sorts of relations
(\eqref{equ:relation_dendr_gamma_1_concise},
\eqref{equ:relation_dendr_gamma_2_concise},
and~\eqref{equ:relation_dendr_gamma_3_concise}), each one involving two
or three terms. We also describe all the associative elements of
$\Dendr_\gamma$ over its two bases
(Propositions~\ref{prop:operateur_associatif_dendr_gamma_autre},
\ref{prop:operateur_associatif_dendr_gamma},
and~\ref{prop:description_operateurs_associatifs_dendr_gamma}). We end
this section by constructing the free $\gamma$-polydendriform algebra
over one generator (Theorem~\ref{thm:algebre_dendr_gamma_libre}). Its
underlying vector space is the vector space of the $\gamma$-edge valued
binary trees and is endowed with $2 \gamma$ products described by
induction. These products are kinds of shuffle of trees, generalizing the
shuffle of trees introduced by Loday~\cite{Lod01} intervening in the
construction of free dendriform algebras.
\medskip

Section~\ref{sec:as_gamma} extends a part of the operadic
butterfly~\cite{Lod01,Lod06}, a diagram of operads gathering the most
classical ones together, including the diassociative, dendriform, and
associative operads. To extends this diagram into our context, we
introduce a generalization $\As_\gamma$ on a nonnegative integer
parameter $\gamma$ of the associative operad $\As$. This operad, called
$\gamma$-multiassociative operad, has $\gamma$ associative generating
operations, subjected to precise relations. We prove that this operad
can be seen as a vector space of corollas labeled on $\{1, \dots, \gamma\}$
and that
is Koszul (Proposition~\ref{prop:realisation_koszulite_as_gamma}).
Unlike the associative operad which is self-dual for Koszul duality,
$\As_\gamma$ is not when $\gamma \geq 2$. The Koszul dual of $\As_\gamma$,
denoted by $\DAs_\gamma$, is described by its presentation
(Proposition~\ref{prop:presentation_as_gamma_duale}) and is realized by
means of $\gamma$-alternating Schröder trees, that are Schröder trees
where internal nodes are labeled on $\{1, \dots, \gamma\}$ with an alternating
condition (Proposition~\ref{prop:realisation_das_gamma}). In passing, we
provide an alternative and simpler basis for the space of relations of
$\DAs_\gamma$ than the one obtained directly by considering the Koszul
dual of $\As_\gamma$ (Proposition~\ref{prop:autre_presentation_das_gamma}).
We end this section by establishing a new version of the diagram gathering
the diassociative, dendriform, and associative operads for the operads
$\Dias_\gamma$, $\As_\gamma$, $\DAs_\gamma$, and $\Dendr_\gamma$
(Theorem~\ref{thm:diagramme_dias_as_das_dendr_gamma}) by defining
appropriate morphisms between these.
\medskip

Finally, in Section~\ref{sec:generalisation_supplementaires}, we sustain
our previous ideas to propose generalizations on a nonnegative integer
parameter $\gamma$ of some more operads. We start by proposing a new operad
$\Dup_\gamma$ generalizing the duplicial operad~\cite{Lod08}, called
$\gamma$-multiplicial operad. We prove that $\Dup_\gamma$ is Koszul and,
like the bases of $\Dendr_\gamma$, that the bases of $\Dup_\gamma$ are
indexed by $\gamma$-edge valued binary trees
(Proposition~\ref{prop:proprietes_dup_gamma}). The operads $\Dendr_\gamma$
and $\Dup_\gamma$ are nevertheless not isomorphic because there are
$2\gamma$ associative elements in $\Dup_\gamma$
(Proposition~\ref{prop:description_operateurs_associatifs_dup_gamma})
against only $\gamma$ in $\Dendr_\gamma$. Then, the free
$\gamma$-multiplicial algebra over one generator is constructed
(Theorem~\ref{thm:algebre_dup_gamma_libre}). Its underlying vector space
is the vector space of the $\gamma$-edge valued binary trees and is
endowed with $2 \gamma$ products, similar to the over and under products
on binary trees of Loday and Ronco~\cite{LR02}. Next, by using almost
the same tools as the ones used in Section~\ref{sec:dendr_gamma}, we
propose a generalization $\TDendr_\gamma$ of the tridendriform operad
$\TDendr$~\cite{LR04}, called $\gamma$-polytridendriform operad. The
operad $\TDendr_\gamma$ is defined as the Koszul dual of the
$\gamma$-pluritridendriform operad
$\Trias_\gamma$, introduced by the author in~\cite{GirI}. We obtain
a presentation of $\TDendr_\gamma$
(Theorem~\ref{thm:presentation_tdendr_gamma}) and an expression for
its Hilbert series
(Proposition~\ref{prop:serie_hilbert_tdendr_gamma}). The dimensions
of $\TDendr_\gamma$ thus obtained lead to establish the fact that the
bases of $\TDendr_\gamma$ are indexed by $\gamma$-edge valued Schröder
trees, that are Schröder trees where internal edges are labelled on
$\{1, \dots, \gamma\}$. We end this work by providing generalizations on a
nonnegative integer parameter $\gamma$ integer generalization of all the
operads intervening in the operadic butterfly. We then define the
operads $\Com_\gamma$, $\Lie_\gamma$, $\Zin_\gamma$, and $\Leib_\gamma$,
that are respective generalizations of the commutative operad, the Lie
operad, the Zinbiel operad~\cite{Lod95} and the Leibniz
operad~\cite{Lod93}. We provide analogous versions for our context of
the arrows between the commutative operad and the Zinbiel operad
(Proposition~\ref{prop:morphism_com_zin_gamma}), and between the
dendriform operad and the Zinbiel operad
(Proposition~\ref{prop:morphism_dendr_zin_gamma}).
\bigskip

{\it Acknowledgements.} The author would like to thank, for interesting
discussions, Jean-Christophe Novelli about Koszul duality for operads
and Vincent Vong about strategies for constructing free objects in the
categories encoded by operads. The author thanks also Matthieu
Josuat-Vergès and Jean-Yves-Thibon for their pertinent remarks and
questions about this work when it was in progress. Finally, the author
warmly thanks the referee for his very careful reading and his
suggestions, improving the quality of the paper.
\bigskip

{\it Notations and general conventions.}
All the algebraic structures of this article have a field of characteristic
zero $\K$ as ground field. For any integers $a$ and $c$, $[a, c]$ denotes
the set $\{b \in \EnsNat : a \leq b \leq c\}$ and $[n]$, the set $[1, n]$.
We use in all this paper the notations introduced in
Section~1 of~\cite{GirI}.
\medskip

%%%%%%%%%%%%%%%%%%%%%%%%%%%%%%%%%%%%%%%%%%%%%%%%%%%%%%%%%%%%%%%%%%%%%%%%
%%%%%%%%%%%%%%%%%%%%%%%%%%%%%%%%%%%%%%%%%%%%%%%%%%%%%%%%%%%%%%%%%%%%%%%%
%%%%%%%%%%%%%%%%%%%%%%%%%%%%%%%%%%%%%%%%%%%%%%%%%%%%%%%%%%%%%%%%%%%%%%%%
\section{Preliminaries: Koszul duality and the dendriform operad}%
\label{sec:preliminaires_Koszul_Dendr}
In the present preliminary section, we will recall the notion of Koszul
duality and several properties of the dendriform operad, the Koszul dual
of the diassociative operad (see Section~1.3 of~\cite{GirI}).
\medskip

%%%%%%%%%%%%%%%%%%%%%%%%%%%%%%%%%%%%%%%%%%%%%%%%%%%%%%%%%%%%%%%%%%%%%%%%
%%%%%%%%%%%%%%%%%%%%%%%%%%%%%%%%%%%%%%%%%%%%%%%%%%%%%%%%%%%%%%%%%%%%%%%%
\subsection{Koszul duality}%
\label{subsec:dual_de_Koszul}
In~\cite{GK94}, Ginzburg and Kapranov extended the notion of Koszul
duality of quadratic associative algebras to quadratic operads. Starting
with a binary and quadratic operad $\Oca$ admitting a presentation
$(\GenLibre, \RelLibre)$, the {\em Koszul dual} of $\Oca$ is the operad
$\Oca^!$, isomorphic to the operad admitting the presentation
$\left(\GenLibre, \RelLibre^\perp\right)$ where $\RelLibre^\perp$ is the
annihilator of $\RelLibre$ in $\OpLibre(\GenLibre)$ with respect to the
scalar product
\begin{equation}
    \langle -, - \rangle :
    \OpLibre(\GenLibre)(3) \otimes \OpLibre(\GenLibre)(3) \to \K
\end{equation}
linearly defined, for all $x, x', y, y' \in \GenLibre(2)$, by
\begin{equation}
    \left\langle x \circ_i y, x' \circ_{i'} y' \right\rangle :=
    \begin{cases}
        1 & \mbox{if }
            x = x', y = y', \mbox{ and } i = i' = 1, \\
        -1 & \mbox{if }
            x = x', y = y', \mbox{ and } i = i' = 2, \\
        0 & \mbox{otherwise}.
    \end{cases}
\end{equation}
Then, knowing a presentation of $\Oca$, one can compute a presentation
of~$\Oca^!$.
\medskip

Furthermore, when $\Oca$ and $\Oca^!$ are two operads Koszul dual one of
the other, and moreover, when they are Koszul operads and admit Hilbert
series, their Hilbert series satisfy~\cite{GK94}
\begin{equation} \label{equ:relation_series_hilbert_operade_duale}
    \Hca_\Oca\left(-\Hca_{\Oca^!}(-t)\right) = t.
\end{equation}
We shall make use of~\eqref{equ:relation_series_hilbert_operade_duale}
to compute the dimensions of Koszul operads defined as Koszul duals of
known ones.
\medskip

%%%%%%%%%%%%%%%%%%%%%%%%%%%%%%%%%%%%%%%%%%%%%%%%%%%%%%%%%%%%%%%%%%%%%%%%
%%%%%%%%%%%%%%%%%%%%%%%%%%%%%%%%%%%%%%%%%%%%%%%%%%%%%%%%%%%%%%%%%%%%%%%%
\subsection{Dendriform operad}%
\label{subsec:dendr}
We recall here the definitions and some properties of the dendriform
operad.
\medskip

The {\em dendriform operad} $\Dendr$ was introduced by
Loday~\cite{Lod01}. It is the operad admitting the presentation
$\left(\GenLibre_{\Dendr}, \RelLibre_{\Dendr}\right)$ where
$\GenLibre_{\Dendr} := \GenLibre_{\Dendr}(2) := \{\GDendr, \DDendr\}$
and $\RelLibre_{\Dendr}$ is the vector space generated by
\begin{subequations}
\begin{equation} \label{equ:relation_dendr_1}
    \GDendr \circ_1 \DDendr - \DDendr \circ_2 \GDendr,
\end{equation}
\begin{equation} \label{equ:relation_dendr_2}
    \GDendr \circ_1 \GDendr -
    \GDendr \circ_2 \GDendr -
    \GDendr \circ_2 \DDendr,
\end{equation}
\begin{equation} \label{equ:relation_dendr_3}
    \DDendr \circ_1 \GDendr +
    \DDendr \circ_1 \DDendr -
    \DDendr \circ_2 \DDendr.
\end{equation}
\end{subequations}
Note that $\Dendr$ is a binary and quadratic operad.
\medskip

This operad admits a quite complicated realization~\cite{Lod01}. For all
$n \geq 1$, the $\Dendr(n)$ are vector spaces of binary trees with $n$
internal nodes. The partial composition of two binary trees can be
described by means of intervals of the Tamari order~\cite{HT72}, a
partial order relation involving binary trees. This realization shows
that $\dim \Dendr(n) = \Cat(n)$ where
\begin{equation}
    \Cat(n) := \frac{1}{n + 1} \binom{2n}{n}
\end{equation}
is the $n$th {\em Catalan number}, counting the binary trees with
respect to their number of internal nodes. Therefore, the Hilbert series
of $\Dendr$ satisfies
\begin{equation}
    \Hca_\Dendr(t) = \frac{1 - \sqrt{1 - 4t} - 2t}{2t}.
\end{equation}
\medskip

Throughout this article, we shall graphically represent binary trees in
a slightly different manner than syntax trees. We represent the leaves
of binary trees by squares $\Feuille$, internal nodes by circles
\tikz{\node[Noeud]{};}, and edges by thick segments
\tikz{\draw[Arete](0,0)--(0,.27);}.
\medskip

From the presentation of $\Dendr$, we deduce that any $\Dendr$-algebra,
also called {\em dendriform algebra}, is a vector space $\Alg_\Dendr$
endowed with linear operations $\GDendr$ and $\DDendr$ satisfying the
relations encoded by~\eqref{equ:relation_dendr_1}---%
\eqref{equ:relation_dendr_3}. Classical examples of dendriform algebras
include Rota-Baxter algebras~\cite{Agu00} and shuffle algebras~\cite{Lod01}.
\medskip

The operation obtained by summing $\GDendr$ and $\DDendr$ is associative.
Therefore, we can see a dendriform algebra as an associative algebra in
which its associative product has been split into two parts satisfying
Relations~\eqref{equ:relation_dendr_1}, \eqref{equ:relation_dendr_2},
and~\eqref{equ:relation_dendr_3}. More precisely, we say that an
associative algebra $\Alg$ admits a {\em dendriform structure} if there
exist two nonzero binary operations $\GDendr$ and $\DDendr$ such that
the associative operation $\MAs$ of $\Alg$ satisfies
$\MAs = \GDendr + \DDendr$, and $\Alg$ endowed with the operations
$\GDendr$ and $\DDendr$, is a dendriform algebra
\medskip

The free dendriform algebra $\AlgLibre_\Dendr$ over one generator is the
vector space $\Dendr$ of binary trees with at least one internal node
endowed with the linear operations
\begin{equation}
    \GDendr, \DDendr :
    \AlgLibre_\Dendr \otimes \AlgLibre_\Dendr \to \AlgLibre_\Dendr,
\end{equation}
defined recursively, for any binary tree $\Sfr$ with at least one
internal node, and binary trees $\Tfr_1$ and $\Tfr_2$ by
\begin{equation}
    \Sfr \GDendr \Feuille
    := \Sfr =:
    \Feuille \DDendr \Sfr,
\end{equation}
\begin{equation}
    \Feuille \GDendr \Sfr := 0 =: \Sfr \DDendr \Feuille,
\end{equation}
\begin{equation}
    \ArbreBin{\Tfr_1}{\Tfr_2} \GDendr \Sfr :=
    \ArbreBin{\Tfr_1}{\Tfr_2 \GDendr \Sfr}
    + \ArbreBin{\Tfr_1}{\Tfr_2 \DDendr \Sfr}\,,
\end{equation}
\begin{equation}
    \begin{split} \Sfr \DDendr \end{split}
    \ArbreBin{\Tfr_1}{\Tfr_2} :=
    \ArbreBin{\Sfr \DDendr \Tfr_1}{\Tfr_2}
    + \ArbreBin{\Sfr \GDendr \Tfr_1}{\Tfr_2}.
\end{equation}
Note that neither $\Feuille \GDendr \Feuille$ nor
$\Feuille \DDendr \Feuille$ are defined.
\medskip

We have for instance,
\begin{equation}
    \begin{split}
    \begin{tikzpicture}[xscale=.2,yscale=.16]
        \node[Feuille](0)at(0.00,-4.50){};
        \node[Feuille](2)at(2.00,-4.50){};
        \node[Feuille](4)at(4.00,-6.75){};
        \node[Feuille](6)at(6.00,-6.75){};
        \node[Feuille](8)at(8.00,-4.50){};
        \node[Noeud](1)at(1.00,-2.25){};
        \node[Noeud](3)at(3.00,0.00){};
        \node[Noeud](5)at(5.00,-4.50){};
        \node[Noeud](7)at(7.00,-2.25){};
        \draw[Arete](0)--(1);
        \draw[Arete](1)--(3);
        \draw[Arete](2)--(1);
        \draw[Arete](4)--(5);
        \draw[Arete](5)--(7);
        \draw[Arete](6)--(5);
        \draw[Arete](7)--(3);
        \draw[Arete](8)--(7);
        \node(r)at(3.00,2.25){};
        \draw[Arete](r)--(3);
    \end{tikzpicture}
    \end{split}
    \GDendr
    \begin{split}
    \begin{tikzpicture}[xscale=.2,yscale=.2]
        \node[Feuille](0)at(0.00,-3.50){};
        \node[Feuille](2)at(2.00,-5.25){};
        \node[Feuille](4)at(4.00,-5.25){};
        \node[Feuille](6)at(6.00,-1.75){};
        \node[Noeud](1)at(1.00,-1.75){};
        \node[Noeud](3)at(3.00,-3.50){};
        \node[Noeud](5)at(5.00,0.00){};
        \draw[Arete](0)--(1);
        \draw[Arete](1)--(5);
        \draw[Arete](2)--(3);
        \draw[Arete](3)--(1);
        \draw[Arete](4)--(3);
        \draw[Arete](6)--(5);
        \node(r)at(5.00,1.75){};
        \draw[Arete](r)--(5);
    \end{tikzpicture}
    \end{split}
    =
    \begin{split}
    \begin{tikzpicture}[xscale=.18,yscale=.13]
        \node[Feuille](0)at(0.00,-5.00){};
        \node[Feuille](10)at(10.00,-12.50){};
        \node[Feuille](12)at(12.00,-12.50){};
        \node[Feuille](14)at(14.00,-7.50){};
        \node[Feuille](2)at(2.00,-5.00){};
        \node[Feuille](4)at(4.00,-7.50){};
        \node[Feuille](6)at(6.00,-7.50){};
        \node[Feuille](8)at(8.00,-10.00){};
        \node[Noeud](1)at(1.00,-2.50){};
        \node[Noeud](11)at(11.00,-10.00){};
        \node[Noeud](13)at(13.00,-5.00){};
        \node[Noeud](3)at(3.00,0.00){};
        \node[Noeud](5)at(5.00,-5.00){};
        \node[Noeud](7)at(7.00,-2.50){};
        \node[Noeud](9)at(9.00,-7.50){};
        \draw[Arete](0)--(1);
        \draw[Arete](1)--(3);
        \draw[Arete](10)--(11);
        \draw[Arete](11)--(9);
        \draw[Arete](12)--(11);
        \draw[Arete](13)--(7);
        \draw[Arete](14)--(13);
        \draw[Arete](2)--(1);
        \draw[Arete](4)--(5);
        \draw[Arete](5)--(7);
        \draw[Arete](6)--(5);
        \draw[Arete](7)--(3);
        \draw[Arete](8)--(9);
        \draw[Arete](9)--(13);
        \node(r)at(3.00,2.50){};
        \draw[Arete](r)--(3);
    \end{tikzpicture}
    \end{split}
    +
    \begin{split}
    \begin{tikzpicture}[xscale=.18,yscale=.13]
        \node[Feuille](0)at(0.00,-5.00){};
        \node[Feuille](10)at(10.00,-12.50){};
        \node[Feuille](12)at(12.00,-12.50){};
        \node[Feuille](14)at(14.00,-5.00){};
        \node[Feuille](2)at(2.00,-5.00){};
        \node[Feuille](4)at(4.00,-10.00){};
        \node[Feuille](6)at(6.00,-10.00){};
        \node[Feuille](8)at(8.00,-10.00){};
        \node[Noeud](1)at(1.00,-2.50){};
        \node[Noeud](11)at(11.00,-10.00){};
        \node[Noeud](13)at(13.00,-2.50){};
        \node[Noeud](3)at(3.00,0.00){};
        \node[Noeud](5)at(5.00,-7.50){};
        \node[Noeud](7)at(7.00,-5.00){};
        \node[Noeud](9)at(9.00,-7.50){};
        \draw[Arete](0)--(1);
        \draw[Arete](1)--(3);
        \draw[Arete](10)--(11);
        \draw[Arete](11)--(9);
        \draw[Arete](12)--(11);
        \draw[Arete](13)--(3);
        \draw[Arete](14)--(13);
        \draw[Arete](2)--(1);
        \draw[Arete](4)--(5);
        \draw[Arete](5)--(7);
        \draw[Arete](6)--(5);
        \draw[Arete](7)--(13);
        \draw[Arete](8)--(9);
        \draw[Arete](9)--(7);
        \node(r)at(3.00,2.50){};
        \draw[Arete](r)--(3);
    \end{tikzpicture}
    \end{split}
    +
    \begin{split}
    \begin{tikzpicture}[xscale=.18,yscale=.13]
        \node[Feuille](0)at(0.00,-5.00){};
        \node[Feuille](10)at(10.00,-10.00){};
        \node[Feuille](12)at(12.00,-10.00){};
        \node[Feuille](14)at(14.00,-5.00){};
        \node[Feuille](2)at(2.00,-5.00){};
        \node[Feuille](4)at(4.00,-12.50){};
        \node[Feuille](6)at(6.00,-12.50){};
        \node[Feuille](8)at(8.00,-10.00){};
        \node[Noeud](1)at(1.00,-2.50){};
        \node[Noeud](11)at(11.00,-7.50){};
        \node[Noeud](13)at(13.00,-2.50){};
        \node[Noeud](3)at(3.00,0.00){};
        \node[Noeud](5)at(5.00,-10.00){};
        \node[Noeud](7)at(7.00,-7.50){};
        \node[Noeud](9)at(9.00,-5.00){};
        \draw[Arete](0)--(1);
        \draw[Arete](1)--(3);
        \draw[Arete](10)--(11);
        \draw[Arete](11)--(9);
        \draw[Arete](12)--(11);
        \draw[Arete](13)--(3);
        \draw[Arete](14)--(13);
        \draw[Arete](2)--(1);
        \draw[Arete](4)--(5);
        \draw[Arete](5)--(7);
        \draw[Arete](6)--(5);
        \draw[Arete](7)--(9);
        \draw[Arete](8)--(7);
        \draw[Arete](9)--(13);
        \node(r)at(3.00,2.50){};
        \draw[Arete](r)--(3);
    \end{tikzpicture}
    \end{split}\,,
\end{equation}
and
\begin{equation}
    \begin{split}
    \begin{tikzpicture}[xscale=.2,yscale=.16]
        \node[Feuille](0)at(0.00,-4.50){};
        \node[Feuille](2)at(2.00,-4.50){};
        \node[Feuille](4)at(4.00,-6.75){};
        \node[Feuille](6)at(6.00,-6.75){};
        \node[Feuille](8)at(8.00,-4.50){};
        \node[Noeud](1)at(1.00,-2.25){};
        \node[Noeud](3)at(3.00,0.00){};
        \node[Noeud](5)at(5.00,-4.50){};
        \node[Noeud](7)at(7.00,-2.25){};
        \draw[Arete](0)--(1);
        \draw[Arete](1)--(3);
        \draw[Arete](2)--(1);
        \draw[Arete](4)--(5);
        \draw[Arete](5)--(7);
        \draw[Arete](6)--(5);
        \draw[Arete](7)--(3);
        \draw[Arete](8)--(7);
        \node(r)at(3.00,2.25){};
        \draw[Arete](r)--(3);
    \end{tikzpicture}
    \end{split}
    \DDendr
    \begin{split}
    \begin{tikzpicture}[xscale=.2,yscale=.2]
        \node[Feuille](0)at(0.00,-3.50){};
        \node[Feuille](2)at(2.00,-5.25){};
        \node[Feuille](4)at(4.00,-5.25){};
        \node[Feuille](6)at(6.00,-1.75){};
        \node[Noeud](1)at(1.00,-1.75){};
        \node[Noeud](3)at(3.00,-3.50){};
        \node[Noeud](5)at(5.00,0.00){};
        \draw[Arete](0)--(1);
        \draw[Arete](1)--(5);
        \draw[Arete](2)--(3);
        \draw[Arete](3)--(1);
        \draw[Arete](4)--(3);
        \draw[Arete](6)--(5);
        \node(r)at(5.00,1.75){};
        \draw[Arete](r)--(5);
    \end{tikzpicture}
    \end{split}
    =
    \begin{split}
    \begin{tikzpicture}[xscale=.18,yscale=.13]
        \node[Feuille](0)at(0.00,-7.50){};
        \node[Feuille](10)at(10.00,-12.50){};
        \node[Feuille](12)at(12.00,-12.50){};
        \node[Feuille](14)at(14.00,-2.50){};
        \node[Feuille](2)at(2.00,-7.50){};
        \node[Feuille](4)at(4.00,-10.00){};
        \node[Feuille](6)at(6.00,-10.00){};
        \node[Feuille](8)at(8.00,-10.00){};
        \node[Noeud](1)at(1.00,-5.00){};
        \node[Noeud](11)at(11.00,-10.00){};
        \node[Noeud](13)at(13.00,0.00){};
        \node[Noeud](3)at(3.00,-2.50){};
        \node[Noeud](5)at(5.00,-7.50){};
        \node[Noeud](7)at(7.00,-5.00){};
        \node[Noeud](9)at(9.00,-7.50){};
        \draw[Arete](0)--(1);
        \draw[Arete](1)--(3);
        \draw[Arete](10)--(11);
        \draw[Arete](11)--(9);
        \draw[Arete](12)--(11);
        \draw[Arete](14)--(13);
        \draw[Arete](2)--(1);
        \draw[Arete](3)--(13);
        \draw[Arete](4)--(5);
        \draw[Arete](5)--(7);
        \draw[Arete](6)--(5);
        \draw[Arete](7)--(3);
        \draw[Arete](8)--(9);
        \draw[Arete](9)--(7);
        \node(r)at(13.00,2.50){};
        \draw[Arete](r)--(13);
        \end{tikzpicture}
    \end{split}
    +
    \begin{split}
    \begin{tikzpicture}[xscale=.18,yscale=.13]
        \node[Feuille](0)at(0.00,-7.50){};
        \node[Feuille](10)at(10.00,-10.00){};
        \node[Feuille](12)at(12.00,-10.00){};
        \node[Feuille](14)at(14.00,-2.50){};
        \node[Feuille](2)at(2.00,-7.50){};
        \node[Feuille](4)at(4.00,-12.50){};
        \node[Feuille](6)at(6.00,-12.50){};
        \node[Feuille](8)at(8.00,-10.00){};
        \node[Noeud](1)at(1.00,-5.00){};
        \node[Noeud](11)at(11.00,-7.50){};
        \node[Noeud](13)at(13.00,0.00){};
        \node[Noeud](3)at(3.00,-2.50){};
        \node[Noeud](5)at(5.00,-10.00){};
        \node[Noeud](7)at(7.00,-7.50){};
        \node[Noeud](9)at(9.00,-5.00){};
        \draw[Arete](0)--(1);
        \draw[Arete](1)--(3);
        \draw[Arete](10)--(11);
        \draw[Arete](11)--(9);
        \draw[Arete](12)--(11);
        \draw[Arete](14)--(13);
        \draw[Arete](2)--(1);
        \draw[Arete](3)--(13);
        \draw[Arete](4)--(5);
        \draw[Arete](5)--(7);
        \draw[Arete](6)--(5);
        \draw[Arete](7)--(9);
        \draw[Arete](8)--(7);
        \draw[Arete](9)--(3);
        \node(r)at(13.00,2.50){};
        \draw[Arete](r)--(13);
        \end{tikzpicture}
    \end{split}
    +
    \begin{split}
    \begin{tikzpicture}[xscale=.18,yscale=.13]
        \node[Feuille](0)at(0.00,-10.00){};
        \node[Feuille](10)at(10.00,-7.50){};
        \node[Feuille](12)at(12.00,-7.50){};
        \node[Feuille](14)at(14.00,-2.50){};
        \node[Feuille](2)at(2.00,-10.00){};
        \node[Feuille](4)at(4.00,-12.50){};
        \node[Feuille](6)at(6.00,-12.50){};
        \node[Feuille](8)at(8.00,-10.00){};
        \node[Noeud](1)at(1.00,-7.50){};
        \node[Noeud](11)at(11.00,-5.00){};
        \node[Noeud](13)at(13.00,0.00){};
        \node[Noeud](3)at(3.00,-5.00){};
        \node[Noeud](5)at(5.00,-10.00){};
        \node[Noeud](7)at(7.00,-7.50){};
        \node[Noeud](9)at(9.00,-2.50){};
        \draw[Arete](0)--(1);
        \draw[Arete](1)--(3);
        \draw[Arete](10)--(11);
        \draw[Arete](11)--(9);
        \draw[Arete](12)--(11);
        \draw[Arete](14)--(13);
        \draw[Arete](2)--(1);
        \draw[Arete](3)--(9);
        \draw[Arete](4)--(5);
        \draw[Arete](5)--(7);
        \draw[Arete](6)--(5);
        \draw[Arete](7)--(3);
        \draw[Arete](8)--(7);
        \draw[Arete](9)--(13);
        \node(r)at(13.00,2.50){};
        \draw[Arete](r)--(13);
        \end{tikzpicture}
    \end{split}\,.
\end{equation}
\medskip

As shown in~\cite{Lod01}, the dendriform operad is the Koszul dual of
the diassociative operad. This can be checked by a simple computation
following what is explained in Section~\ref{subsec:dual_de_Koszul}.
Besides that, since theses two operads are Koszul operads, the alternating
versions of their Hilbert series are the inverses for each other for
series composition.
\medskip

We invite the reader to take a look
at~\cite{LR98,Agu00,Lod02,Foi07,EMP08,EM09,LV12} for a supplementary
review of properties of dendriform algebras and of the dendriform operad.
\medskip

%%%%%%%%%%%%%%%%%%%%%%%%%%%%%%%%%%%%%%%%%%%%%%%%%%%%%%%%%%%%%%%%%%%%%%%%
%%%%%%%%%%%%%%%%%%%%%%%%%%%%%%%%%%%%%%%%%%%%%%%%%%%%%%%%%%%%%%%%%%%%%%%%
%%%%%%%%%%%%%%%%%%%%%%%%%%%%%%%%%%%%%%%%%%%%%%%%%%%%%%%%%%%%%%%%%%%%%%%%
\section{Polydendriform operads} \label{sec:dendr_gamma}
We introduce at this point our generalization on a nonnegative integer
parameter $\gamma$ of the dendriform operad and dendriform algebras. We
first construct this operad, compute its dimensions, and give then two
presentations by generators and relations. This section ends by a
description of free algebras over one generator in the category encoded
by our generalization.
\medskip

%%%%%%%%%%%%%%%%%%%%%%%%%%%%%%%%%%%%%%%%%%%%%%%%%%%%%%%%%%%%%%%%%%%%%%%%
%%%%%%%%%%%%%%%%%%%%%%%%%%%%%%%%%%%%%%%%%%%%%%%%%%%%%%%%%%%%%%%%%%%%%%%%
\subsection{Construction and properties}%
\label{subsec:construcion_dendr_gamma}
Theorem 2.2.6 of~\cite{GirI}, by exhibiting
a presentation of $\Dias_\gamma$, shows that this operad is binary and
quadratic. It then admits a Koszul dual, denoted by $\Dendr_\gamma$ and
called {\em $\gamma$-polydendriform operad}.
\medskip

%%%%%%%%%%%%%%%%%%%%%%%%%%%%%%%%%%%%%%%%%%%%%%%%%%%%%%%%%%%%%%%%%%%%%%%%
\subsubsection{Definition and presentation}%
\label{subsubsec:presentation_dendr_gamma_alternative}
A description of $\Dendr_\gamma$ is provided by the following presentation
by generators and relations.
\medskip

\begin{Theoreme} \label{thm:presentation_dendr_gamma}
    For any integer $\gamma \geq 0$, the operad $\Dendr_\gamma$
    admits the following presentation. It is generated by
    $\GenDendr := \GenDendr(2):=
    \{\GDendrA_a, \DDendrA_a : a \in [\gamma]\}$ and its space of
    relations $\RelDendr$ is generated by
    \begin{subequations}
    \begin{equation} \label{equ:relation_dendr_gamma_1_alternative}
        \GDendrA_a \circ_1 \DDendrA_{a'} - \DDendrA_{a'} \circ_2 \GDendrA_a,
        \qquad a, a' \in [\gamma],
    \end{equation}
    \begin{equation} \label{equ:relation_dendr_gamma_2_alternative}
        \GDendrA_a \circ_1 \GDendrA_b - \GDendrA_a \circ_2 \DDendrA_b,
        \qquad a < b \in [\gamma],
    \end{equation}
    \begin{equation} \label{equ:relation_dendr_gamma_3_alternative}
        \DDendrA_a \circ_1 \GDendrA_b - \DDendrA_a \circ_2 \DDendrA_b,
        \qquad a < b \in [\gamma],
    \end{equation}
    \begin{equation} \label{equ:relation_dendr_gamma_4_alternative}
        \GDendrA_a \circ_1 \GDendrA_b - \GDendrA_a \circ_2 \GDendrA_b,
        \qquad a < b \in [\gamma],
    \end{equation}
    \begin{equation} \label{equ:relation_dendr_gamma_5_alternative}
        \DDendrA_a \circ_1 \DDendrA_b - \DDendrA_a \circ_2 \DDendrA_b,
        \qquad a < b \in [\gamma],
    \end{equation}
    \begin{equation} \label{equ:relation_dendr_gamma_6_alternative}
        \GDendrA_d \circ_1 \GDendrA_d -
        \left(\sum_{c \in [d]} \GDendrA_d \circ_2 \GDendrA_c
                + \GDendrA_d \circ_2 \DDendrA_c\right),
        \qquad d \in [\gamma],
    \end{equation}
    \begin{equation} \label{equ:relation_dendr_gamma_7_alternative}
        \left(\sum_{c \in [d]} \DDendrA_d \circ_1 \DDendrA_c
            + \DDendrA_d \circ_1 \GDendrA_c\right)
            - \DDendrA_d \circ_2 \DDendrA_d,
        \qquad d \in [\gamma].
    \end{equation}
    \end{subequations}
\end{Theoreme}
\begin{proof}
    By Theorem~2.2.6 of~\cite{GirI}, we know
    that $\Dias_\gamma$ is a binary and quadratic operad, and that its
    space of relations $\RelDias$ is the space induced by the equivalence
    relation $\Rel_\gamma$ defined
    by~(2.2.11a)--(2.2.11g)
    in~\cite{GirI}. Now, by a straightforward computation, and by
    identifying $\GDendrA_a$ (resp. $\DDendrA_a$) with $\GDias_a$ (resp.
    $\DDias_a$) for any $a \in [\gamma]$, we obtain that the space
    $\RelDendr$ of the statement of the theorem satisfies
    $\RelDias^\perp = \RelDendr$. Hence, $\Dendr_\gamma$ admits the
    claimed presentation.
\end{proof}
\medskip

Theorem \ref{thm:presentation_dendr_gamma} provides a quite complicated
presentation of $\Dendr_\gamma$. We shall below define a more convenient
basis for the space of relations of $\Dendr_\gamma$.
\medskip

%%%%%%%%%%%%%%%%%%%%%%%%%%%%%%%%%%%%%%%%%%%%%%%%%%%%%%%%%%%%%%%%%%%%%%%%
\subsubsection{Elements and dimensions}
\begin{Proposition} \label{prop:serie_hilbert_dendr_gamma}
    For any integer $\gamma \geq 0$, the Hilbert series
    $\Hca_{\Dendr_\gamma}(t)$ of the operad $\Dendr_\gamma$ satisfies
    \begin{equation} \label{equ:serie_hilbert_dendr_gamma}
        \Hca_{\Dendr_\gamma}(t)
            = t + 2\gamma t \, \Hca_{\Dendr_\gamma}(t)
            + \gamma^2 t \, \Hca_{\Dendr_\gamma}(t)^2.
    \end{equation}
\end{Proposition}
\begin{proof}
    By setting $\bar \Hca_{\Dendr_\gamma}(t) := \Hca_{\Dendr_\gamma}(-t)$,
    from~\eqref{equ:serie_hilbert_dendr_gamma}, we obtain
    \begin{equation}
        t = \frac{-\bar \Hca_{\Dendr_\gamma}(t)}
                 {\left(1 + \gamma \, \bar \Hca_{\Dendr_\gamma}(t)\right)^2}.
    \end{equation}
    Moreover, by setting
    $\bar \Hca_{\Dias_\gamma}(t) := \Hca_{\Dias_\gamma}(-t)$, where
    $\Hca_{\Dias_\gamma}(t)$ is the Hilbert series of $\Dias_\gamma$
    defined by~(2.1.8) in~\cite{GirI}, we
    have
    \begin{equation} \label{equ:serie_hilbert_dendr_gamma_demo}
        \bar \Hca_{\Dias_\gamma}\left(\bar \Hca_{\Dendr_\gamma}(t)\right)
            = \frac{-\bar \Hca_{\Dendr_\gamma}(t)}
                  {\left(1 + \gamma \, \bar \Hca_{\Dendr_\gamma}(t)\right)^2}
            = t,
    \end{equation}
    showing that $\bar \Hca_{\Dias_\gamma}(t)$ and
    $\bar \Hca_{\Dendr_\gamma}(t)$ are the inverses for each other for
    series composition.
    \smallskip

    Now, since  by Theorem~2.3.1 and
    Proposition~2.1.1 of~\cite{GirI},
    $\Dias_\gamma$ is a Koszul operad and its Hilbert series is
    $\Hca_{\Dias_\gamma}(t)$, and since $\Dendr_\gamma$ is by definition
    the Koszul dual of $\Dias_\gamma$, the Hilbert series of these two
    operads satisfy~\eqref{equ:relation_series_hilbert_operade_duale}.
    Therefore, \eqref{equ:serie_hilbert_dendr_gamma_demo} implies that
    the Hilbert series of $\Dendr_\gamma$ is $\Hca_{\Dendr_\gamma}(t)$.
\end{proof}
\medskip

By examining the expression for $\Hca_{\Dendr_\gamma}(t)$ of the
statement of Proposition~\ref{prop:serie_hilbert_dendr_gamma}, we
observe that for any $n \geq 1$, $\Dendr_\gamma(n)$ can be seen as the
vector space $\AlgLibre_{\Dendr_\gamma}(n)$ of binary trees with $n$
internal nodes wherein its $n - 1$ edges connecting two internal nodes
are labeled on $[\gamma]$. We call these trees {\em $\gamma$-edge valued
binary trees}. In our graphical representations of $\gamma$-edge valued
binary trees, any edge label is drawn into a hexagon located half the
edge. For instance,
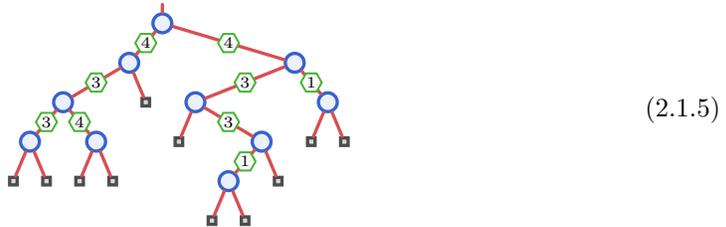
\begin{equation}
    \begin{split}
    \begin{tikzpicture}[xscale=.22,yscale=.15]
        \node[Feuille](0)at(0.00,-14.00){};
        \node[Feuille](10)at(10.00,-10.50){};
        \node[Feuille](12)at(12.00,-17.50){};
        \node[Feuille](14)at(14.00,-17.50){};
        \node[Feuille](16)at(16.00,-14.00){};
        \node[Feuille](18)at(18.00,-10.50){};
        \node[Feuille](2)at(2.00,-14.00){};
        \node[Feuille](20)at(20.00,-10.50){};
        \node[Feuille](4)at(4.00,-14.00){};
        \node[Feuille](6)at(6.00,-14.00){};
        \node[Feuille](8)at(8.00,-7.00){};
        \node[Noeud](1)at(1.00,-10.50){};
        \node[Noeud](11)at(11.00,-7.00){};
        \node[Noeud](13)at(13.00,-14.00){};
        \node[Noeud](15)at(15.00,-10.50){};
        \node[Noeud](17)at(17.00,-3.50){};
        \node[Noeud](19)at(19.00,-7.00){};
        \node[Noeud](3)at(3.00,-7.00){};
        \node[Noeud](5)at(5.00,-10.50){};
        \node[Noeud](7)at(7.00,-3.50){};
        \node[Noeud](9)at(9.00,0.00){};
        \draw[Arete](0)--(1);
        \draw[Arete](1)edge[]node[EtiqArete]{\begin{math}3\end{math}}(3);
        \draw[Arete](10)--(11);
        \draw[Arete](11)edge[]node[EtiqArete]{\begin{math}3\end{math}}(17);
        \draw[Arete](12)--(13);
        \draw[Arete](13)edge[]node[EtiqArete]{\begin{math}1\end{math}}(15);
        \draw[Arete](14)--(13);
        \draw[Arete](15)edge[]node[EtiqArete]{\begin{math}3\end{math}}(11);
        \draw[Arete](16)--(15);
        \draw[Arete](17)edge[]node[EtiqArete]{\begin{math}4\end{math}}(9);
        \draw[Arete](18)--(19);
        \draw[Arete](19)edge[]node[EtiqArete]{\begin{math}1\end{math}}(17);
        \draw[Arete](2)--(1);
        \draw[Arete](20)--(19);
        \draw[Arete](3)edge[]node[EtiqArete]{\begin{math}3\end{math}}(7);
        \draw[Arete](4)--(5);
        \draw[Arete](5)edge[]node[EtiqArete]{\begin{math}4\end{math}}(3);
        \draw[Arete](6)--(5);
        \draw[Arete](7)edge[]node[EtiqArete]{\begin{math}4\end{math}}(9);
        \draw[Arete](8)--(7);
        \node(r)at(9.00,2.50){};
        \draw[Arete](r)--(9);
    \end{tikzpicture}
    \end{split}
\end{equation}
is a $4$-edge valued binary tree and a basis element of $\Dendr_4(10)$.
\medskip

We deduce from Proposition~\ref{prop:serie_hilbert_dendr_gamma} that the
Hilbert series of $\Dendr_\gamma$ satisfies
\begin{equation}
    \Hca_{\Dendr_\gamma}(t) =
    \frac{1 - \sqrt{1 - 4 \gamma t} - 2 \gamma t}{2\gamma^2 t},
\end{equation}
and we also obtain that for all $n \geq 1$,
$\dim \Dendr_\gamma(n) = \gamma^{n - 1} \Cat(n)$.
For instance, the first dimensions of $\Dendr_1$, $\Dendr_2$, $\Dendr_3$,
and $\Dendr_4$ are respectively
\begin{equation}
    1, 2, 5, 14, 42, 132, 429, 1430, 4862, 16796, 58786,
\end{equation}
\begin{equation}
    1, 4, 20, 112, 672, 4224, 27456, 183040, 1244672, 8599552, 60196864,
\end{equation}
\begin{equation}
    1, 6, 45, 378, 3402, 32076, 312741, 3127410, 31899582, 330595668, 3471254514,
\end{equation}
\begin{equation}
    1, 8, 80, 896, 10752, 135168, 1757184, 23429120, 318636032,
    4402970624, 61641588736.
\end{equation}
The first one is Sequence~\Sloane{A000108}, the second one is
Sequence~\Sloane{A003645}, and the third one is Sequence~\Sloane{A101600}
of~\cite{Slo}. Last sequence is not listed in~\cite{Slo} at this time.
\medskip

%%%%%%%%%%%%%%%%%%%%%%%%%%%%%%%%%%%%%%%%%%%%%%%%%%%%%%%%%%%%%%%%%%%%%%%%
\subsubsection{Associative operations}
In the same manner as in the dendriform operad the sum of its two
operations produces an associative operation, in the $\gamma$-dendriform
operad there is a way to build associative operations, as shows next
statement.
\medskip

\begin{Proposition} \label{prop:operateur_associatif_dendr_gamma_autre}
    For any integers $\gamma \geq 0$ and $b \in [\gamma]$, the element
    \begin{equation}
        \OpAsDendrA_b :=
        \pi\left(\sum_{a \in [b]} \GDendrA_a + \DDendrA_a\right)
    \end{equation}
    of $\Dendr_\gamma$, where
    $\pi : \OpLibre\left(\GenDendr\right) \to \Dendr_\gamma$ is the
    canonical surjection map, is associative.
\end{Proposition}
\begin{proof}
    By setting
    \begin{equation}
        x := \sum_{a \in [b]} \GDendrA_a + \DDendrA_a,
    \end{equation}
    we have
    \begin{multline} \label{equ:operateur_associatif_dendr_gamma_autre_demo}
        x \circ_1 x - x \circ_2 x =
        \GDendrA_a \circ_1 \GDendrA_{a'} +
        \GDendrA_a \circ_1 \DDendrA_{a'}
            + \DDendrA_a \circ_1 \GDendrA_{a'} +
        \DDendrA_a \circ_1 \DDendrA_{a'} \\
            -
            \GDendrA_a \circ_2 \GDendrA_{a'} -
        \GDendrA_a \circ_2 \DDendrA_{a'}
            -
        \DDendrA_a \circ_2 \GDendrA_{a'} -
        \DDendrA_a \circ_2 \DDendrA_{a'}.
    \end{multline}
    We the observe that~\eqref{equ:operateur_associatif_dendr_gamma_autre_demo}
    is the sum of
    elements~\eqref{equ:relation_dendr_gamma_1_alternative}---%
    \eqref{equ:relation_dendr_gamma_7_alternative}
    which generate, by Theorem~\ref{thm:presentation_dendr_gamma}, the
    space of relations of $\Dendr_\gamma$. Therefore, we have
    $\pi(x \circ_1 x - x \circ_2 x) = 0$, implying
    $\OpAsDendrA_b \circ_1 \OpAsDendrA_b -
    \OpAsDendrA_b \circ_2 \OpAsDendrA_b = 0$ and
    the associativity of $\OpAsDendrA_b$.
\end{proof}
\medskip

%%%%%%%%%%%%%%%%%%%%%%%%%%%%%%%%%%%%%%%%%%%%%%%%%%%%%%%%%%%%%%%%%%%%%%%%
\subsubsection{Alternative presentation}
For any integer $\gamma \geq 0$, let $\GDendr_b$ and $\DDendr_b$,
$b \in [\gamma]$, the elements of $\OpLibre\left(\GenDendr\right)$
defined by
\begin{subequations}
\begin{equation} \label{equ:definition_operateur_dendr_gauche}
    \GDendr_b := \sum_{a \in [b]} \GDendrA_a,
\end{equation}
and
\begin{equation} \label{equ:definition_operateur_dendr_droite}
    \DDendr_b := \sum_{a \in [b]} \DDendrA_a.
\end{equation}
\end{subequations}
Then, since for all $b \in [\gamma]$ we have
\begin{subequations}
\begin{equation}
    \GDendrA_b =
    \begin{cases}
        \GDendr_1 & \mbox{if } b = 1, \\
        \GDendr_b - \GDendr_{b - 1} & \mbox{otherwise},
    \end{cases}
\end{equation}
and
\begin{equation}
    \DDendrA_b =
    \begin{cases}
        \DDendr_1 & \mbox{if } b = 1, \\
        \DDendr_b - \DDendr_{b - 1} & \mbox{otherwise},
    \end{cases}
\end{equation}
\end{subequations}
by triangularity, the family
$\GenDendr' := \{\GDendr_b, \DDendr_b : b \in [\gamma]\}$ forms a
basis of $\OpLibre\left(\GenDendr\right)(2)$ and then, generates
$\OpLibre\left(\GenDendr\right)$ as an operad. This change of basis
from $\OpLibre\left(\GenDendr\right)$ to $\OpLibre(\GenDendr')$
is similar to the change of basis from $\OpLibre(\GenDias')$
to $\OpLibre\left(\GenDias\right)$ introduced in
Section~2.3.6 of~\cite{GirI}.
Let us now express a presentation of $\Dendr_\gamma$ through the
family~$\GenDendr'$.
\medskip

\begin{Theoreme} \label{thm:autre_presentation_dendr_gamma}
    For any integer $\gamma \geq 0$, the operad $\Dendr_\gamma$
    admits the following presentation. It is generated by $\GenDendr'$
    and its space of relations $\RelDendr'$ is generated by
    \begin{subequations}
    \begin{equation} \label{equ:relation_dendr_gamma_1}
        \GDendr_a \circ_1 \DDendr_{a'} - \DDendr_{a'} \circ_2 \GDendr_a,
        \qquad a, a' \in [\gamma],
    \end{equation}
    \begin{equation} \label{equ:relation_dendr_gamma_2}
        \GDendr_a \circ_1 \GDendr_b
            - \GDendr_a \circ_2 \DDendr_b
            - \GDendr_a \circ_2 \GDendr_a,
        \qquad a < b \in [\gamma],
    \end{equation}
    \begin{equation} \label{equ:relation_dendr_gamma_3}
            \DDendr_a \circ_1 \DDendr_a
            + \DDendr_a \circ_1 \GDendr_b
            - \DDendr_a \circ_2 \DDendr_b,
            \qquad a < b \in [\gamma],
    \end{equation}
    \begin{equation} \label{equ:relation_dendr_gamma_4}
        \GDendr_b \circ_1 \GDendr_a
            - \GDendr_a \circ_2 \GDendr_b
            - \GDendr_a \circ_2 \DDendr_a,
        \qquad a < b \in [\gamma],
    \end{equation}
    \begin{equation} \label{equ:relation_dendr_gamma_5}
        \DDendr_a \circ_1 \GDendr_a
            + \DDendr_a \circ_1 \DDendr_b
            - \DDendr_b \circ_2 \DDendr_a,
        \qquad a < b \in [\gamma],
    \end{equation}
    \begin{equation} \label{equ:relation_dendr_gamma_6}
        \GDendr_a \circ_1 \GDendr_a
            - \GDendr_a \circ_2 \DDendr_a
            - \GDendr_a \circ_2 \GDendr_a,
        \qquad a \in [\gamma],
    \end{equation}
    \begin{equation} \label{equ:relation_dendr_gamma_7}
        \DDendr_a \circ_1 \DDendr_a
            + \DDendr_a \circ_1 \GDendr_a
            - \DDendr_a \circ_2 \DDendr_a,
        \qquad a \in [\gamma].
    \end{equation}
    \end{subequations}
\end{Theoreme}
\begin{proof}
    Let us show that $\RelDendr'$ is equal to the space of relations
    $\RelDendr$ of $\Dendr_\gamma$ defined in the statement of
    Theorem~\ref{thm:presentation_dendr_gamma}. By this last theorem,
    for any $x \in \OpLibre\left(\GenDendr\right)(3)$, $x$ is in
    $\RelDendr$ if and only if $\pi(x) = 0$ where
    $\pi : \OpLibre\left(\GenDendr\right) \to \Dendr_\gamma$ is the
    canonical surjection map. By straightforward computations, by
    expanding any element $x$ of~\eqref{equ:relation_dendr_gamma_1}---%
    \eqref{equ:relation_dendr_gamma_7} over the elements
    $\GDendrA_a$, $\DDendrA_a$, $a \in [\gamma]$, by
    using~\eqref{equ:definition_operateur_dendr_gauche}
    and~\eqref{equ:definition_operateur_dendr_droite} we obtain that
    $x$ can be expressed as a sum of elements of $\RelDendr$. This
    implies that $\pi(x) = 0$ and hence that $\RelDendr'$ is a subspace
    of $\RelDendr$.
    \smallskip

    Now, one can observe that
    elements~\eqref{equ:relation_dendr_gamma_1}---%
    \eqref{equ:relation_dendr_gamma_6} are linearly independent.
    Then, $\RelDendr'$ has dimension $3\gamma^2$ which is also, by
    Theorem~\ref{thm:presentation_dendr_gamma}, the dimension of
    $\RelDendr$. The statement of the theorem follows.
\end{proof}
\medskip

The presentation of $\Dendr_\gamma$ provided by
Theorem~\ref{thm:autre_presentation_dendr_gamma} is easier to
handle than the one provided by Theorem \ref{thm:presentation_dendr_gamma}.
The main reason is that Relations~\eqref{equ:relation_dendr_gamma_6_alternative}
and~\eqref{equ:relation_dendr_gamma_7_alternative} of the first
presentation involve a nonconstant number of terms, while all relations
of this second presentation always involve only two or three terms. As a
very remarkable fact, it is worthwhile to note that the presentation of
$\Dendr_\gamma$ provided by Theorem~\ref{thm:autre_presentation_dendr_gamma}
can be directly obtained by considering the Koszul dual of $\Dias_\gamma$
over the $\Ksf$-basis (see Sections~2.3.5
and~2.3.6 of~\cite{GirI}).
Therefore, an alternative way to establish this presentation consists in
computing the Koszul dual of $\Dias_\gamma$ seen through the presentation
having $\RelDendr'$ as space of relations, which is made of the relations
of $\Dias_\gamma$ expressed over the $\Ksf$-basis
(see Proposition~2.3.8
of~\cite{GirI}).
\medskip

From now on, $\Min$ denotes the operation $\min$ on integers. Using this
notation, the space of relations $\RelDendr'$ of $\Dendr_\gamma$
exhibited by Theorem~\ref{thm:autre_presentation_dendr_gamma}
can be rephrased in a more compact way as the space generated by
\begin{subequations}
\begin{equation} \label{equ:relation_dendr_gamma_1_concise}
    \GDendr_a \circ_1 \DDendr_{a'} - \DDendr_{a'} \circ_2 \GDendr_a,
    \qquad a, a' \in [\gamma],
\end{equation}
\begin{equation} \label{equ:relation_dendr_gamma_2_concise}
    \GDendr_a \circ_1 \GDendr_{a'}
        - \GDendr_{a \Min a'} \circ_2 \GDendr_a
        - \GDendr_{a \Min a'} \circ_2 \DDendr_{a'},
    \qquad a, a' \in [\gamma],
\end{equation}
\begin{equation} \label{equ:relation_dendr_gamma_3_concise}
    \DDendr_{a \Min a'} \circ_1 \GDendr_{a'}
        + \DDendr_{a \Min a'} \circ_1 \DDendr_a
        - \DDendr_a \circ_2 \DDendr_{a'},
    \qquad a, a' \in [\gamma].
\end{equation}
\end{subequations}
\medskip

Over the family $\GenDendr'$, one can build associative operations in
$\Dendr_\gamma$ in the following way.
\medskip

\begin{Proposition} \label{prop:operateur_associatif_dendr_gamma}
    For any integers $\gamma \geq 0$ and $b \in [\gamma]$, the element
    \begin{equation}
        \OpAsDendr_b := \pi(\GDendr_b + \DDendr_b)
    \end{equation}
    of $\Dendr_\gamma$, where
    $\pi : \OpLibre(\GenDendr') \to \Dendr_\gamma$ is the
    canonical surjection map, is associative.
\end{Proposition}
\begin{proof}
    By definition of the $\GDendr_b$ and $\DDendr_b$, $b \in [\gamma]$,
    we have
    \begin{equation}
        \GDendr_b + \DDendr_b =
        \sum_{a \in [b]} \GDendrA_a + \DDendrA_a.
    \end{equation}
    We hence observe that $\OpAsDendr_b = \OpAsDendrA_b$, where
    $\OpAsDendrA_b$ is the element of $\Dendr_\gamma$ defined in the
    statement of Proposition~\ref{prop:operateur_associatif_dendr_gamma_autre}.
    Hence, by this latter proposition, $\OpAsDendr_b$ is associative.
\end{proof}
\medskip

\begin{Proposition}
\label{prop:description_operateurs_associatifs_dendr_gamma}
    For any integer $\gamma \geq 0$, any associative element of
    $\Dendr_\gamma$ is proportional to $\OpAsDendr_b$ for a
    $b \in [\gamma]$.
\end{Proposition}
\begin{proof}
    Let $\pi : \OpLibre(\GenDendr') \to \Dendr_\gamma$ be the canonical
    surjection map. Consider the element
    \begin{equation}
        x :=
        \sum_{a \in [\gamma]} \alpha_a \GDendr_a + \beta_a \DDendr_a
    \end{equation}
    of $\OpLibre(\GenDendr')$, where $\alpha_a, \beta_a \in \K$ for all
    $a \in [\gamma]$, such that $\pi(x)$ is associative in $\Dendr_\gamma$.
    Since we have $\pi(r) = 0$ for all elements $r$ of $\RelDendr'$
    (see~\eqref{equ:relation_dendr_gamma_1_concise},
    \eqref{equ:relation_dendr_gamma_2_concise},
    and~\eqref{equ:relation_dendr_gamma_3_concise}),
    the fact that $\pi(x \circ_1 x - x \circ_2 x) = 0$ implies the
    constraints
    \begin{equation}\begin{split}
        \alpha_a \, \beta_{a'} = \beta_{a'} \, \alpha_a,
        & \qquad a, a' \in [\gamma], \\
        \alpha_a \, \alpha_{a'} = \alpha_{a \Min a'} \, \alpha_a =
        \alpha_{a \Min a'} \, \beta_{a'}, & \qquad a, a' \in [\gamma], \\
        \beta_{a \Min a'} \, \alpha_{a'} = \beta_{a \Min a'} \, \beta_a =
        \beta_a \, \beta_{a'}, & \qquad a, a' \in [\gamma],
    \end{split}\end{equation}
    on the coefficients intervening in $x$. Moreover, since the syntax
    trees $\DDendr_b \circ_1 \DDendr_a$, $\DDendr_b \circ_1 \GDendr_a$,
    $\GDendr_b \circ_2 \GDendr_a$, and $\GDendr_b \circ_2 \DDendr_a$ do
    not appear in $\RelDendr'$ for all $a < b \in [\gamma]$ , we have the
    further constraints
    \begin{equation}\begin{split}
        \beta_b \, \beta_a & = 0, \qquad a < b \in [\gamma], \\
        \beta_b \, \alpha_a & = 0, \qquad a < b \in [\gamma], \\
        \alpha_b \, \alpha_a & = 0, \qquad a < b \in [\gamma], \\
        \alpha_b \, \beta_a & = 0, \qquad a < b \in [\gamma].
    \end{split}\end{equation}
    These relations imply that there are at most one $c \in [\gamma]$ and
    one $d \in [\gamma]$ such that $\alpha_c \ne 0$ and $\beta_d \ne 0$.
    In this case, these relations imply also that $c = d$, and
    $\alpha_c = \beta_c$. Therefore, $x$ is of the form
    $x = \alpha_a \GDendr_a + \alpha_a \DDendr_a$ for an $a \in [\gamma]$,
    whence the statement of the proposition.
\end{proof}
\medskip

%%%%%%%%%%%%%%%%%%%%%%%%%%%%%%%%%%%%%%%%%%%%%%%%%%%%%%%%%%%%%%%%%%%%%%%%
%%%%%%%%%%%%%%%%%%%%%%%%%%%%%%%%%%%%%%%%%%%%%%%%%%%%%%%%%%%%%%%%%%%%%%%%
\subsection{Category of polydendriform algebras and free objects}
The aim of this section is to describe the category of
$\Dendr_\gamma$-algebras and more particularly the free
$\Dendr_\gamma$-algebra over one generator.
\medskip

%%%%%%%%%%%%%%%%%%%%%%%%%%%%%%%%%%%%%%%%%%%%%%%%%%%%%%%%%%%%%%%%%%%%%%%%
\subsubsection{Polydendriform algebras}
We call {\em $\gamma$-polydendriform algebra} any
$\Dendr_\gamma$-algebra. From the presentation of $\Dendr_\gamma$
provided by Theorem~\ref{thm:presentation_dendr_gamma}, any
$\gamma$-polydendriform algebra is a vector space endowed with linear
operations $\GDendrA_a, \DDendrA_a$, $a \in [\gamma]$, satisfying the
relations encoded by~\eqref{equ:relation_dendr_gamma_1_alternative}---%
\eqref{equ:relation_dendr_gamma_7_alternative}. By considering the
presentation of $\Dendr_\gamma$ exhibited by
Theorem~\ref{thm:autre_presentation_dendr_gamma}, any
$\gamma$-polydendriform algebra is a vector space endowed with linear
operations $\GDendr_a, \DDendr_a$, $a \in [\gamma]$, satisfying the
relations encoded by~\eqref{equ:relation_dendr_gamma_1_concise}---%
\eqref{equ:relation_dendr_gamma_3_concise}.
\medskip

%%%%%%%%%%%%%%%%%%%%%%%%%%%%%%%%%%%%%%%%%%%%%%%%%%%%%%%%%%%%%%%%%%%%%%%%
\subsubsection{Two ways to split associativity}
Like dendriform algebras, which offer a way to split an associative
operation into two parts, $\gamma$-polydendriform algebras propose
two ways to split associativity depending on its chosen presentation.
\medskip

On the one hand, in a $\gamma$-polydendriform algebra $\Dca$ over the
operations $\GDendrA_a$, $\DDendrA_a$, $a \in [\gamma]$,
by Proposition~\ref{prop:operateur_associatif_dendr_gamma_autre}, an
associative operation $\OpAsDendrA$ is split into the $2\gamma$ operations
$\GDendrA_a$, $\DDendrA_a$, $a \in [\gamma]$, so that for all $x, y \in \Dca$,
\begin{equation}
    x \OpAsDendrA y =
    \sum_{a \in [\gamma]}  x \GDendrA_a y + x \DDendrA_a y,
\end{equation}
and all partial sums operations $\OpAsDendrA_b$, $b \in [\gamma]$,
satisfying
\begin{equation}
    x \OpAsDendrA_b y = \sum_{a \in [b]} x \GDendrA_a y + x \DDendrA_a x,
\end{equation}
also are associative.
\medskip

On the other hand, in a $\gamma$-polydendriform algebra over the operations
$\GDendr_a$, $\DDendr_a$, $a \in [\gamma]$, by
Proposition~\ref{prop:operateur_associatif_dendr_gamma}, several
associative operations $\OpAsDendr_a$, $a \in [\gamma]$, are each split
into two operations $\GDendr_a$, $\DDendr_a$, $a \in [\gamma]$, so
that for all $x, y \in \Dca$,
\begin{equation}
    x \OpAsDendr_a y = x \GDendr_a y + x \DDendr_a y.
\end{equation}
\medskip

Therefore, we can observe that $\gamma$-polydendriform algebras over the
operations $\GDendrA_a$, $\DDendrA_a$, $a \in [\gamma]$, are adapted to
study associative algebras (by splitting its single product in the way
we have described above) while $\gamma$-polydendriform algebras over the
operations $\GDendr_a$, $\DDendr_a$, $a \in [\gamma]$, are adapted to
study vectors spaces endowed with several associative products (by
splitting each one in the way we have described above). Algebras with
several associative products will be studied in Section~\ref{sec:as_gamma}.
\medskip

%%%%%%%%%%%%%%%%%%%%%%%%%%%%%%%%%%%%%%%%%%%%%%%%%%%%%%%%%%%%%%%%%%%%%%%%
\subsubsection{Free polydendriform algebras}
From now, in order to simplify and make uniform next definitions, we
consider that in any $\gamma$-edge valued binary tree $\Tfr$, all edges
connecting internal nodes of $\Tfr$ with leaves are labeled by
$\infty$. By convention, for all $a \in [\gamma]$, we have
$a \Min \infty = a = \infty \Min a$.
\medskip

Let us endow the vector space $\AlgLibre_{\Dendr_\gamma}$
of $\gamma$-edge valued binary trees
with linear operations
\begin{equation}
    \GDendr_a, \DDendr_a :
    \AlgLibre_{\Dendr_\gamma} \otimes \AlgLibre_{\Dendr_\gamma}
    \to \AlgLibre_{\Dendr_\gamma},
    \qquad
    a \in [\gamma],
\end{equation}
recursively defined, for any $\gamma$-edge valued binary tree $\Sfr$
and any $\gamma$-edge valued binary trees or leaves $\Tfr_1$ and $\Tfr_2$
by
\begin{equation}
    \Sfr \GDendr_a \Feuille
    := \Sfr =:
    \Feuille \DDendr_a \Sfr,
\end{equation}
\begin{equation}
    \Feuille \GDendr_a \Sfr := 0 =: \Sfr \DDendr_a \Feuille,
\end{equation}
\begin{equation}
    \ArbreBinValue{x}{y}{\Tfr_1}{\Tfr_2}
    \GDendr_a \Sfr :=
    \ArbreBinValue{x}{z}{\Tfr_1}{\Tfr_2 \GDendr_a \Sfr}
    +
    \ArbreBinValue{x}{z}{\Tfr_1}{\Tfr_2 \DDendr_y \Sfr}\,,
    \qquad
    z := a \Min y,
\end{equation}
\begin{equation}
    \begin{split}\Sfr \DDendr_a\end{split}
    \ArbreBinValue{x}{y}{\Tfr_1}{\Tfr_2}
    :=
    \ArbreBinValue{z}{y}{\Sfr \DDendr_a \Tfr_1}{\Tfr_2}
    +
    \ArbreBinValue{z}{y}{\Sfr \GDendr_x \Tfr_1}{\Tfr_2}\,,
    \qquad
    z := a \Min x.
\end{equation}
Note that neither $\Feuille \GDendr_a \Feuille$ nor
$\Feuille \DDendr_a \Feuille$ are defined.
\medskip

For example, we have
\begin{multline}
    \begin{split}
    \begin{tikzpicture}[xscale=.25,yscale=.2]
        \node[Feuille](0)at(0.00,-4.50){};
        \node[Feuille](2)at(2.00,-4.50){};
        \node[Feuille](4)at(4.00,-4.50){};
        \node[Feuille](6)at(6.00,-6.75){};
        \node[Feuille](8)at(8.00,-6.75){};
        \node[Noeud](1)at(1.00,-2.25){};
        \node[Noeud](3)at(3.00,0.00){};
        \node[Noeud](5)at(5.00,-2.25){};
        \node[Noeud](7)at(7.00,-4.50){};
        \draw[Arete](0)--(1);
        \draw[Arete](1)edge[]node[EtiqArete]{\begin{math}1\end{math}}(3);
        \draw[Arete](2)--(1);
        \draw[Arete](4)--(5);
        \draw[Arete](5)edge[]node[EtiqArete]{\begin{math}3\end{math}}(3);
        \draw[Arete](6)--(7);
        \draw[Arete](7)edge[]node[EtiqArete]{\begin{math}1\end{math}}(5);
        \draw[Arete](8)--(7);
        \node(r)at(3.00,1.7){};
        \draw[Arete](r)--(3);
    \end{tikzpicture}
    \end{split}
    \GDendr_2
    \begin{split}
    \begin{tikzpicture}[xscale=.25,yscale=.2]
        \node[Feuille](0)at(0.00,-4.67){};
        \node[Feuille](2)at(2.00,-4.67){};
        \node[Feuille](4)at(4.00,-4.67){};
        \node[Feuille](6)at(6.00,-4.67){};
        \node[Noeud](1)at(1.00,-2.33){};
        \node[Noeud](3)at(3.00,0.00){};
        \node[Noeud](5)at(5.00,-2.33){};
        \draw[Arete](0)--(1);
        \draw[Arete](1)edge[]node[EtiqArete]{\begin{math}1\end{math}}(3);
        \draw[Arete](2)--(1);
        \draw[Arete](4)--(5);
        \draw[Arete](5)edge[]node[EtiqArete]{\begin{math}2\end{math}}(3);
        \draw[Arete](6)--(5);
        \node(r)at(3.00,1.7){};
        \draw[Arete](r)--(3);
    \end{tikzpicture}
    \end{split}
    =
    \begin{split}
    \begin{tikzpicture}[xscale=.22,yscale=.15]
        \node[Feuille](0)at(0.00,-5.00){};
        \node[Feuille](10)at(10.00,-12.50){};
        \node[Feuille](12)at(12.00,-12.50){};
        \node[Feuille](14)at(14.00,-12.50){};
        \node[Feuille](2)at(2.00,-5.00){};
        \node[Feuille](4)at(4.00,-5.00){};
        \node[Feuille](6)at(6.00,-7.50){};
        \node[Feuille](8)at(8.00,-12.50){};
        \node[Noeud](1)at(1.00,-2.50){};
        \node[Noeud](11)at(11.00,-7.50){};
        \node[Noeud](13)at(13.00,-10.00){};
        \node[Noeud](3)at(3.00,0.00){};
        \node[Noeud](5)at(5.00,-2.50){};
        \node[Noeud](7)at(7.00,-5.00){};
        \node[Noeud](9)at(9.00,-10.00){};
        \draw[Arete](0)--(1);
        \draw[Arete](1)edge[]node[EtiqArete]{\begin{math}1\end{math}}(3);
        \draw[Arete](10)--(9);
        \draw[Arete](11)edge[]node[EtiqArete]{\begin{math}2\end{math}}(7);
        \draw[Arete](12)--(13);
        \draw[Arete](13)edge[]node[EtiqArete]{\begin{math}2\end{math}}(11);
        \draw[Arete](14)--(13);
        \draw[Arete](2)--(1);
        \draw[Arete](4)--(5);
        \draw[Arete](5)edge[]node[EtiqArete]{\begin{math}2\end{math}}(3);
        \draw[Arete](6)--(7);
        \draw[Arete](7)edge[]node[EtiqArete]{\begin{math}1\end{math}}(5);
        \draw[Arete](8)--(9);
        \draw[Arete](9)edge[]node[EtiqArete]{\begin{math}1\end{math}}(11);
        \node(r)at(3.00,2.50){};
        \draw[Arete](r)--(3);
    \end{tikzpicture}
    \end{split}
    +
    \begin{split}
    \begin{tikzpicture}[xscale=.22,yscale=.15]
        \node[Feuille](0)at(0.00,-5.00){};
        \node[Feuille](10)at(10.00,-12.50){};
        \node[Feuille](12)at(12.00,-10.00){};
        \node[Feuille](14)at(14.00,-10.00){};
        \node[Feuille](2)at(2.00,-5.00){};
        \node[Feuille](4)at(4.00,-5.00){};
        \node[Feuille](6)at(6.00,-10.00){};
        \node[Feuille](8)at(8.00,-12.50){};
        \node[Noeud](1)at(1.00,-2.50){};
        \node[Noeud](11)at(11.00,-5.00){};
        \node[Noeud](13)at(13.00,-7.50){};
        \node[Noeud](3)at(3.00,0.00){};
        \node[Noeud](5)at(5.00,-2.50){};
        \node[Noeud](7)at(7.00,-7.50){};
        \node[Noeud](9)at(9.00,-10.00){};
        \draw[Arete](0)--(1);
        \draw[Arete](1)edge[]node[EtiqArete]{\begin{math}1\end{math}}(3);
        \draw[Arete](10)--(9);
        \draw[Arete](11)edge[]node[EtiqArete]{\begin{math}1\end{math}}(5);
        \draw[Arete](12)--(13);
        \draw[Arete](13)edge[]node[EtiqArete]{\begin{math}2\end{math}}(11);
        \draw[Arete](14)--(13);
        \draw[Arete](2)--(1);
        \draw[Arete](4)--(5);
        \draw[Arete](5)edge[]node[EtiqArete]{\begin{math}2\end{math}}(3);
        \draw[Arete](6)--(7);
        \draw[Arete](7)edge[]node[EtiqArete]{\begin{math}1\end{math}}(11);
        \draw[Arete](8)--(9);
        \draw[Arete](9)edge[]node[EtiqArete]{\begin{math}1\end{math}}(7);
        \node(r)at(3.00,2.50){};
        \draw[Arete](r)--(3);
    \end{tikzpicture}
    \end{split} \\
    +
    \begin{split}
    \begin{tikzpicture}[xscale=.22,yscale=.15]
        \node[Feuille](0)at(0.00,-5.00){};
        \node[Feuille](10)at(10.00,-10.00){};
        \node[Feuille](12)at(12.00,-10.00){};
        \node[Feuille](14)at(14.00,-10.00){};
        \node[Feuille](2)at(2.00,-5.00){};
        \node[Feuille](4)at(4.00,-5.00){};
        \node[Feuille](6)at(6.00,-12.50){};
        \node[Feuille](8)at(8.00,-12.50){};
        \node[Noeud](1)at(1.00,-2.50){};
        \node[Noeud](11)at(11.00,-5.00){};
        \node[Noeud](13)at(13.00,-7.50){};
        \node[Noeud](3)at(3.00,0.00){};
        \node[Noeud](5)at(5.00,-2.50){};
        \node[Noeud](7)at(7.00,-10.00){};
        \node[Noeud](9)at(9.00,-7.50){};
        \draw[Arete](0)--(1);
        \draw[Arete](1)edge[]node[EtiqArete]{\begin{math}1\end{math}}(3);
        \draw[Arete](10)--(9);
        \draw[Arete](11)edge[]node[EtiqArete]{\begin{math}1\end{math}}(5);
        \draw[Arete](12)--(13);
        \draw[Arete](13)edge[]node[EtiqArete]{\begin{math}2\end{math}}(11);
        \draw[Arete](14)--(13);
        \draw[Arete](2)--(1);
        \draw[Arete](4)--(5);
        \draw[Arete](5)edge[]node[EtiqArete]{\begin{math}2\end{math}}(3);
        \draw[Arete](6)--(7);
        \draw[Arete](7)edge[]node[EtiqArete]{\begin{math}1\end{math}}(9);
        \draw[Arete](8)--(7);
        \draw[Arete](9)edge[]node[EtiqArete]{\begin{math}1\end{math}}(11);
        \node(r)at(3.00,2.50){};
        \draw[Arete](r)--(3);
    \end{tikzpicture}
    \end{split}
    +
    \begin{split}
    \begin{tikzpicture}[xscale=.22,yscale=.15]
        \node[Feuille](0)at(0.00,-5.00){};
        \node[Feuille](10)at(10.00,-12.50){};
        \node[Feuille](12)at(12.00,-7.50){};
        \node[Feuille](14)at(14.00,-7.50){};
        \node[Feuille](2)at(2.00,-5.00){};
        \node[Feuille](4)at(4.00,-7.50){};
        \node[Feuille](6)at(6.00,-10.00){};
        \node[Feuille](8)at(8.00,-12.50){};
        \node[Noeud](1)at(1.00,-2.50){};
        \node[Noeud](11)at(11.00,-2.50){};
        \node[Noeud](13)at(13.00,-5.00){};
        \node[Noeud](3)at(3.00,0.00){};
        \node[Noeud](5)at(5.00,-5.00){};
        \node[Noeud](7)at(7.00,-7.50){};
        \node[Noeud](9)at(9.00,-10.00){};
        \draw[Arete](0)--(1);
        \draw[Arete](1)edge[]node[EtiqArete]{\begin{math}1\end{math}}(3);
        \draw[Arete](10)--(9);
        \draw[Arete](11)edge[]node[EtiqArete]{\begin{math}2\end{math}}(3);
        \draw[Arete](12)--(13);
        \draw[Arete](13)edge[]node[EtiqArete]{\begin{math}2\end{math}}(11);
        \draw[Arete](14)--(13);
        \draw[Arete](2)--(1);
        \draw[Arete](4)--(5);
        \draw[Arete](5)edge[]node[EtiqArete]{\begin{math}1\end{math}}(11);
        \draw[Arete](6)--(7);
        \draw[Arete](7)edge[]node[EtiqArete]{\begin{math}1\end{math}}(5);
        \draw[Arete](8)--(9);
        \draw[Arete](9)edge[]node[EtiqArete]{\begin{math}1\end{math}}(7);
        \node(r)at(3.00,2.50){};
        \draw[Arete](r)--(3);
    \end{tikzpicture}
    \end{split}
    +
    \begin{split}
    \begin{tikzpicture}[xscale=.22,yscale=.15]
        \node[Feuille](0)at(0.00,-5.00){};
        \node[Feuille](10)at(10.00,-10.00){};
        \node[Feuille](12)at(12.00,-7.50){};
        \node[Feuille](14)at(14.00,-7.50){};
        \node[Feuille](2)at(2.00,-5.00){};
        \node[Feuille](4)at(4.00,-7.50){};
        \node[Feuille](6)at(6.00,-12.50){};
        \node[Feuille](8)at(8.00,-12.50){};
        \node[Noeud](1)at(1.00,-2.50){};
        \node[Noeud](11)at(11.00,-2.50){};
        \node[Noeud](13)at(13.00,-5.00){};
        \node[Noeud](3)at(3.00,0.00){};
        \node[Noeud](5)at(5.00,-5.00){};
        \node[Noeud](7)at(7.00,-10.00){};
        \node[Noeud](9)at(9.00,-7.50){};
        \draw[Arete](0)--(1);
        \draw[Arete](1)edge[]node[EtiqArete]{\begin{math}1\end{math}}(3);
        \draw[Arete](10)--(9);
        \draw[Arete](11)edge[]node[EtiqArete]{\begin{math}2\end{math}}(3);
        \draw[Arete](12)--(13);
        \draw[Arete](13)edge[]node[EtiqArete]{\begin{math}2\end{math}}(11);
        \draw[Arete](14)--(13);
        \draw[Arete](2)--(1);
        \draw[Arete](4)--(5);
        \draw[Arete](5)edge[]node[EtiqArete]{\begin{math}1\end{math}}(11);
        \draw[Arete](6)--(7);
        \draw[Arete](7)edge[]node[EtiqArete]{\begin{math}1\end{math}}(9);
        \draw[Arete](8)--(7);
        \draw[Arete](9)edge[]node[EtiqArete]{\begin{math}1\end{math}}(5);
        \node(r)at(3.00,2.50){};
        \draw[Arete](r)--(3);
    \end{tikzpicture}
    \end{split}
    +
    \begin{split}
    \begin{tikzpicture}[xscale=.22,yscale=.15]
        \node[Feuille](0)at(0.00,-5.00){};
        \node[Feuille](10)at(10.00,-7.50){};
        \node[Feuille](12)at(12.00,-7.50){};
        \node[Feuille](14)at(14.00,-7.50){};
        \node[Feuille](2)at(2.00,-5.00){};
        \node[Feuille](4)at(4.00,-10.00){};
        \node[Feuille](6)at(6.00,-12.50){};
        \node[Feuille](8)at(8.00,-12.50){};
        \node[Noeud](1)at(1.00,-2.50){};
        \node[Noeud](11)at(11.00,-2.50){};
        \node[Noeud](13)at(13.00,-5.00){};
        \node[Noeud](3)at(3.00,0.00){};
        \node[Noeud](5)at(5.00,-7.50){};
        \node[Noeud](7)at(7.00,-10.00){};
        \node[Noeud](9)at(9.00,-5.00){};
        \draw[Arete](0)--(1);
        \draw[Arete](1)edge[]node[EtiqArete]{\begin{math}1\end{math}}(3);
        \draw[Arete](10)--(9);
        \draw[Arete](11)edge[]node[EtiqArete]{\begin{math}2\end{math}}(3);
        \draw[Arete](12)--(13);
        \draw[Arete](13)edge[]node[EtiqArete]{\begin{math}2\end{math}}(11);
        \draw[Arete](14)--(13);
        \draw[Arete](2)--(1);
        \draw[Arete](4)--(5);
        \draw[Arete](5)edge[]node[EtiqArete]{\begin{math}3\end{math}}(9);
        \draw[Arete](6)--(7);
        \draw[Arete](7)edge[]node[EtiqArete]{\begin{math}1\end{math}}(5);
        \draw[Arete](8)--(7);
        \draw[Arete](9)edge[]node[EtiqArete]{\begin{math}1\end{math}}(11);
        \node(r)at(3.00,2.50){};
        \draw[Arete](r)--(3);
    \end{tikzpicture}
    \end{split}\,,
\end{multline}
and
\begin{multline}
    \begin{split}
    \begin{tikzpicture}[xscale=.25,yscale=.2]
        \node[Feuille](0)at(0.00,-4.50){};
        \node[Feuille](2)at(2.00,-4.50){};
        \node[Feuille](4)at(4.00,-4.50){};
        \node[Feuille](6)at(6.00,-6.75){};
        \node[Feuille](8)at(8.00,-6.75){};
        \node[Noeud](1)at(1.00,-2.25){};
        \node[Noeud](3)at(3.00,0.00){};
        \node[Noeud](5)at(5.00,-2.25){};
        \node[Noeud](7)at(7.00,-4.50){};
        \draw[Arete](0)--(1);
        \draw[Arete](1)edge[]node[EtiqArete]{\begin{math}1\end{math}}(3);
        \draw[Arete](2)--(1);
        \draw[Arete](4)--(5);
        \draw[Arete](5)edge[]node[EtiqArete]{\begin{math}3\end{math}}(3);
        \draw[Arete](6)--(7);
        \draw[Arete](7)edge[]node[EtiqArete]{\begin{math}1\end{math}}(5);
        \draw[Arete](8)--(7);
        \node(r)at(3.00,1.7){};
        \draw[Arete](r)--(3);
    \end{tikzpicture}
    \end{split}
    \DDendr_2
    \begin{split}
    \begin{tikzpicture}[xscale=.25,yscale=.2]
        \node[Feuille](0)at(0.00,-4.67){};
        \node[Feuille](2)at(2.00,-4.67){};
        \node[Feuille](4)at(4.00,-4.67){};
        \node[Feuille](6)at(6.00,-4.67){};
        \node[Noeud](1)at(1.00,-2.33){};
        \node[Noeud](3)at(3.00,0.00){};
        \node[Noeud](5)at(5.00,-2.33){};
        \draw[Arete](0)--(1);
        \draw[Arete](1)edge[]node[EtiqArete]{\begin{math}1\end{math}}(3);
        \draw[Arete](2)--(1);
        \draw[Arete](4)--(5);
        \draw[Arete](5)edge[]node[EtiqArete]{\begin{math}2\end{math}}(3);
        \draw[Arete](6)--(5);
        \node(r)at(3.00,1.7){};
        \draw[Arete](r)--(3);
    \end{tikzpicture}
    \end{split}
    =
    \begin{split}
    \begin{tikzpicture}[xscale=.22,yscale=.15]
        \node[Feuille](0)at(0.00,-7.50){};
        \node[Feuille](10)at(10.00,-12.50){};
        \node[Feuille](12)at(12.00,-5.00){};
        \node[Feuille](14)at(14.00,-5.00){};
        \node[Feuille](2)at(2.00,-7.50){};
        \node[Feuille](4)at(4.00,-7.50){};
        \node[Feuille](6)at(6.00,-10.00){};
        \node[Feuille](8)at(8.00,-12.50){};
        \node[Noeud](1)at(1.00,-5.00){};
        \node[Noeud](11)at(11.00,0.00){};
        \node[Noeud](13)at(13.00,-2.50){};
        \node[Noeud](3)at(3.00,-2.50){};
        \node[Noeud](5)at(5.00,-5.00){};
        \node[Noeud](7)at(7.00,-7.50){};
        \node[Noeud](9)at(9.00,-10.00){};
        \draw[Arete](0)--(1);
        \draw[Arete](1)edge[]node[EtiqArete]{\begin{math}1\end{math}}(3);
        \draw[Arete](10)--(9);
        \draw[Arete](12)--(13);
        \draw[Arete](13)edge[]node[EtiqArete]{\begin{math}2\end{math}}(11);
        \draw[Arete](14)--(13);
        \draw[Arete](2)--(1);
        \draw[Arete](3)edge[]node[EtiqArete]{\begin{math}1\end{math}}(11);
        \draw[Arete](4)--(5);
        \draw[Arete](5)edge[]node[EtiqArete]{\begin{math}1\end{math}}(3);
        \draw[Arete](6)--(7);
        \draw[Arete](7)edge[]node[EtiqArete]{\begin{math}1\end{math}}(5);
        \draw[Arete](8)--(9);
        \draw[Arete](9)edge[]node[EtiqArete]{\begin{math}1\end{math}}(7);
        \node(r)at(11.00,2.50){};
        \draw[Arete](r)--(11);
    \end{tikzpicture}
    \end{split}
    +
    \begin{split}
    \begin{tikzpicture}[xscale=.22,yscale=.15]
        \node[Feuille](0)at(0.00,-7.50){};
        \node[Feuille](10)at(10.00,-10.00){};
        \node[Feuille](12)at(12.00,-5.00){};
        \node[Feuille](14)at(14.00,-5.00){};
        \node[Feuille](2)at(2.00,-7.50){};
        \node[Feuille](4)at(4.00,-7.50){};
        \node[Feuille](6)at(6.00,-12.50){};
        \node[Feuille](8)at(8.00,-12.50){};
        \node[Noeud](1)at(1.00,-5.00){};
        \node[Noeud](11)at(11.00,0.00){};
        \node[Noeud](13)at(13.00,-2.50){};
        \node[Noeud](3)at(3.00,-2.50){};
        \node[Noeud](5)at(5.00,-5.00){};
        \node[Noeud](7)at(7.00,-10.00){};
        \node[Noeud](9)at(9.00,-7.50){};
        \draw[Arete](0)--(1);
        \draw[Arete](1)edge[]node[EtiqArete]{\begin{math}1\end{math}}(3);
        \draw[Arete](10)--(9);
        \draw[Arete](12)--(13);
        \draw[Arete](13)edge[]node[EtiqArete]{\begin{math}2\end{math}}(11);
        \draw[Arete](14)--(13);
        \draw[Arete](2)--(1);
        \draw[Arete](3)edge[]node[EtiqArete]{\begin{math}1\end{math}}(11);
        \draw[Arete](4)--(5);
        \draw[Arete](5)edge[]node[EtiqArete]{\begin{math}1\end{math}}(3);
        \draw[Arete](6)--(7);
        \draw[Arete](7)edge[]node[EtiqArete]{\begin{math}1\end{math}}(9);
        \draw[Arete](8)--(7);
        \draw[Arete](9)edge[]node[EtiqArete]{\begin{math}1\end{math}}(5);
        \node(r)at(11.00,2.50){};
        \draw[Arete](r)--(11);
    \end{tikzpicture}
    \end{split} \\
    +
    \begin{split}
    \begin{tikzpicture}[xscale=.22,yscale=.15]
        \node[Feuille](0)at(0.00,-7.50){};
        \node[Feuille](10)at(10.00,-7.50){};
        \node[Feuille](12)at(12.00,-5.00){};
        \node[Feuille](14)at(14.00,-5.00){};
        \node[Feuille](2)at(2.00,-7.50){};
        \node[Feuille](4)at(4.00,-10.00){};
        \node[Feuille](6)at(6.00,-12.50){};
        \node[Feuille](8)at(8.00,-12.50){};
        \node[Noeud](1)at(1.00,-5.00){};
        \node[Noeud](11)at(11.00,0.00){};
        \node[Noeud](13)at(13.00,-2.50){};
        \node[Noeud](3)at(3.00,-2.50){};
        \node[Noeud](5)at(5.00,-7.50){};
        \node[Noeud](7)at(7.00,-10.00){};
        \node[Noeud](9)at(9.00,-5.00){};
        \draw[Arete](0)--(1);
        \draw[Arete](1)edge[]node[EtiqArete]{\begin{math}1\end{math}}(3);
        \draw[Arete](10)--(9);
        \draw[Arete](12)--(13);
        \draw[Arete](13)edge[]node[EtiqArete]{\begin{math}2\end{math}}(11);
        \draw[Arete](14)--(13);
        \draw[Arete](2)--(1);
        \draw[Arete](3)edge[]node[EtiqArete]{\begin{math}1\end{math}}(11);
        \draw[Arete](4)--(5);
        \draw[Arete](5)edge[]node[EtiqArete]{\begin{math}3\end{math}}(9);
        \draw[Arete](6)--(7);
        \draw[Arete](7)edge[]node[EtiqArete]{\begin{math}1\end{math}}(5);
        \draw[Arete](8)--(7);
        \draw[Arete](9)edge[]node[EtiqArete]{\begin{math}1\end{math}}(3);
        \node(r)at(11.00,2.50){};
        \draw[Arete](r)--(11);
    \end{tikzpicture}
    \end{split}
    +
    \begin{split}
    \begin{tikzpicture}[xscale=.22,yscale=.15]
        \node[Feuille](0)at(0.00,-10.00){};
        \node[Feuille](10)at(10.00,-5.00){};
        \node[Feuille](12)at(12.00,-5.00){};
        \node[Feuille](14)at(14.00,-5.00){};
        \node[Feuille](2)at(2.00,-10.00){};
        \node[Feuille](4)at(4.00,-10.00){};
        \node[Feuille](6)at(6.00,-12.50){};
        \node[Feuille](8)at(8.00,-12.50){};
        \node[Noeud](1)at(1.00,-7.50){};
        \node[Noeud](11)at(11.00,0.00){};
        \node[Noeud](13)at(13.00,-2.50){};
        \node[Noeud](3)at(3.00,-5.00){};
        \node[Noeud](5)at(5.00,-7.50){};
        \node[Noeud](7)at(7.00,-10.00){};
        \node[Noeud](9)at(9.00,-2.50){};
        \draw[Arete](0)--(1);
        \draw[Arete](1)edge[]node[EtiqArete]{\begin{math}1\end{math}}(3);
        \draw[Arete](10)--(9);
        \draw[Arete](12)--(13);
        \draw[Arete](13)edge[]node[EtiqArete]{\begin{math}2\end{math}}(11);
        \draw[Arete](14)--(13);
        \draw[Arete](2)--(1);
        \draw[Arete](3)edge[]node[EtiqArete]{\begin{math}2\end{math}}(9);
        \draw[Arete](4)--(5);
        \draw[Arete](5)edge[]node[EtiqArete]{\begin{math}3\end{math}}(3);
        \draw[Arete](6)--(7);
        \draw[Arete](7)edge[]node[EtiqArete]{\begin{math}1\end{math}}(5);
        \draw[Arete](8)--(7);
        \draw[Arete](9)edge[]node[EtiqArete]{\begin{math}1\end{math}}(11);
        \node(r)at(11.00,2.50){};
        \draw[Arete](r)--(11);
    \end{tikzpicture}
    \end{split}\,.
\end{multline}
\medskip

\begin{Lemme} \label{lem:produit_gamma_dendriforme}
    For any integer $\gamma \geq 0$, the vector space
    $\AlgLibre_{\Dendr_\gamma}$ of $\gamma$-edge valued binary trees
    endowed with the operations $\GDendr_a$, $\DDendr_a$, $a \in [\gamma]$,
    is a $\gamma$-polydendriform algebra.
\end{Lemme}
\begin{proof}
    We have to check that the operations $\GDendr_a$, $\DDendr_a$,
    $a \in [\gamma]$, of $\AlgLibre_{\Dendr_\gamma}$ satisfy
    Relations~\eqref{equ:relation_dendr_gamma_1_concise},
    \eqref{equ:relation_dendr_gamma_2_concise},
    and~\eqref{equ:relation_dendr_gamma_3_concise} of
    $\gamma$-polydendriform algebras. Let $\Rfr$, $\Sfr$, and $\Tfr$ be
    three $\gamma$-edge valued binary trees and $a, a' \in [\gamma]$.
    \smallskip

    Denote by $\Sfr_1$ (resp. $\Sfr_2$) the left subtree (resp. right
    subtree) of $\Sfr$ and by $x$ (resp. $y$) the label of the left
    (resp. right) edge incident to the root of $\Sfr$. We have
    \begin{multline}
        (\Rfr \DDendr_{a'} \Sfr) \GDendr_a \Tfr  =
        \left(\Rfr \DDendr_{a'} \ArbreBinValue{x}{y}{\Sfr_1}{\Sfr_2}\right)
        \GDendr_a \Tfr
        = \left(\ArbreBinValue{z}{y}{\Rfr \DDendr_{a'} \Sfr_1}{\Sfr_2}
        +
        \ArbreBinValue{z}{y}{\Rfr \GDendr_x \Sfr_1}{\Sfr_2}\right)
        \GDendr_a \Tfr
        \displaybreak[0]
        \\
        = \ArbreBinValue{z}{t}{\Rfr \DDendr_{a'} \Sfr_1}
                {\Sfr_2 \GDendr_a \Tfr}
        +
        \ArbreBinValue{z}{t}{\Rfr \DDendr_{a'} \Sfr_1}
                {\Sfr_2 \DDendr_y \Tfr}
        +
        \ArbreBinValue{z}{t}{\Rfr \GDendr_x \Sfr_1}
                {\Sfr_2 \GDendr_a \Tfr}
        +
        \ArbreBinValue{z}{t}{\Rfr \GDendr_x \Sfr_1}
                {\Sfr_2 \DDendr_y \Tfr}
        \displaybreak[0]
        \\
        = \Rfr \DDendr_{a'}
        \left(\ArbreBinValue{x}{t}{\Sfr_1}{\Sfr_2 \GDendr_a \Tfr}
        +
        \ArbreBinValue{x}{t}{\Sfr_1}{\Sfr_2 \DDendr_y \Tfr}\right)
        = \Rfr \DDendr_{a'}
        \left(\ArbreBinValue{x}{y}{\Sfr_1}{\Sfr_2}
        \GDendr_a \Tfr \right) =
        \Rfr \DDendr_{a'} (\Sfr \GDendr_a \Tfr),
    \end{multline}
    where $z := a' \Min x$ and $t := a \Min y$.
    This shows that~\eqref{equ:relation_dendr_gamma_1_concise}
    is satisfied in $\AlgLibre_{\Dendr_\gamma}$.
    \medskip

    We now prove that
    Relations~\eqref{equ:relation_dendr_gamma_2_concise}
    and~\eqref{equ:relation_dendr_gamma_3_concise} hold
    by induction on the sum of the number of internal nodes of $\Rfr$, $\Sfr$,
    and $\Tfr$. Base case holds when all these trees have exactly one
    internal node, and since
    \begin{multline}
        \left(\Noeud \GDendr_{a'} \Noeud\right) \GDendr_a \Noeud
        - \Noeud \GDendr_{a \Min a'} \left( \Noeud \GDendr_a \Noeud \right)
        - \Noeud \GDendr_{a \Min a'} \left( \Noeud \DDendr_{a'} \Noeud \right)
        \displaybreak[0]
        \\
        =
        \begin{split}
        \begin{tikzpicture}[xscale=.3,yscale=.25]
            \node[Feuille](0)at(0.00,-1.67){};
            \node[Feuille](2)at(2.00,-3.33){};
            \node[Feuille](4)at(4.00,-3.33){};
            \node[Noeud](1)at(1.00,0.00){};
            \node[Noeud](3)at(3.00,-1.67){};
            \draw[Arete](0)--(1);
            \draw[Arete](2)--(3);
            \draw[Arete](3)edge[]node[EtiqArete]{\begin{math}a'\end{math}}(1);
            \draw[Arete](4)--(3);
            \node(r)at(1.00,1.67){};
            \draw[Arete](r)--(1);
        \end{tikzpicture}
        \end{split}
        \GDendr_a \Noeud -
        \Noeud \GDendr_{a \Min a'}
        \begin{split}
        \begin{tikzpicture}[xscale=.3,yscale=.25]
            \node[Feuille](0)at(0.00,-1.67){};
            \node[Feuille](2)at(2.00,-3.33){};
            \node[Feuille](4)at(4.00,-3.33){};
            \node[Noeud](1)at(1.00,0.00){};
            \node[Noeud](3)at(3.00,-1.67){};
            \draw[Arete](0)--(1);
            \draw[Arete](2)--(3);
            \draw[Arete](3)edge[]node[EtiqArete]{\begin{math}a\end{math}}(1);
            \draw[Arete](4)--(3);
            \node(r)at(1.00,1.67){};
            \draw[Arete](r)--(1);
        \end{tikzpicture}
        \end{split}
        - \Noeud \GDendr_{a \Min a'}
        \begin{split}
        \begin{tikzpicture}[xscale=.3,yscale=.25]
            \node[Feuille](0)at(0.00,-3.33){};
            \node[Feuille](2)at(2.00,-3.33){};
            \node[Feuille](4)at(4.00,-1.67){};
            \node[Noeud](1)at(1.00,-1.67){};
            \node[Noeud](3)at(3.00,0.00){};
            \draw[Arete](0)--(1);
            \draw[Arete](1)edge[]node[EtiqArete]{\begin{math}a'\end{math}}(3);
            \draw[Arete](2)--(1);
            \draw[Arete](4)--(3);
            \node(r)at(3.00,1.67){};
            \draw[Arete](r)--(3);
        \end{tikzpicture}
        \end{split}
        \\
        =
        \begin{split}
        \begin{tikzpicture}[xscale=.3,yscale=.25]
            \node[Feuille](0)at(0.00,-1.75){};
            \node[Feuille](2)at(2.00,-3.50){};
            \node[Feuille](4)at(4.00,-5.25){};
            \node[Feuille](6)at(6.00,-5.25){};
            \node[Noeud](1)at(1.00,0.00){};
            \node[Noeud](3)at(3.00,-1.75){};
            \node[Noeud](5)at(5.00,-3.50){};
            \draw[Arete](0)--(1);
            \draw[Arete](2)--(3);
            \draw[Arete](3)edge[]node[EtiqArete]{\begin{math}z\end{math}}(1);
            \draw[Arete](4)--(5);
            \draw[Arete](5)edge[]node[EtiqArete]{\begin{math}a\end{math}}(3);
            \draw[Arete](6)--(5);
            \node(r)at(1.00,1.75){};
            \draw[Arete](r)--(1);
        \end{tikzpicture}
        \end{split}
        +
        \begin{split}
        \begin{tikzpicture}[xscale=.3,yscale=.25]
            \node[Feuille](0)at(0.00,-1.75){};
            \node[Feuille](2)at(2.00,-5.25){};
            \node[Feuille](4)at(4.00,-5.25){};
            \node[Feuille](6)at(6.00,-3.50){};
            \node[Noeud](1)at(1.00,0.00){};
            \node[Noeud](3)at(3.00,-3.50){};
            \node[Noeud](5)at(5.00,-1.75){};
            \draw[Arete](0)--(1);
            \draw[Arete](2)--(3);
            \draw[Arete](3)edge[]node[EtiqArete]{\begin{math}a'\end{math}}(5);
            \draw[Arete](4)--(3);
            \draw[Arete](5)edge[]node[EtiqArete]{\begin{math}z\end{math}}(1);
            \draw[Arete](6)--(5);
            \node(r)at(1.00,1.75){};
            \draw[Arete](r)--(1);
        \end{tikzpicture}
        \end{split}
        -
        \begin{split}
        \begin{tikzpicture}[xscale=.3,yscale=.25]
            \node[Feuille](0)at(0.00,-1.75){};
            \node[Feuille](2)at(2.00,-3.50){};
            \node[Feuille](4)at(4.00,-5.25){};
            \node[Feuille](6)at(6.00,-5.25){};
            \node[Noeud](1)at(1.00,0.00){};
            \node[Noeud](3)at(3.00,-1.75){};
            \node[Noeud](5)at(5.00,-3.50){};
            \draw[Arete](0)--(1);
            \draw[Arete](2)--(3);
            \draw[Arete](3)edge[]node[EtiqArete]{\begin{math}z\end{math}}(1);
            \draw[Arete](4)--(5);
            \draw[Arete](5)edge[]node[EtiqArete]{\begin{math}a\end{math}}(3);
            \draw[Arete](6)--(5);
            \node(r)at(1.00,1.75){};
            \draw[Arete](r)--(1);
        \end{tikzpicture}
        \end{split}
        -
        \begin{split}
        \begin{tikzpicture}[xscale=.3,yscale=.25]
            \node[Feuille](0)at(0.00,-1.75){};
            \node[Feuille](2)at(2.00,-5.25){};
            \node[Feuille](4)at(4.00,-5.25){};
            \node[Feuille](6)at(6.00,-3.50){};
            \node[Noeud](1)at(1.00,0.00){};
            \node[Noeud](3)at(3.00,-3.50){};
            \node[Noeud](5)at(5.00,-1.75){};
            \draw[Arete](0)--(1);
            \draw[Arete](2)--(3);
            \draw[Arete](3)edge[]node[EtiqArete]{\begin{math}a'\end{math}}(5);
            \draw[Arete](4)--(3);
            \draw[Arete](5)edge[]node[EtiqArete]{\begin{math}z\end{math}}(1);
            \draw[Arete](6)--(5);
            \node(r)at(1.00,1.75){};
            \draw[Arete](r)--(1);
        \end{tikzpicture}
        \end{split}
        \\ = 0,
    \end{multline}
    where $z := a \Min a'$,
    \eqref{equ:relation_dendr_gamma_2_concise} holds on trees
    with exactly one internal node. For the same arguments, we can show
    that~\eqref{equ:relation_dendr_gamma_3_concise} holds on
    trees with exactly one internal node. Denote now by $\Rfr_1$ (resp.
    $\Rfr_2$) the left subtree (resp. right subtree) of $\Rfr$ and by $x$
    (resp. $y$) the label of the left (resp. right) edge incident to the
    root of $\Rfr$. We have
    \begin{multline} \label{equ:produit_gamma_dendriforme_expr}
        (\Rfr \GDendr_{a'} \Sfr) \GDendr_a \Tfr
        - \Rfr \GDendr_{a \Min a'} (\Sfr \GDendr_a \Tfr)
        - \Rfr \GDendr_{a \Min a'} (\Sfr \DDendr_{a'} \Tfr)
        \\[1em]
        =
        \left(\ArbreBinValue{x}{y}{\Rfr_1}{\Rfr_2}
        \GDendr_{a'} \Sfr \right) \GDendr_a \Tfr
        -
        \ArbreBinValue{x}{y}{\Rfr_1}{\Rfr_2}
        \GDendr_{a \Min a'} (\Sfr \GDendr_a \Tfr)
        -
        \ArbreBinValue{x}{y}{\Rfr_1}{\Rfr_2}
        \GDendr_{a \Min a'} (\Sfr \DDendr_{a'} \Tfr)
        \displaybreak[0]
        \\
        =
        \left(\ArbreBinValue{x}{z}{\Rfr_1}{\Rfr_2 \GDendr_{a'} \Sfr}
        +
        \ArbreBinValue{x}{z}{\Rfr_1}{\Rfr_2 \DDendr_y \Sfr}
        \right) \GDendr_a \Tfr \\
        -
        \ArbreBinValue{x}{y}{\Rfr_1}{\Rfr_2}
        \GDendr_{a \Min a'} (\Sfr \GDendr_a \Tfr)
        -
        \ArbreBinValue{x}{y}{\Rfr_1}{\Rfr_2}
        \GDendr_{a \Min a'} (\Sfr \DDendr_{a'} \Tfr)
        \displaybreak[0]
        \\
        =
        \ArbreBinValue{x}{t}{\Rfr_1}{(\Rfr_2 \GDendr_{a'} \Sfr) \GDendr_a \Tfr}
        +
        \ArbreBinValue{x}{t}{\Rfr_1}{(\Rfr_2 \GDendr_{a'} \Sfr) \DDendr_z \Tfr}
        +
        \ArbreBinValue{x}{t}{\Rfr_1}{(\Rfr_2 \DDendr_y \Sfr) \GDendr_a \Tfr}
        +
        \ArbreBinValue{x}{t}{\Rfr_1}{(\Rfr_2 \DDendr_y \Sfr) \DDendr_z \Tfr} \\
        -
        \ArbreBinValue{x}{t}{\Rfr_1}{\Rfr_2 \GDendr_u (\Sfr \GDendr_a \Tfr)}
        -
        \ArbreBinValue{x}{t}{\Rfr_1}{\Rfr_2 \DDendr_y (\Sfr \GDendr_a \Tfr)}
        -
        \ArbreBinValue{x}{t}{\Rfr_1}{\Rfr_2 \GDendr_u (\Sfr \DDendr_{a'} \Tfr)}
        -
        \ArbreBinValue{x}{t}{\Rfr_1}{\Rfr_2 \DDendr_y (\Sfr \DDendr_{a'} \Tfr)}\,,
    \end{multline}
    where $z := y \Min a'$, $t := z \Min a = y \Min a' \Min a$, and
    $u := a \Min a'$. Now, by induction hypothesis,
    Relation~\eqref{equ:relation_dendr_gamma_2_concise} holds
    on $\Rfr_2$, $\Sfr$, and $\Tfr$. Hence, the sum of the first, fifth,
    and seventh terms of~\eqref{equ:produit_gamma_dendriforme_expr} is
    zero. Again by induction hypothesis,
    Relation~\eqref{equ:relation_dendr_gamma_3_concise} holds
    on $\Rfr_2$, $\Sfr$, and $\Tfr$. Thus, the sum of the second, fourth,
    and last terms of~\eqref{equ:produit_gamma_dendriforme_expr} is zero.
    Finally, by what we just have proven in the first part of this proof,
    the sum of the third and sixth terms
    of~\eqref{equ:relation_dendr_gamma_3_concise} is zero.
    Therefore, \eqref{equ:produit_gamma_dendriforme_expr} is zero
    and~\eqref{equ:relation_dendr_gamma_2_concise}
    is satisfied in $\AlgLibre_{\Dendr_\gamma}$.
    \smallskip

    Finally, for the same arguments, we can show
    that~\eqref{equ:relation_dendr_gamma_3_concise}
    is satisfied in $\AlgLibre_{\Dendr_\gamma}$, implying the statement
    of the lemma.
\end{proof}
\medskip

\begin{Lemme} \label{lem:produit_gamma_engendre_dendriforme}
    For any integer $\gamma \geq 0$, the $\gamma$-pluriassociative
    algebra $\AlgLibre_{\Dendr_\gamma}$ of $\gamma$-edge valued binary
    trees endowed with the operations $\GDendr_a$, $\DDendr_a$,
    $a \in [\gamma]$, is generated by
    \begin{equation}
        \Noeud\,.
    \end{equation}
\end{Lemme}
\begin{proof}
    First, Lemma~\ref{lem:produit_gamma_dendriforme} shows that
    $\AlgLibre_{\Dendr_\gamma}$ is a $\gamma$-polydendriform algebra.
    Let $\Dca$ be the $\gamma$-polydendriform subalgebra of
    $\AlgLibre_{\Dendr_\gamma}$ generated by $\NoeudTexte$. Let us show
    that any $\gamma$-edge valued binary tree $\Tfr$ is in $\Dca$ by
    induction on the number $n$ of its internal nodes. When $n = 1$,
    $\Tfr = \NoeudTexte$ and hence the property is satisfied. Otherwise,
    let $\Tfr_1$ (resp. $\Tfr_2$) be the left (resp. right) subtree of
    the root of $\Tfr$ and denote by $x$ (resp. $y$) the label of the
    left (resp. right) edge incident to the root of $\Tfr$. Since
    $\Tfr_1$ and $\Tfr_2$ have less internal nodes than $\Tfr$, by
    induction hypothesis, $\Tfr_1$ and~$\Tfr_2$ are in $\Dca$. Moreover,
    by definition of the operations $\GDendr_a$, $\DDendr_a$,
    $a \in [\gamma]$, of $\AlgLibre_{\Dendr_\gamma}$, one has
    \begin{equation}
        \begin{split}\end{split}
        \left(\Tfr_1 \DDendr_x \Noeud\right) \GDendr_y \Tfr_2
        =
        \begin{split}
        \begin{tikzpicture}[xscale=.5,yscale=.4]
            \node(0)at(-.50,-1.50){\begin{math}\Tfr_1\end{math}};
            \node[Feuille](2)at(2.0,-1.00){};
            \node[Noeud](1)at(1.00,.50){};
            \draw[Arete](0)edge[]node[EtiqArete]
                {\begin{math}x\end{math}}(1);
            \draw[Arete](2)--(1);
            \node(r)at(1.00,1.5){};
            \draw[Arete](r)--(1);
        \end{tikzpicture}
        \end{split}
        \GDendr_y \Tfr_2
        =
        \ArbreBinValue{x}{y}{\Tfr_1}{\Tfr_2}
        = \Tfr,
    \end{equation}
    showing that $\Tfr$ also is in $\Dca$. Therefore,
    $\Dca$ is $\AlgLibre_{\Dendr_\gamma}$, showing that
    $\AlgLibre_{\Dendr_\gamma}$
    is generated by~$\NoeudTexte$.
\end{proof}
\medskip

\begin{Theoreme} \label{thm:algebre_dendr_gamma_libre}
    For any integer $\gamma \geq 0$, the vector space
    $\AlgLibre_{\Dendr_\gamma}$ of $\gamma$-edge valued binary trees
    endowed with the operations $\GDendr_a$, $\DDendr_a$, $a \in [\gamma]$,
    is the free $\gamma$-polydendriform algebra over one generator.
\end{Theoreme}
\begin{proof}
    By Lemmas~\ref{lem:produit_gamma_dendriforme}
    and~\ref{lem:produit_gamma_engendre_dendriforme},
    $\AlgLibre_{\Dendr_\gamma}$ is a $\gamma$-polydendriform algebra
    over one generator.
    \smallskip

    Moreover, since by Proposition~\ref{prop:serie_hilbert_dendr_gamma},
    for any $n \geq 1$, the dimension of $\AlgLibre_{\Dendr_\gamma}(n)$
    is the same as the dimension of $\Dendr_\gamma(n)$, there cannot be
    relations in $\AlgLibre_{\Dendr_\gamma}(n)$ involving $\Gfr$ that are
    not $\gamma$-polydendriform
    relations (see~\eqref{equ:relation_dendr_gamma_1_concise},
    \eqref{equ:relation_dendr_gamma_2_concise},
    and~\eqref{equ:relation_dendr_gamma_3_concise}). Hence,
    $\AlgLibre_{\Dendr_\gamma}$ is free as a $\gamma$-polydendriform
    algebra over one generator.
\end{proof}
\medskip

%%%%%%%%%%%%%%%%%%%%%%%%%%%%%%%%%%%%%%%%%%%%%%%%%%%%%%%%%%%%%%%%%%%%%%%%
%%%%%%%%%%%%%%%%%%%%%%%%%%%%%%%%%%%%%%%%%%%%%%%%%%%%%%%%%%%%%%%%%%%%%%%%
%%%%%%%%%%%%%%%%%%%%%%%%%%%%%%%%%%%%%%%%%%%%%%%%%%%%%%%%%%%%%%%%%%%%%%%%
\section{Multiassociative operads} \label{sec:as_gamma}
There is a well-known diagram, whose definition is recalled below,
gathering the diassociative, associative, and  dendriform operads.
The main goal of this section is to define a generalization on a
nonnegative integer parameter of the associative operad to obtain a new
version of this diagram, suited to the context of pluriassociative and
polydendriform operads.
\medskip

%%%%%%%%%%%%%%%%%%%%%%%%%%%%%%%%%%%%%%%%%%%%%%%%%%%%%%%%%%%%%%%%%%%%%%%%
%%%%%%%%%%%%%%%%%%%%%%%%%%%%%%%%%%%%%%%%%%%%%%%%%%%%%%%%%%%%%%%%%%%%%%%%
\subsection{Two generalizations of the associative operad}
The associative operad is generated by one binary element. This operad
admits two different generalizations generated by $\gamma$ binary
elements with the particularity that one is the Koszul dual of the other.
We introduce and study in this section these two operads.
\medskip

%%%%%%%%%%%%%%%%%%%%%%%%%%%%%%%%%%%%%%%%%%%%%%%%%%%%%%%%%%%%%%%%%%%%%%%%
\subsubsection{Nonsymmetric associative operad}
Recall that the {\em nonsymmetric associative operad}, or the
{\em associative operad} for short, is the operad $\As$ admitting the
presentation $\left(\GenLibre_\As, \RelLibre_\As\right)$, where
$\GenLibre_\As := \GenLibre_\As(2) := \{\MAs\}$ and $\RelLibre_\As$ is
generated by $\MAs \circ_1 \MAs - \MAs \circ_2 \MAs$. It admits the
following realization. For any $n \geq 1$, $\As(n)$ is the vector space
of dimension one generated by the corolla of arity $n$ and the partial
composition $\Cfr_1 \circ_i \Cfr_2$ where $\Cfr_1$ is the corolla of
arity $n$ and $\Cfr_2$ is the corolla of arity $m$ is the corolla of
arity $n + m - 1$ for all valid $i$.
\medskip

%%%%%%%%%%%%%%%%%%%%%%%%%%%%%%%%%%%%%%%%%%%%%%%%%%%%%%%%%%%%%%%%%%%%%%%%
\subsubsection{Multiassociative operads} \label{subsubsec:as_gamma}
For any integer $\gamma \geq 0$, we define $\As_\gamma$ as the operad
admitting the presentation $\left(\GenAs, \RelAs\right)$, where
$\GenLibre_{\As_\gamma} := \GenLibre_{\As_\gamma}(2)
:= \{\MAs_a : a \in [\gamma]\}$ and $\RelLibre_{\As_\gamma}$ is generated
by
\begin{subequations}
\begin{equation} \label{equ:relation_as_gamma_1}
    \MAs_a \circ_1 \MAs_b - \MAs_b \circ_2 \MAs_b,
    \qquad a \leq b \in [\gamma],
\end{equation}
\begin{equation} \label{equ:relation_as_gamma_2}
    \MAs_b \circ_1 \MAs_a - \MAs_b \circ_2 \MAs_b,
    \qquad a < b \in [\gamma],
\end{equation}
\begin{equation} \label{equ:relation_as_gamma_3}
    \MAs_a \circ_2 \MAs_b - \MAs_b \circ_2 \MAs_b,
    \qquad a < b \in [\gamma],
\end{equation}
\begin{equation} \label{equ:relation_as_gamma_4}
    \MAs_b \circ_2 \MAs_a - \MAs_b \circ_2 \MAs_b,
    \qquad a < b \in [\gamma].
\end{equation}
\end{subequations}
This space of relations can be rephrased in a more compact way as the
space generated by
\begin{subequations}
\begin{equation} \label{equ:relation_as_gamma_1_concise}
    \MAs_a \circ_1 \MAs_{a'} - \MAs_{a \Max a'} \circ_2 \MAs_{a \Max a'},
    \qquad a, a' \in [\gamma],
\end{equation}
\begin{equation} \label{equ:relation_as_gamma_1_concise}
    \MAs_a \circ_2 \MAs_{a'} - \MAs_{a \Max a'} \circ_2 \MAs_{a \Max a'},
    \qquad a, a' \in [\gamma].
\end{equation}
\end{subequations}
We call $\As_\gamma$ the {\em $\gamma$-multiassociative operad}.
\medskip

It follows immediately that $\As_\gamma$ is a set-operad and that it
provides a generalization of the associative operad. The algebras over
$\As_\gamma$ are the $\gamma$-multiassociative algebras introduced in
Section~3.3.1 of~\cite{GirI}.
\medskip

Let us now provide a realization of $\As_\gamma$. A {\em $\gamma$-corolla}
is a rooted tree with at most one internal node labeled on $[\gamma]$.
Denote by $\AlgLibre_{\As_\gamma}(n)$ the vector space of $\gamma$-corollas
of arity $n \geq 1$, by $\AlgLibre_{\As_\gamma}$ the graded vector space
of all $\gamma$-corollas, and let
\begin{equation}
    \MAs : \AlgLibre_{\As_\gamma} \otimes \AlgLibre_{\As_\gamma}
    \to \AlgLibre_{\As_\gamma}
\end{equation}
be the linear operation where, for any $\gamma$-corollas $\Cfr_1$ and
$\Cfr_2$, $\Cfr_1 \MAs \Cfr_2$ is the $\gamma$-corolla
with $n + m - 1$ leaves and labeled by $a \Max a'$ where $n$ (resp. $m$)
is the number of leaves of $\Cfr_1$ (resp. $\Cfr_2$) and $a$ (resp. $a'$)
is the label of $\Cfr_1$ (resp. $\Cfr_2$).
\medskip

\begin{Proposition} \label{prop:realisation_koszulite_as_gamma}
    For any integer $\gamma \geq 0$, the operad $\As_\gamma$ is the
    vector space $\AlgLibre_{\As_\gamma}$ of $\gamma$-corollas and its
    partial compositions satisfy, for any $\gamma$-corollas $\Cfr_1$ and
    $\Cfr_2$, $\Cfr_1 \circ_i \Cfr_2 = \Cfr_1 \MAs \Cfr_2$ for all valid
    integer $i$. Besides, $\As_\gamma$ is a Koszul operad and the set of
    right comb syntax trees of $\OpLibre\left(\GenAs\right)$ where all
    internal nodes have a same label forms a Poincaré-Birkhoff-Witt basis
    of~$\As_\gamma$.
\end{Proposition}
\begin{proof}
    In this proof, we consider that $\GenAs$ is totally ordered by the
    relation $\leq$ satisfying $\MAs_a \leq \MAs_b$ whenever
    $a \leq b \in [\gamma]$. It is immediate that the vector space
    $\AlgLibre_{\As_\gamma}$ endowed with the partial compositions
    described in the statement of the proposition is an operad. Let us
    prove that this operad  admits the presentation
    $\left(\GenAs, \RelAs\right)$.
    \smallskip

    For this purpose, consider the quadratic rewrite rule $\Recr_\gamma$ on
    $\OpLibre\left(\GenAs\right)$ satisfying
    \begin{subequations}
    \begin{equation} \label{equ:reecriture_as_gamma_1}
        \MAs_a \circ_1 \MAs_b
        \enspace \Recr_\gamma \enspace
        \MAs_b \circ_2 \MAs_b,
        \qquad a \leq b \in [\gamma],
    \end{equation}
    \begin{equation} \label{equ:reecriture_as_gamma_2}
        \MAs_b \circ_1 \MAs_a
        \enspace \Recr_\gamma \enspace
        \MAs_b \circ_2 \MAs_b,
        \qquad a < b \in [\gamma],
    \end{equation}
    \begin{equation} \label{equ:reecriture_as_gamma_3}
        \MAs_a \circ_2 \MAs_b
        \enspace \Recr_\gamma \enspace
        \MAs_b \circ_2 \MAs_b,
        \qquad a < b \in [\gamma],
    \end{equation}
    \begin{equation} \label{equ:reecriture_as_gamma_4}
        \MAs_b \circ_2 \MAs_a
        \enspace \Recr_\gamma \enspace
        \MAs_b \circ_2 \MAs_b,
        \qquad a < b \in [\gamma].
    \end{equation}
    \end{subequations}
    Observe first that the space induced by the operad congruence induced
    by $\Recr_\gamma$ is $\RelAs$
    (see~\eqref{equ:relation_as_gamma_1}---\eqref{equ:relation_as_gamma_4}).
    Moreover, $\Recr_\gamma$ is a terminating rewrite rule and its normal
    forms are right comb syntax trees of $\OpLibre\left(\GenAs\right)$
    where all internal nodes have a same label. Besides, one can show
    that for any syntax tree $\Tfr$ of $\OpLibre\left(\GenAs\right)$, we
    have $\Tfr \overset{*}{\Recr_\gamma} \Sfr$ with $\Sfr$ is a right
    comb syntax tree where all internal nodes labeled by the greatest
    label of $\Tfr$. Therefore, $\Recr_\gamma$ is a convergent rewrite
    rule and the operad $\As$, admitting by definition the presentation
    $\left(\GenAs, \RelAs\right)$, has bases indexed by such trees.
    \smallskip

    Now, let
    \begin{equation}
        \phi : \As_\gamma \simeq
        \OpLibre\left(\GenAs\right)/_{\left\langle \RelAs \right\rangle}
        \to \AlgLibre_{\As_\gamma}
    \end{equation}
    be the map satisfying $\phi(\pi(\MAs_a)) = \Cfr_a$ where $\Cfr_a$ is
    the $\gamma$-corolla of arity $2$ with internal node labeled by
    $a \in [\gamma]$ and $\pi : \OpLibre\left(\GenAs\right) \to
    \As_\gamma$ is the canonical surjection map. Since we have
    $\phi(\pi(x)) \circ_i \phi(\pi(y)) = \phi(\pi(x')) \circ_{i'} \phi(\pi(y'))$
    for all relations $x \circ_i y \Recr_\gamma x' \circ_{i'} y'$ of
    \eqref{equ:reecriture_as_gamma_1}---\eqref{equ:reecriture_as_gamma_4},
    $\phi$ extends in a unique way into an operad morphism. First, since
    the set $G_\gamma$ of all $\gamma$-corollas of arity two is a
    generating set of $\AlgLibre_{\As_\gamma}$ and the image of $\phi$
    contains $G_\gamma$, $\phi$ is surjective. Second, since by
    definition of $\AlgLibre_{\As_\gamma}$, the bases of
    $\AlgLibre_{\As_\gamma}$ are indexed by $\gamma$-corollas, in
    accordance with what we have shown in the previous paragraph of this
    proof, $\AlgLibre_{\As_\gamma}$ and $\As_\gamma$ are isomorphic
    as graded vector spaces. Hence, $\phi$ is an operad isomorphism,
    showing that $\As_\gamma$ admits the claimed realization.
    \smallskip

    Finally, the existence of the convergent rewrite rule $\Recr_\gamma$
    implies, by the Koszulity criterion~\cite{Hof10,DK10,LV12} we have
    reformulated in Section~1.2.5 of~\cite{GirI},
    that $\As_\gamma$ is Koszul and that its Poincaré-Birkhoff-Witt basis
    is the one described in the statement of the proposition.
\end{proof}
\medskip

We have for instance in $\As_3$,
\begin{equation}
    \begin{split}
    \begin{tikzpicture}[xscale=.27,yscale=.23]
        \node[Feuille](0)at(0.00,-1.50){};
        \node[Feuille](2)at(1.00,-1.50){};
        \node[Feuille](3)at(2.00,-1.50){};
        \node[NoeudCor](1)at(1.00,0.00){\begin{math}2\end{math}};
        \draw[Arete](0)--(1);
        \draw[Arete](2)--(1);
        \draw[Arete](3)--(1);
        \node(r)at(1.00,1.7){};
        \draw[Arete](r)--(1);
    \end{tikzpicture}
    \end{split}
    \circ_1
    \begin{split}
    \begin{tikzpicture}[xscale=.27,yscale=.23]
        \node[Feuille](0)at(0.00,-1.50){};
        \node[Feuille](2)at(2.00,-1.50){};
        \node[NoeudCor](1)at(1.00,0.00){\begin{math}1\end{math}};
        \draw[Arete](0)--(1);
        \draw[Arete](2)--(1);
        \node(r)at(1.00,1.7){};
        \draw[Arete](r)--(1);
    \end{tikzpicture}
    \end{split}
    =
    \begin{split}
    \begin{tikzpicture}[xscale=.27,yscale=.23]
        \node[Feuille](0)at(0.00,-1.50){};
        \node[Feuille](1)at(1.00,-1.50){};
        \node[Feuille](3)at(3.00,-1.50){};
        \node[Feuille](4)at(4.00,-1.50){};
        \node[NoeudCor](2)at(2.00,0.00){\begin{math}2\end{math}};
        \draw[Arete](0)--(2);
        \draw[Arete](1)--(2);
        \draw[Arete](3)--(2);
        \draw[Arete](4)--(2);
        \node(r)at(2.00,1.7){};
        \draw[Arete](r)--(2);
    \end{tikzpicture}
    \end{split}\,,
\end{equation}
and
\begin{equation}
    \begin{split}
    \begin{tikzpicture}[xscale=.27,yscale=.23]
        \node[Feuille](0)at(0.00,-1.50){};
        \node[Feuille](2)at(2.00,-1.50){};
        \node[NoeudCor](1)at(1.00,0.00){\begin{math}2\end{math}};
        \draw[Arete](0)--(1);
        \draw[Arete](2)--(1);
        \node(r)at(1.00,1.7){};
        \draw[Arete](r)--(1);
    \end{tikzpicture}
    \end{split}
    \circ_2
    \begin{split}
    \begin{tikzpicture}[xscale=.27,yscale=.23]
        \node[Feuille](0)at(0.00,-1.50){};
        \node[Feuille](2)at(1.00,-1.50){};
        \node[Feuille](3)at(2.00,-1.50){};
        \node[NoeudCor](1)at(1.00,0.00){\begin{math}3\end{math}};
        \draw[Arete](0)--(1);
        \draw[Arete](2)--(1);
        \draw[Arete](3)--(1);
        \node(r)at(1.00,1.7){};
        \draw[Arete](r)--(1);
    \end{tikzpicture}
    \end{split}
    =
    \begin{split}
    \begin{tikzpicture}[xscale=.27,yscale=.23]
        \node[Feuille](0)at(0.00,-1.50){};
        \node[Feuille](1)at(1.00,-1.50){};
        \node[Feuille](3)at(3.00,-1.50){};
        \node[Feuille](4)at(4.00,-1.50){};
        \node[NoeudCor](2)at(2.00,0.00){\begin{math}3\end{math}};
        \draw[Arete](0)--(2);
        \draw[Arete](1)--(2);
        \draw[Arete](3)--(2);
        \draw[Arete](4)--(2);
        \node(r)at(2.00,1.7){};
        \draw[Arete](r)--(2);
    \end{tikzpicture}
    \end{split}\,.
\end{equation}
\medskip

We deduce from Proposition~\ref{prop:realisation_koszulite_as_gamma}
that the Hilbert series of $\As_\gamma$ satisfies
\begin{equation} \label{equ:serie_hilbert_as_gamma}
    \Hca_{\As_\gamma}(t) = \frac{t + (\gamma - 1)t^2}{1 - t}.
\end{equation}
and that for all $n \geq 2$, $\dim \As_\gamma(n) = \gamma$.
\medskip

%%%%%%%%%%%%%%%%%%%%%%%%%%%%%%%%%%%%%%%%%%%%%%%%%%%%%%%%%%%%%%%%%%%%%%%%
\subsubsection{Dual multiassociative operads} \label{subsubsec:das_gamma}
Since $\As_\gamma$ is a binary and quadratic operad, its admits a Koszul
dual, denoted by $\DAs_\gamma$ and called
{\em $\gamma$-dual multiassociative operad}. The presentation of this
operad is provided by next result.
\medskip

\begin{Proposition} \label{prop:presentation_as_gamma_duale}
    For any integer $\gamma \geq 0$, the operad $\DAs_\gamma$ admits the
    following presentation. It is generated by
    $\GenDAs := \GenDAs(2) := \{\MDAsA_a : a \in [\gamma]\}$ and its space
    of relations $\RelDAs$ is generated by
    \begin{equation} \label{equ:relation_das_gamma_alternative}
        \MDAsA_b \circ_1 \MDAsA_b - \MDAsA_b \circ_2 \MDAsA_b
        + \left(\sum_{a < b} \MDAsA_a \circ_1 \MDAsA_b
        + \MDAsA_b \circ_1 \MDAsA_a - \MDAsA_a \circ_2 \MDAsA_b
        - \MDAsA_b \circ_2 \MDAsA_a\right),
        \qquad b \in [\gamma].
    \end{equation}
\end{Proposition}
\begin{proof}
    By a straightforward computation, and by identifying $\MDAsA_a$ with
    $\MAs_a$ for any $a \in [\gamma]$, we obtain that the space $\RelDAs$
    of the statement of the proposition satisfies $\RelDAs^\perp = \RelAs$.
    Hence, $\DAs$ admits the claimed presentation.
\end{proof}
\medskip

For any integer $\gamma \geq 0$, let $\MDAs_b$, $b \in [\gamma]$, the
elements of $\OpLibre\left(\GenDAs\right)$ defined by
\begin{equation} \label{equ:definition_operateur_das_gamma}
    \MDAs_b := \sum_{a \in [b]} \MDAsA_a.
\end{equation}
Then, since for all $b \in [\gamma]$ we have
\begin{equation}
    \MDAsA_b =
    \begin{cases}
        \MDAs_1 & \mbox{if } b = 1, \\
        \MDAs_b - \MDAs_{b - 1} & \mbox{otherwise},
    \end{cases}
\end{equation}
by triangularity, the family $\GenDAs' := \{\MDAs_b : b \in [\gamma]\}$
forms a basis of $\OpLibre\left(\GenDAs\right)(2)$ and then, generates
$\OpLibre\left(\GenDAs\right)$ as an operad. Let us now express a
presentation of $\DAs_\gamma$ through the family $\GenDAs'$.

\begin{Proposition} \label{prop:autre_presentation_das_gamma}
    For any integer $\gamma \geq 0$, the operad $\DAs_\gamma$ admits the
    following presentation. It is generated by $\GenDAs'$ and its space
    of relations $\RelDAs'$ is generated by
    \begin{equation} \label{equ:relation_das_gamma}
        \MDAs_a \circ_1 \MDAs_a - \MDAs_a \circ_2 \MDAs_a,
        \qquad a \in [\gamma].
    \end{equation}
\end{Proposition}
\begin{proof}
    Let us show that $\RelDAs'$ is equal to the space of relations
    $\RelDAs$ of $\DAs_\gamma$ defined in the statement of
    Proposition~\ref{prop:presentation_as_gamma_duale}. By this last
    proposition, for any $x \in \OpLibre\left(\GenDAs\right)(3)$,
    $x$ is in $\RelDAs$ if and only if $\pi(x) = 0$ where
    $\pi : \OpLibre\left(\GenDAs\right) \to \DAs$ is the canonical
    surjection map. By a straightforward computation, by
    expanding~\eqref{equ:relation_das_gamma} over the elements $\MDAsA_a$,
    $a \in [\gamma]$, by
    using~\eqref{equ:definition_operateur_das_gamma} we obtain
    that~\eqref{equ:relation_das_gamma} can be expressed as a sum
    of elements of $\RelDAs$. This implies that $\pi(x) = 0$ and hence
    that $\RelDAs'$ is a subspace of $\RelDAs$.
    \smallskip

    Now, one can observe that for all $a \in [\gamma]$, the
    elements~\eqref{equ:relation_das_gamma} are linearly independent.
    Then, $\RelDAs'$ has dimension $\gamma$ which is also, by
    Proposition~\ref{prop:presentation_as_gamma_duale}, the dimension of
    $\RelDAs$. The statement of the proposition follows.
\end{proof}
\medskip

Observe, from the presentation provided by
Proposition~\ref{prop:autre_presentation_das_gamma} of $\DAs_\gamma$,
that $\DAs_2$ is the operad denoted by $\DeuxAs$ in \cite{LR06}.
\medskip

Notice that the Koszul dual of $\DAs_\gamma$ through its presentation
$\left(\GenDAs', \RelDAs'\right)$ of
Proposition~\ref{prop:autre_presentation_das_gamma} gives rise to
the following presentation for $\As_\gamma$. This last operad
admits the presentation $\left(\GenAs', \RelAs'\right)$ where
$\GenAs' := \GenAs'(2) := \{\MAsA_a : a \in [\gamma]\}$
and $\RelAs'$ is generated by
\begin{subequations}
\begin{equation}
    \MAsA_a \circ_1 \MAsA_{a'}, \qquad a \ne a' \in [\gamma],
\end{equation}
\begin{equation}
    \MAsA_a \circ_2 \MAsA_{a'}, \qquad a \ne a' \in [\gamma],
\end{equation}
\begin{equation}
    \MAsA_a \circ_1 \MAsA_a -
    \MAsA_a \circ_2 \MAsA_a, \qquad a \in [\gamma].
\end{equation}
\end{subequations}
Indeed, $\RelAs'$ is the space $\RelAs$ through the identification
\begin{equation}
    \MAsA_a =
    \begin{cases}
        \MAs_\gamma & \mbox{if } a = \gamma, \\
        \MAs_a - \MAs_{a + 1} & \mbox{otherwise}.
    \end{cases}
\end{equation}
\medskip

\begin{Proposition} \label{prop:serie_hilbert_das_gamma}
    For any integer $\gamma \geq 0$, the Hilbert series
    $\Hca_{\DAs_\gamma}(t)$ of the operad $\DAs_\gamma$ satisfies
    \begin{equation} \label{equ:serie_hilbert_das_gamma}
        \Hca_{\DAs_\gamma}(t) =
        t + t\, \Hca_{\DAs_\gamma}(t) +
        (\gamma - 1) \, \Hca_{\DAs_\gamma}(t)^2.
    \end{equation}
\end{Proposition}
\begin{proof}
    By setting $\bar \Hca_{\DAs_\gamma}(t) := \Hca_{\DAs_\gamma}(-t)$,
    from~\eqref{equ:serie_hilbert_das_gamma}, we obtain
    \begin{equation}
        t =
        \frac{-\bar \Hca_{\DAs_\gamma}(t) +
            (\gamma - 1)\bar \Hca_{\DAs_\gamma}(t)^2}
             {1 + \bar \Hca_{\DAs_\gamma}(t)}.
    \end{equation}
    Moreover, by setting
    $\bar \Hca_{\As_\gamma}(t) := \Hca_{\As_\gamma}(-t)$, where
    $\Hca_{\As_\gamma}(t)$ is defined by~\eqref{equ:serie_hilbert_as_gamma},
    we have
    \begin{equation} \label{equ:serie_hilbert_das_gamma_demo}
        \bar \Hca_{\As_\gamma}\left(\bar \Hca_{\DAs_\gamma}(t)\right) =
        \frac{-\bar \Hca_{\DAs_\gamma}(t) +
            (\gamma - 1)\bar \Hca_{\DAs_\gamma}(t)^2}
             {1 + \bar \Hca_{\DAs_\gamma}(t)}
        = t,
    \end{equation}
    showing that $\bar \Hca_{\As_\gamma}(t)$ and $\bar \Hca_{\DAs_\gamma}(t)$
    are the inverses for each other for series composition.
    \smallskip

    Now, since by Proposition~\ref{prop:realisation_koszulite_as_gamma},
    $\As_\gamma$ is a Koszul operad and its Hilbert series is
    $\Hca_{\As_\gamma}(t)$, and since $\DAs_\gamma$ is by definition
    the Koszul dual of $\As_\gamma$, the Hilbert series of these two
    operads satisfy~\eqref{equ:relation_series_hilbert_operade_duale}.
    Therefore, \eqref{equ:serie_hilbert_das_gamma_demo} implies that
    the Hilbert series of $\DAs_\gamma$ is $\Hca_{\DAs_\gamma}(t)$.
\end{proof}
\medskip

A {\em Schröder tree}~\cite{Sta01,Sta11} is a planar rooted tree such
that internal nodes have two of more children. By examining the expression
for $\Hca_{\DAs_\gamma}(t)$ of the statement of
Proposition~\ref{prop:serie_hilbert_das_gamma}, we observe that for any
$n \geq 1$, $\DAs_\gamma(n)$ can be seen as the vector space
$\AlgLibre_{\DAs_\gamma}(n)$ of Schröder trees with $n$ internal nodes,
all labeled on $[\gamma]$ such that the label of an internal node is
different from the labels of its children that are internal nodes. We
call these trees {\em $\gamma$-alternating Schröder trees}. Let us also
denote by $\AlgLibre_{\DAs_\gamma}$the graded vector space of all
$\gamma$-alternating Schröder trees. For instance,
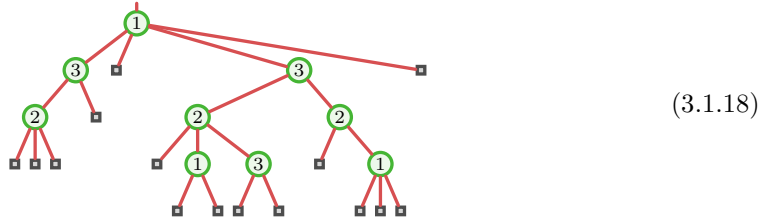
\begin{equation}
    \begin{split}
    \begin{tikzpicture}[xscale=.27,yscale=.13]
        \node[Feuille](0)at(0.00,-14.40){};
        \node[Feuille](10)at(8.00,-19.20){};
        \node[Feuille](12)at(10.00,-19.20){};
        \node[Feuille](13)at(11.00,-19.20){};
        \node[Feuille](15)at(13.00,-19.20){};
        \node[Feuille](17)at(15.00,-14.40){};
        \node[Feuille](19)at(17.00,-19.20){};
        \node[Feuille](2)at(1.00,-14.40){};
        \node[Feuille](21)at(18.00,-19.20){};
        \node[Feuille](22)at(19.00,-19.20){};
        \node[Feuille](23)at(20.00,-4.80){};
        \node[Feuille](3)at(2.00,-14.40){};
        \node[Feuille](5)at(4.00,-9.60){};
        \node[Feuille](6)at(5.00,-4.80){};
        \node[Feuille](8)at(7.00,-14.40){};
        \node[NoeudSchr](1)at(1.00,-9.60){\begin{math}2\end{math}};
        \node[NoeudSchr](11)at(9.00,-14.40){\begin{math}1\end{math}};
        \node[NoeudSchr](14)at(12.00,-14.40){\begin{math}3\end{math}};
        \node[NoeudSchr](16)at(14.00,-4.80){\begin{math}3\end{math}};
        \node[NoeudSchr](18)at(16.00,-9.60){\begin{math}2\end{math}};
        \node[NoeudSchr](20)at(18.00,-14.40){\begin{math}1\end{math}};
        \node[NoeudSchr](4)at(3.00,-4.80){\begin{math}3\end{math}};
        \node[NoeudSchr](7)at(6.00,0.00){\begin{math}1\end{math}};
        \node[NoeudSchr](9)at(9.00,-9.60){\begin{math}2\end{math}};
        \draw[Arete](0)--(1);
        \draw[Arete](1)--(4);
        \draw[Arete](10)--(11);
        \draw[Arete](11)--(9);
        \draw[Arete](12)--(11);
        \draw[Arete](13)--(14);
        \draw[Arete](14)--(9);
        \draw[Arete](15)--(14);
        \draw[Arete](16)--(7);
        \draw[Arete](17)--(18);
        \draw[Arete](18)--(16);
        \draw[Arete](19)--(20);
        \draw[Arete](2)--(1);
        \draw[Arete](20)--(18);
        \draw[Arete](21)--(20);
        \draw[Arete](22)--(20);
        \draw[Arete](23)--(7);
        \draw[Arete](3)--(1);
        \draw[Arete](4)--(7);
        \draw[Arete](5)--(4);
        \draw[Arete](6)--(7);
        \draw[Arete](8)--(9);
        \draw[Arete](9)--(16);
        \node(r)at(6.00,3){};
        \draw[Arete](r)--(7);
    \end{tikzpicture}
    \end{split}
\end{equation}
is a $3$-alternating Schröder tree and a basis element of $\DAs_3(9)$.
\medskip

We deduce also from Proposition~\ref{prop:serie_hilbert_das_gamma} that
\begin{equation}
    \Hca_{\DAs_\gamma}(t) =
    \frac{1 - \sqrt{1 - (4\gamma - 2)t + t^2} - t}{2(\gamma - 1)}.
\end{equation}
By denoting by $\Nar(n, k)$ the {\em Narayana number}~\cite{Nar55}
defined by
\begin{equation} \label{equ:definition_narayana}
    \Nar(n, k) := \frac{1}{k + 1} \binom{n - 1}{k} \binom{n}{k},
\end{equation}
we obtain that for all $n \geq 1$,
\begin{equation}
    \dim \DAs_\gamma(n) =
    \sum_{k = 0}^{n - 2}
    \gamma^{k + 1} (\gamma - 1)^{n - k - 2} \, \Nar(n - 1, k).
\end{equation}
This formula is a consequence of the fact that $\Nar(n - 1, k)$ is the
number of binary trees with $n$ leaves and with exactly $k$ internal
nodes having a internal node as a left child, the fact that  the number
$\Schr(n)$ of Schröder trees with $n$ leaves expresses as
\begin{equation}
    \Schr(n) = \sum_{k = 0}^{n - 2} 2^k \, \Nar(n - 1, k),
\end{equation}
and the fact that any Schröder tree $\Sfr$ with $n$ leaves can be encoded
by a binary tree $\Tfr$ with $n$ leaves where any left oriented edge
connecting two internal nodes of $\Tfr$ is labeled on $[2]$ ($\Sfr$ is
obtained from $\Tfr$ by contracting all edges labeled by $2$).
\medskip

For instance, the first dimensions of  $\DAs_1$, $\DAs_2$, $\DAs_3$, and
$\DAs_4$ are respectively
\begin{equation}
    1, 1, 1, 1, 1, 1, 1, 1, 1, 1, 1,
\end{equation}
\begin{equation}
    1, 2, 6, 22, 90, 394, 1806, 8558, 41586, 206098, 1037718,
\end{equation}
\begin{equation}
    1, 3, 15, 93, 645, 4791, 37275, 299865, 2474025, 20819307, 178003815,
\end{equation}
\begin{equation}
    1, 4, 28, 244, 2380, 24868, 272188, 3080596, 35758828, 423373636, 5092965724.
\end{equation}
The second one is Sequence~\Sloane{A006318}, the third one is
Sequence~\Sloane{A103210}, and the last one is Sequence~\Sloane{A103211}
of~\cite{Slo}.
\medskip

Let us now establish a realization of $\DAs_\gamma$.
\medskip

\begin{Proposition} \label{prop:realisation_das_gamma}
    For any nonnegative integer $\gamma$, the operad $\DAs_\gamma$ is
    the vector space $\AlgLibre_{\DAs_\gamma}$ of $\gamma$-alternating
    Schröder trees. Moreover, for any $\gamma$-alternating Schröder
    trees $\Sfr$ and $\Tfr$, $\Sfr \circ_i \Tfr$ is the $\gamma$-alternating
    Schröder tree obtained by grafting the root of $\Tfr$ on the $i$th
    leaf $x$ of $\Sfr$ and then, if the father $y$ of $x$ and the root
    $z$ of $\Tfr$ have a same label, by contracting the edge connecting
    $y$ and $z$.
\end{Proposition}
\begin{proof}
    First, it is immediate that the vector space $\AlgLibre_{\DAs_\gamma}$
    endowed with the partial compositions described in the statement of
    the proposition is an operad.
    \smallskip

    Let
    \begin{equation}
        \phi : \DAs_\gamma \simeq \OpLibre\left(\GenDAs'\right)
        /_{\left\langle \RelDAs' \right\rangle} \to \AlgLibre_{\DAs_\gamma}
    \end{equation}
    be the map satisfying $\phi(\pi(\MDAs_a)) := \Cfr_a$ where $\Cfr_a$
    is the $\gamma$-alternating Schröder with two leaves and one internal
    node labeled by $a \in [\gamma]$ and
    $\pi : \OpLibre(\GenDAs') \to \DAs_\gamma$ is the canonical
    surjection map. Since we have
    $\phi(\pi(\MDAs_a)) \circ_1 \phi(\pi(\MDAs_a)) =
    \phi(\pi(\MDAs_a)) \circ_2 \phi(\pi(\MDAs_a))$
    for all $a \in [\gamma]$, $\phi$ extends in a unique way into an
    operad morphism. First, since the set $G_\gamma$ of all
    $\gamma$-alternating Schröder trees with two leaves and one internal
    node is a generating set of $\AlgLibre_{\DAs_\gamma}$ and the image
    of $\phi$ contains $G_\gamma$, $\phi$ is surjective. Second, since
    by definition of $\AlgLibre_{\DAs_\gamma}$, the bases of
    $\AlgLibre_{\DAs_\gamma}$ are indexed by $\gamma$-alternating
    Schröder trees, by Proposition~\ref{prop:serie_hilbert_das_gamma},
    $\AlgLibre_{\DAs_\gamma}$ and $\DAs_\gamma$ are isomorphic as graded
    vector spaces. Hence, $\phi$ is an operad isomorphism, showing that
    $\DAs_\gamma$ admits the claimed realization.
\end{proof}
\medskip

We have for instance in $\DAs_3$,
\begin{equation} \label{equ:exemple_composition_as_gamma_duale}
    \begin{split}
    \begin{tikzpicture}[xscale=.23,yscale=.13]
        \node[Feuille](0)at(0.00,-6.50){};
        \node[Feuille](10)at(8.00,-9.75){};
        \node[Feuille](12)at(10.00,-9.75){};
        \node[Feuille](2)at(2.00,-9.75){};
        \node[Feuille](4)at(4.00,-9.75){};
        \node[Feuille](6)at(5.00,-3.25){};
        \node[Feuille](7)at(6.00,-6.50){};
        \node[Feuille](9)at(7.00,-6.50){};
        \node[NoeudSchr](1)at(1.00,-3.25){\begin{math}1\end{math}};
        \node[NoeudSchr](11)at(9.00,-6.50){\begin{math}1\end{math}};
        \node[NoeudSchr](3)at(3.00,-6.50){\begin{math}2\end{math}};
        \node[NoeudSchr](5)at(5.00,0.00){\begin{math}2\end{math}};
        \node[NoeudSchr](8)at(7.00,-3.25){\begin{math}3\end{math}};
        \draw[Arete](0)--(1);
        \draw[Arete](1)--(5);
        \draw[Arete](10)--(11);
        \draw[Arete](11)--(8);
        \draw[Arete](12)--(11);
        \draw[Arete](2)--(3);
        \draw[Arete](3)--(1);
        \draw[Arete](4)--(3);
        \draw[Arete](6)--(5);
        \draw[Arete](7)--(8);
        \draw[Arete](8)--(5);
        \draw[Arete](9)--(8);
        \node(r)at(5.00,2.7){};
        \draw[Arete](r)--(5);
    \end{tikzpicture}
    \end{split}
    \circ_4
    \begin{split}
    \begin{tikzpicture}[xscale=.2,yscale=.17]
        \node[Feuille](0)at(0.00,-3.33){};
        \node[Feuille](2)at(2.00,-3.33){};
        \node[Feuille](4)at(4.00,-1.67){};
        \node[NoeudSchr](1)at(1.00,-1.67){\begin{math}2\end{math}};
        \node[NoeudSchr](3)at(3.00,0.00){\begin{math}3\end{math}};
        \draw[Arete](0)--(1);
        \draw[Arete](1)--(3);
        \draw[Arete](2)--(1);
        \draw[Arete](4)--(3);
        \node(r)at(3.00,2){};
        \draw[Arete](r)--(3);
    \end{tikzpicture}
    \end{split}
    =
    \begin{split}
    \begin{tikzpicture}[xscale=.2,yscale=.10]
        \node[Feuille](0)at(0.00,-8.50){};
        \node[Feuille](10)at(9.00,-8.50){};
        \node[Feuille](11)at(10.00,-8.50){};
        \node[Feuille](13)at(11.00,-8.50){};
        \node[Feuille](14)at(12.00,-12.75){};
        \node[Feuille](16)at(14.00,-12.75){};
        \node[Feuille](2)at(2.00,-12.75){};
        \node[Feuille](4)at(4.00,-12.75){};
        \node[Feuille](6)at(5.00,-12.75){};
        \node[Feuille](8)at(7.00,-12.75){};
        \node[NoeudSchr](1)at(1.00,-4.25){1};
        \node[NoeudSchr](12)at(11.00,-4.25){\begin{math}3\end{math}};
        \node[NoeudSchr](15)at(13.00,-8.50){\begin{math}1\end{math}};
        \node[NoeudSchr](3)at(3.00,-8.50){\begin{math}2\end{math}};
        \node[NoeudSchr](5)at(7.00,0.00){\begin{math}2\end{math}};
        \node[NoeudSchr](7)at(6.00,-8.50){\begin{math}2\end{math}};
        \node[NoeudSchr](9)at(8.00,-4.25){\begin{math}3\end{math}};
        \draw[Arete](0)--(1);
        \draw[Arete](1)--(5);
        \draw[Arete](10)--(9);
        \draw[Arete](11)--(12);
        \draw[Arete](12)--(5);
        \draw[Arete](13)--(12);
        \draw[Arete](14)--(15);
        \draw[Arete](15)--(12);
        \draw[Arete](16)--(15);
        \draw[Arete](2)--(3);
        \draw[Arete](3)--(1);
        \draw[Arete](4)--(3);
        \draw[Arete](6)--(7);
        \draw[Arete](7)--(9);
        \draw[Arete](8)--(7);
        \draw[Arete](9)--(5);
        \node(r)at(7.00,3.5){};
        \draw[Arete](r)--(5);
        \end{tikzpicture}
    \end{split}\,,
\end{equation}
and
\begin{equation}
    \begin{split}
    \begin{tikzpicture}[xscale=.23,yscale=.13]
        \node[Feuille](0)at(0.00,-6.50){};
        \node[Feuille](10)at(8.00,-9.75){};
        \node[Feuille](12)at(10.00,-9.75){};
        \node[Feuille](2)at(2.00,-9.75){};
        \node[Feuille](4)at(4.00,-9.75){};
        \node[Feuille](6)at(5.00,-3.25){};
        \node[Feuille](7)at(6.00,-6.50){};
        \node[Feuille](9)at(7.00,-6.50){};
        \node[NoeudSchr](1)at(1.00,-3.25){\begin{math}1\end{math}};
        \node[NoeudSchr](11)at(9.00,-6.50){\begin{math}1\end{math}};
        \node[NoeudSchr](3)at(3.00,-6.50){\begin{math}2\end{math}};
        \node[NoeudSchr](5)at(5.00,0.00){\begin{math}2\end{math}};
        \node[NoeudSchr](8)at(7.00,-3.25){\begin{math}3\end{math}};
        \draw[Arete](0)--(1);
        \draw[Arete](1)--(5);
        \draw[Arete](10)--(11);
        \draw[Arete](11)--(8);
        \draw[Arete](12)--(11);
        \draw[Arete](2)--(3);
        \draw[Arete](3)--(1);
        \draw[Arete](4)--(3);
        \draw[Arete](6)--(5);
        \draw[Arete](7)--(8);
        \draw[Arete](8)--(5);
        \draw[Arete](9)--(8);
        \node(r)at(5.00,2.7){};
        \draw[Arete](r)--(5);
    \end{tikzpicture}
    \end{split}
    \circ_5
    \begin{split}
    \begin{tikzpicture}[xscale=.2,yscale=.17]
        \node[Feuille](0)at(0.00,-3.33){};
        \node[Feuille](2)at(2.00,-3.33){};
        \node[Feuille](4)at(4.00,-1.67){};
        \node[NoeudSchr](1)at(1.00,-1.67){\begin{math}2\end{math}};
        \node[NoeudSchr](3)at(3.00,0.00){\begin{math}3\end{math}};
        \draw[Arete](0)--(1);
        \draw[Arete](1)--(3);
        \draw[Arete](2)--(1);
        \draw[Arete](4)--(3);
        \node(r)at(3.00,2){};
        \draw[Arete](r)--(3);
    \end{tikzpicture}
    \end{split}
    =
    \begin{split}
    \begin{tikzpicture}[xscale=.2,yscale=.10]
        \node[Feuille](0)at(0.00,-8.00){};
        \node[Feuille](10)at(9.00,-8.00){};
        \node[Feuille](12)at(11.00,-8.00){};
        \node[Feuille](13)at(12.00,-12.00){};
        \node[Feuille](15)at(14.00,-12.00){};
        \node[Feuille](2)at(2.00,-12.00){};
        \node[Feuille](4)at(4.00,-12.00){};
        \node[Feuille](6)at(5.00,-4.00){};
        \node[Feuille](7)at(6.00,-12.00){};
        \node[Feuille](9)at(8.00,-12.00){};
        \node[NoeudSchr](1)at(1.00,-4.00){1};
        \node[NoeudSchr](11)at(10.00,-4.00){3};
        \node[NoeudSchr](14)at(13.00,-8.00){1};
        \node[NoeudSchr](3)at(3.00,-8.00){2};
        \node[NoeudSchr](5)at(5.00,0.00){2};
        \node[NoeudSchr](8)at(7.00,-8.00){2};
        \draw[Arete](0)--(1);
        \draw[Arete](1)--(5);
        \draw[Arete](10)--(11);
        \draw[Arete](11)--(5);
        \draw[Arete](12)--(11);
        \draw[Arete](13)--(14);
        \draw[Arete](14)--(11);
        \draw[Arete](15)--(14);
        \draw[Arete](2)--(3);
        \draw[Arete](3)--(1);
        \draw[Arete](4)--(3);
        \draw[Arete](6)--(5);
        \draw[Arete](7)--(8);
        \draw[Arete](8)--(11);
        \draw[Arete](9)--(8);
        \node(r)at(5.00,3.5){};
        \draw[Arete](r)--(5);
    \end{tikzpicture}
    \end{split}\,.
\end{equation}
\medskip

%%%%%%%%%%%%%%%%%%%%%%%%%%%%%%%%%%%%%%%%%%%%%%%%%%%%%%%%%%%%%%%%%%%%%%%%
%%%%%%%%%%%%%%%%%%%%%%%%%%%%%%%%%%%%%%%%%%%%%%%%%%%%%%%%%%%%%%%%%%%%%%%%
\subsection{A diagram of operads}%
\label{subsec:diagramme_dias_as_dendr_gamma}
We now define morphisms between the operads $\Dias_\gamma$, $\As_\gamma$,
$\DAs_\gamma$, and $\Dendr_\gamma$ to obtain a generalization of a
classical diagram involving the diassociative, associative, and
dendriform operads.
\medskip

%%%%%%%%%%%%%%%%%%%%%%%%%%%%%%%%%%%%%%%%%%%%%%%%%%%%%%%%%%%%%%%%%%%%%%%%
\subsubsection{Relating the diassociative and dendriform operads}
The diagram
\begin{equation} \label{equ:diagramme_dendr_as_dias}
    \begin{split}
    \begin{tikzpicture}[yscale=.65]
        \node(Dendr)at(0,0){\begin{math} \Dendr \end{math}};
        \node(As)at(2,0){\begin{math} \As \end{math}};
        \node(Dias)at(4,0){\begin{math} \Dias \end{math}};
        \draw[->](Dias)
            edge node[anchor=south]{\begin{math} \eta \end{math}}(As);
        \draw[->](As)
            edge node[anchor=south]{\begin{math} \zeta \end{math}}(Dendr);
        \draw[<->,dotted,loop above,looseness=13](As)
            edge node[anchor=south]{\begin{math} ! \end{math}}(As);
        \draw[<->,dotted,loop above,looseness=1.5](Dendr)
            edge node[anchor=south]{\begin{math} ! \end{math}}(Dias);
    \end{tikzpicture}
    \end{split}
\end{equation}
is a well-known diagram of operads, being a part of the so-called
{\em operadic butterfly} \cite{Lod01,Lod06} and summarizing in a nice way
the links between the dendriform, associative, and diassociative operads.
The operad $\As$, being at the center of the diagram, is it own Koszul
dual, while $\Dias$ and $\Dendr$ are Koszul dual one of the other.
\medskip

The operad morphisms $\eta : \Dias \to \As$ and $\zeta : \As \to \Dendr$
are linearly defined through the realizations of $\Dias$ and $\Dendr$
recalled respectively in Section~1.3 of~\cite{GirI} and
in Section~\ref{subsec:dendr} by
\begin{equation}\begin{split}\end{split}
    \eta(\Efr_{2, 1}) :=
    \begin{split}
    \begin{tikzpicture}[xscale=.2,yscale=.15]
        \node[Feuille](0)at(0.00,-1.50){};
        \node[Feuille](2)at(2.00,-1.50){};
        \node[NoeudCor](1)at(1.00,0.00){};
        \draw[Arete](0)--(1);
        \draw[Arete](2)--(1);
        \node(r)at(1.00,2){};
        \draw[Arete](r)--(1);
    \end{tikzpicture}
    \end{split}
     =: \eta(\Efr_{2, 2})\,,
\end{equation}
and
\begin{equation}\begin{split}\end{split}
    \zeta\left(
    \begin{split}
    \begin{tikzpicture}[xscale=.2,yscale=.15]
        \node[Feuille](0)at(0.00,-1.50){};
        \node[Feuille](2)at(2.00,-1.50){};
        \node[NoeudCor](1)at(1.00,0.00){};
        \draw[Arete](0)--(1);
        \draw[Arete](2)--(1);
        \node(r)at(1.00,2){};
        \draw[Arete](r)--(1);
    \end{tikzpicture}
    \end{split}
    \right) :=
    \begin{split}
    \begin{tikzpicture}[xscale=.22,yscale=.17]
        \node[Feuille](0)at(0.00,-3.33){};
        \node[Feuille](2)at(2.00,-3.33){};
        \node[Feuille](4)at(4.00,-1.67){};
        \node[Noeud](1)at(1.00,-1.67){};
        \node[Noeud](3)at(3.00,0.00){};
        \draw[Arete](0)--(1);
        \draw[Arete](1)--(3);
        \draw[Arete](2)--(1);
        \draw[Arete](4)--(3);
        \node(r)at(3.00,1.67){};
        \draw[Arete](r)--(3);
    \end{tikzpicture}
    \end{split}
    +
    \begin{split}
    \begin{tikzpicture}[xscale=.22,yscale=.17]
        \node[Feuille](0)at(0.00,-1.67){};
        \node[Feuille](2)at(2.00,-3.33){};
        \node[Feuille](4)at(4.00,-3.33){};
        \node[Noeud](1)at(1.00,0.00){};
        \node[Noeud](3)at(3.00,-1.67){};
        \draw[Arete](0)--(1);
        \draw[Arete](2)--(3);
        \draw[Arete](3)--(1);
        \draw[Arete](4)--(3);
        \node(r)at(1.00,1.67){};
        \draw[Arete](r)--(1);
    \end{tikzpicture}
    \end{split}\,.
\end{equation}
Since $\Dias$ is generated by $\Efr_{2, 1}$ and $\Efr_{2, 2}$, and since
$\As$ is generated by
$\raisebox{-.4em}{\begin{tikzpicture}[xscale=.2,yscale=.15]
    \node[Feuille](0)at(0.00,-1.50){};
    \node[Feuille](2)at(2.00,-1.50){};
    \node[NoeudCor](1)at(1.00,0.00){};
    \draw[Arete](0)--(1);
    \draw[Arete](2)--(1);
    \node(r)at(1.00,2){};
    \draw[Arete](r)--(1);
\end{tikzpicture}}$,
$\eta$ and $\zeta$ are wholly defined.
\medskip

%%%%%%%%%%%%%%%%%%%%%%%%%%%%%%%%%%%%%%%%%%%%%%%%%%%%%%%%%%%%%%%%%%%%%%%%
\subsubsection{Relating the pluriassociative and polydendriform operads}
\begin{Proposition} \label{prop:morphisme_dias_gamma_vers_as_gamma}
    For any integer $\gamma \geq 0$, the map
    $\eta_\gamma : \Dias_\gamma \to \As_\gamma$ satisfying
    \begin{equation}
        \begin{split}\end{split}
        \eta_\gamma(0a) =
        \begin{split}
        \begin{tikzpicture}[xscale=.25,yscale=.2]
            \node[Feuille](0)at(0.00,-1.50){};
            \node[Feuille](2)at(2.00,-1.50){};
            \node[NoeudCor](1)at(1.00,0.00){\begin{math}a\end{math}};
            \draw[Arete](0)--(1);
            \draw[Arete](2)--(1);
            \node(r)at(1.00,1.7){};
            \draw[Arete](r)--(1);
        \end{tikzpicture}
        \end{split}
        = \eta_\gamma(a0),
        \qquad a \in [\gamma],
    \end{equation}
    extends in a unique way into an operad morphism. Moreover, this
    morphism is surjective.
\end{Proposition}
\begin{proof}
    Theorem~2.2.6 of~\cite{GirI} and
    Proposition~\ref{prop:realisation_das_gamma} allow to interpret
    the map $\eta_\gamma$ over the presentations of $\Dias_\gamma$
    and $\As_\gamma$. Then, via this interpretation, one has
    \begin{equation}
        \eta_\gamma(\pi(\GDias_a)) = \pi'(\MAs_a) = \eta_\gamma(\pi(\DDias_a)),
        \qquad a \in [\gamma],
    \end{equation}
    where $\pi : \OpLibre\left(\GenDias\right) \to \Dias_\gamma$ and
    $\pi' : \OpLibre\left(\GenAs\right) \to \As_\gamma$ are canonical
    surjection maps. Now, for any element $x$ of $\OpLibre\left(\GenDias\right)$
    generating the space of relations $\RelDias$ of $\Dias_\gamma$, we
    can check that $\eta_\gamma(\pi(x)) = 0$. This shows that $\eta_\gamma$
    extends in a unique way into an operad morphism. Finally, this
    morphism is a surjection since its image contains the set of all
    $\gamma$-corollas of arity $2$, which is a generating set of
    $\As_\gamma$.
\end{proof}
\medskip

By Proposition~\ref{prop:morphisme_dias_gamma_vers_as_gamma}, the map
$\eta_\gamma$, whose definition is only given in arity $2$, defines an
operad morphism. Nevertheless, by induction on the arity, one can prove
that for any word $x$ of $\Dias_\gamma$, $\eta_\gamma(x)$ is the
$\gamma$-corolla of arity $|x|$ labeled by the greatest letter of $x$.
\medskip

\begin{Proposition} \label{prop:morphisme_das_gamma_vers_dendr_gamma}
    For any integer $\gamma \geq 0$, the map
    $\zeta_\gamma : \DAs_\gamma \to \Dendr_\gamma$ satisfying
    \begin{equation}
        \begin{split}\end{split}
        \zeta_\gamma\left(
        \begin{split}
        \begin{tikzpicture}[xscale=.25,yscale=.2]
            \node[Feuille](0)at(0.00,-1.50){};
            \node[Feuille](2)at(2.00,-1.50){};
            \node[NoeudSchr](1)at(1.00,0.00){\begin{math}a\end{math}};
            \draw[Arete](0)--(1);
            \draw[Arete](2)--(1);
            \node(r)at(1.00,1.7){};
            \draw[Arete](r)--(1);
        \end{tikzpicture}
        \end{split}
        \right)
        = \ArbreBinGDeux{a} + \ArbreBinDDeux{a},
        \qquad a \in [\gamma],
    \end{equation}
    extends in a unique way into an operad morphism.
\end{Proposition}
\begin{proof}
    Propositions~\ref{prop:autre_presentation_das_gamma}
    and~\ref{prop:realisation_das_gamma}, and
    Theorem~\ref{thm:autre_presentation_dendr_gamma} allow to interpret
    the map $\zeta_\gamma$ over the presentations of $\DAs_\gamma$ and
    $\Dendr_\gamma$. Then, via this interpretation, one has
    \begin{equation}
        \zeta_\gamma(\pi(\MDAs_a)) = \pi'(\GDendr_a + \DDendr_a),
        \qquad a \in [\gamma],
    \end{equation}
    where $\pi : \OpLibre(\GenDAs') \to \DAs_\gamma$ and
    $\pi' : \OpLibre\left(\GenDendr'\right) \to \Dendr_\gamma$ are
    canonical surjection maps. We now observe that the image of
    $\pi(\MDAs_a)$ is $\OpAsDendr_a$, where $\OpAsDendr_a$ is the element
    of $\Dendr_\gamma$ defined in the statement of
    Proposition~\ref{prop:operateur_associatif_dendr_gamma}. Then, since
    by this last proposition this element is associative, for any element
    $x$ of $\OpLibre(\GenDAs')$ generating the space of relations of
    $\RelDAs'$ of $\DAs_\gamma$, $\zeta_\gamma(\pi(x)) = 0$. This shows
    that $\zeta_\gamma$ extends in a unique way into an operad morphism.
\end{proof}
\medskip

We have to observe that the morphism $\zeta_\gamma$ defined in the
statement of Proposition~\ref{prop:morphisme_das_gamma_vers_dendr_gamma}
is injective only for $\gamma \leq 1$. Indeed, when $\gamma \geq 2$,
we have the relation
\begin{equation}\begin{split}\end{split}
    \zeta_2\left(
    \begin{split}
    \begin{tikzpicture}[xscale=.2,yscale=.2]
        \node[Feuille](0)at(0.00,-5.25){};
        \node[Feuille](2)at(2.00,-5.25){};
        \node[Feuille](4)at(4.00,-3.50){};
        \node[Feuille](6)at(6.00,-1.75){};
        \node[NoeudSchr](1)at(1.00,-3.50){\begin{math}1\end{math}};
        \node[NoeudSchr](3)at(3.00,-1.75){\begin{math}2\end{math}};
        \node[NoeudSchr](5)at(5.00,0.00){\begin{math}1\end{math}};
        \draw[Arete](0)--(1);
        \draw[Arete](1)--(3);
        \draw[Arete](2)--(1);
        \draw[Arete](3)--(5);
        \draw[Arete](4)--(3);
        \draw[Arete](6)--(5);
        \node(r)at(5.00,1.75){};
        \draw[Arete](r)--(5);
    \end{tikzpicture}
    \end{split}\right)
    \begin{split} \; + \; \end{split}
    \zeta_2\left(
    \begin{split}
    \begin{tikzpicture}[xscale=.2,yscale=.2]
        \node[Feuille](0)at(0.00,-1.75){};
        \node[Feuille](2)at(2.00,-3.50){};
        \node[Feuille](4)at(4.00,-5.25){};
        \node[Feuille](6)at(6.00,-5.25){};
        \node[NoeudSchr](1)at(1.00,0.00){\begin{math}1\end{math}};
        \node[NoeudSchr](3)at(3.00,-1.75){\begin{math}2\end{math}};
        \node[NoeudSchr](5)at(5.00,-3.50){\begin{math}1\end{math}};
        \draw[Arete](0)--(1);
        \draw[Arete](2)--(3);
        \draw[Arete](3)--(1);
        \draw[Arete](4)--(5);
        \draw[Arete](5)--(3);
        \draw[Arete](6)--(5);
        \node(r)at(1.00,1.75){};
        \draw[Arete](r)--(1);
    \end{tikzpicture}
    \end{split}\right)
    \begin{split} \enspace = \enspace \end{split}
    \zeta_2\left(
    \begin{split}
    \begin{tikzpicture}[xscale=.25,yscale=.2]
        \node[Feuille](0)at(0.00,-2.00){};
        \node[Feuille](2)at(1.00,-4.00){};
        \node[Feuille](4)at(3.00,-4.00){};
        \node[Feuille](5)at(4.00,-2.00){};
        \node[NoeudSchr](1)at(2.00,0.00){\begin{math}1\end{math}};
        \node[NoeudSchr](3)at(2.00,-2.00){\begin{math}2\end{math}};
        \draw[Arete](0)--(1);
        \draw[Arete](2)--(3);
        \draw[Arete](3)--(1);
        \draw[Arete](4)--(3);
        \draw[Arete](5)--(1);
        \node(r)at(2.00,2.00){};
        \draw[Arete](r)--(1);
    \end{tikzpicture}
    \end{split}\right)
    \begin{split} \; + \; \end{split}
    \zeta_2\left(
    \begin{split}
    \begin{tikzpicture}[xscale=.22,yscale=.17]
        \node[Feuille](0)at(0.00,-4.67){};
        \node[Feuille](2)at(2.00,-4.67){};
        \node[Feuille](4)at(4.00,-4.67){};
        \node[Feuille](6)at(6.00,-4.67){};
        \node[NoeudSchr](1)at(1.00,-2.33){\begin{math}1\end{math}};
        \node[NoeudSchr](3)at(3.00,0.00){\begin{math}2\end{math}};
        \node[NoeudSchr](5)at(5.00,-2.33){\begin{math}1\end{math}};
        \draw[Arete](0)--(1);
        \draw[Arete](1)--(3);
        \draw[Arete](2)--(1);
        \draw[Arete](4)--(5);
        \draw[Arete](5)--(3);
        \draw[Arete](6)--(5);
        \node(r)at(3.00,2.33){};
        \draw[Arete](r)--(3);
    \end{tikzpicture}
    \end{split}\right)\,.
\end{equation}
\medskip

\begin{Theoreme} \label{thm:diagramme_dias_as_das_dendr_gamma}
    For any integer $\gamma \geq 0$, the operads $\Dias_\gamma$,
    $\Dendr_\gamma$, $\As_\gamma$, and $\DAs_\gamma$ fit into the
    diagram
    \begin{equation} \label{equ:diagramme_dendr_gamma_as_gamma_dias_gamma}
        \begin{split}
        \begin{tikzpicture}[yscale=.6]
            \node(Dendr)at(0,0){\begin{math} \Dendr_\gamma \end{math}};
            \node(DAs)at(2,0){\begin{math} \DAs_\gamma \end{math}};
            \node(As)at(4,0){\begin{math} \As_\gamma \end{math}};
            \node(Dias)at(6,0){\begin{math} \Dias_\gamma \end{math}};
            \draw[->](Dias)
                edge node[anchor=south]{\begin{math} \eta_\gamma \end{math}}(As);
            \draw[<->,dotted](As)
                edge node[anchor=south]{\begin{math} ! \end{math}}(DAs);
            \draw[->](DAs)
                edge node[anchor=south]{\begin{math} \zeta_\gamma \end{math}}(Dendr);
            \draw[<->,dotted,loop above,looseness=.9](Dendr)
                edge node[anchor=south]{\begin{math} ! \end{math}}(Dias);
        \end{tikzpicture}
        \end{split}\,,
    \end{equation}
    where $\eta_\gamma$ is the surjection defined in the statement
    of Proposition~\ref{prop:morphisme_dias_gamma_vers_as_gamma}
    and $\zeta_\gamma$ is the operad morphism defined in the statement
    of Proposition~\ref{prop:morphisme_das_gamma_vers_dendr_gamma}.
\end{Theoreme}
\begin{proof}
    This is a direct consequence of
    Propositions~\ref{prop:morphisme_dias_gamma_vers_as_gamma}
    and~\ref{prop:morphisme_das_gamma_vers_dendr_gamma}.
\end{proof}
\medskip

Diagram~\eqref{equ:diagramme_dendr_gamma_as_gamma_dias_gamma} is a
generalization of \eqref{equ:diagramme_dendr_as_dias} in which the
associative operad split into operads $\As_\gamma$ and $\DAs_\gamma$.
\medskip

%%%%%%%%%%%%%%%%%%%%%%%%%%%%%%%%%%%%%%%%%%%%%%%%%%%%%%%%%%%%%%%%%%%%%%%%
%%%%%%%%%%%%%%%%%%%%%%%%%%%%%%%%%%%%%%%%%%%%%%%%%%%%%%%%%%%%%%%%%%%%%%%%
%%%%%%%%%%%%%%%%%%%%%%%%%%%%%%%%%%%%%%%%%%%%%%%%%%%%%%%%%%%%%%%%%%%%%%%%
\section{Further generalizations}%
\label{sec:generalisation_supplementaires}
In this last section, we propose some generalizations on a nonnegative
integer parameter of well-known operads. For this, we use similar tools
as the ones used in the first sections of this paper.
\medskip

%%%%%%%%%%%%%%%%%%%%%%%%%%%%%%%%%%%%%%%%%%%%%%%%%%%%%%%%%%%%%%%%%%%%%%%%
%%%%%%%%%%%%%%%%%%%%%%%%%%%%%%%%%%%%%%%%%%%%%%%%%%%%%%%%%%%%%%%%%%%%%%%%
\subsection{Duplicial operad}
We construct here a generalization on a nonnegative integer parameter of
the duplicial operad and describe the free algebras over one generator
in the category encoded by this generalization.
\medskip

%%%%%%%%%%%%%%%%%%%%%%%%%%%%%%%%%%%%%%%%%%%%%%%%%%%%%%%%%%%%%%%%%%%%%%%%
\subsubsection{Multiplicial operads}%
\label{subsub:dup_gamma}
It is well-known~\cite{LV12} that the dendriform operad and the duplicial
operad $\Dup$~\cite{Lod08} are both specializations of a same operad
$\DupDendr_q$ with one parameter $q \in \K$. This operad  admits the
presentation
$\left(\GenLibre_{\DupDendr_q}, \RelLibre_{\DupDendr_q}\right)$, where
$\GenLibre_{\DupDendr_q} := \GenLibre_\Dendr$ and
$\RelLibre_{\DupDendr_q}$ is the vector space generated by
\begin{subequations}
\begin{equation}
    \GDendr \circ_1 \DDendr - \DDendr \circ_2 \GDendr,
\end{equation}
\begin{equation}
    \GDendr \circ_1 \GDendr - \GDendr \circ_2 \GDendr
    - q \GDendr \circ_2 \DDendr,
\end{equation}
\begin{equation}
    q \DDendr \circ_1 \GDendr + \DDendr \circ_1 \DDendr
     - \DDendr \circ_2 \DDendr.
\end{equation}
\end{subequations}
One can observe that $\DupDendr_1$ is the dendriform operad and that
$\DupDendr_0$ is the duplicial operad.
\medskip

On the basis of this observation, from the presentation of $\Dendr_\gamma$
provided by Theorem~\ref{thm:autre_presentation_dendr_gamma} and its
concise form provided by Relations~\eqref{equ:relation_dendr_gamma_1_concise},
\eqref{equ:relation_dendr_gamma_2_concise},
and~\eqref{equ:relation_dendr_gamma_3_concise} for its space of relations,
we define the operad $\DupDendr_{q, \gamma}$ with two parameters, an
integer $\gamma \geq 0$ and $q \in \K$, in the following way. We set
$\DupDendr_{q, \gamma}$ as the operad admitting the presentation
$\left(\GenLibre_{\DupDendr_{q, \gamma}},
\RelLibre_{\DupDendr_{q, \gamma}}\right)$,
where $\GenLibre_{\DupDendr_{q, \gamma}} := \GenDendr'$ and
$\RelLibre_{\DupDendr_{q, \gamma}}$ is the vector space generated by
\begin{subequations}
\begin{equation} \label{equ:dupdendr_gamma_1}
    \GDendr_a \circ_1 \DDendr_{a'} - \DDendr_{a'} \circ_2 \GDendr_a,
    \qquad a, a' \in [\gamma],
\end{equation}
\begin{equation} \label{equ:dupdendr_gamma_2}
    \GDendr_a \circ_1 \GDendr_{a'}
    - \GDendr_{a \Min a'} \circ_2 \GDendr_a
    - q\GDendr_{a \Min a'} \circ_2 \DDendr_{a'},
    \qquad a, a' \in [\gamma],
\end{equation}
\begin{equation} \label{equ:dupdendr_gamma_3}
    q\DDendr_{a \Min a'} \circ_1 \GDendr_{a'}
    + \DDendr_{a \Min a'} \circ_1 \DDendr_a
    - \DDendr_a \circ_2 \DDendr_{a'},
    \qquad a, a' \in [\gamma].
\end{equation}
\end{subequations}
One can observe that $\DupDendr_{1, \gamma}$ is the operad $\Dendr_\gamma$.
\medskip

Let us define the operad $\Dup_\gamma$, called
{\em $\gamma$-multiplicial operad}, as the operad $\DupDendr_{0, \gamma}$.
By using respectively the symbols $\GDup_a$ and $\DDup_a$ instead of
$\GDendr_a$ and $\DDendr_a$ for all $a \in [\gamma]$, we obtain that the
space of relations $\RelDup$ of $\Dup_\gamma$ is generated by
\begin{subequations}
\begin{equation} \label{equ:relation_dup_gamma_1}
    \GDup_a \circ_1 \DDup_{a'} - \DDup_{a'} \circ_2 \GDup_a,
    \qquad a, a' \in [\gamma],
\end{equation}
\begin{equation} \label{equ:relation_dup_gamma_2}
    \GDup_a \circ_1 \GDup_{a'} - \GDup_{a \Min a'} \circ_2 \GDup_a,
    \qquad a, a' \in [\gamma],
\end{equation}
\begin{equation} \label{equ:relation_dup_gamma_3}
    \DDup_{a \Min a'} \circ_1 \DDup_a - \DDup_a \circ_2 \DDup_{a'},
    \qquad a, a' \in [\gamma].
\end{equation}
\end{subequations}
We denote by $\GenDup$ the set of generators
$\{\GDup_a, \DDup_a : a \in [\gamma] \}$ of $\Dup_\gamma$.
\medskip

In order to establish some properties of $\Dup_\gamma$, let us consider
the quadratic rewrite rule $\Recr_\gamma$ on $\OpLibre(\GenDup)$
satisfying
\begin{subequations}
\begin{equation}
    \GDup_a \circ_1 \DDup_{a'}
    \enspace \Recr_\gamma \enspace
    \DDup_{a'} \circ_2 \GDup_a,
    \qquad a, a' \in [\gamma],
\end{equation}
\begin{equation}
    \GDup_a \circ_1 \GDup_{a'}
    \enspace \Recr_\gamma \enspace
    \GDup_{a \Min a'} \circ_2 \GDup_a,
    \qquad a, a' \in [\gamma],
\end{equation}
\begin{equation}
    \DDup_a \circ_2 \DDup_{a'}
    \enspace \Recr_\gamma \enspace
    \DDup_{a \Min a'} \circ_1 \DDup_a,
    \qquad a, a' \in [\gamma].
\end{equation}
\end{subequations}
Observe that the space induced by the operad congruence induced by
$\Recr_\gamma$ is $\RelDup$.
\medskip

\begin{Lemme} \label{lem:dup_gamma_reecriture}
    For any integer $\gamma \geq 0$, the rewrite rule $\Recr_\gamma$ is
    convergent and the generating series $\Gca_\gamma(t)$ of its normal
    forms counted by arity satisfies
    \begin{equation} \label{equ:dup_gamma_serie_gen_formes_normales}
        \Gca_\gamma(t) =
        t + 2\gamma t \, \Gca_\gamma(t) + \gamma^2 t \, \Gca_\gamma(t)^2.
    \end{equation}
\end{Lemme}
\begin{proof}
    Let us first prove that $\Recr_\gamma$ is terminating. Consider the
    map $\phi : \OpLibre(\GenDup) \to \EnsNat^2$ defined, for any syntax
    tree $\Tfr$ by $\phi(\Tfr) := (\alpha + \alpha', \beta)$, where
    $\alpha$ (resp. $\alpha'$, $\beta$) is the sum, for all internal
    nodes of $\Tfr$ labeled by $\GDup_a$ (resp. $\DDup_a$, $\DDup_a$),
    $a \in [\gamma]$, of the number of internal nodes in its right
    (resp. left, right) subtree. For the lexicographical order $\leq$ on
    $\EnsNat^2$, we can check that for all $\Recr_\gamma$-rewritings
    $\Sfr \Recr_\gamma \Tfr$ where $\Sfr$ and $\Tfr$ are syntax trees
    with two internal nodes, we have $\phi(\Sfr) \ne \phi(\Tfr)$ and
    $\phi(\Sfr) \leq \phi(\Tfr)$. This implies that any syntax tree
    $\Tfr$ obtained by a sequence of $\Recr_\gamma$-rewritings from a
    syntax tree $\Sfr$ satisfies $\phi(\Sfr) \ne \phi(\Tfr)$ and
    $\phi(\Sfr) \leq \phi(\Tfr)$. Then, since the set of syntax trees of
    $\OpLibre(\GenDup)$ of a fixed arity is finite, this shows that
    $\Recr_\gamma$ is a terminating rewrite rule.
    \smallskip

    Let us now prove that $\Recr_\gamma$ is convergent. We call
    {\em critical tree} any syntax tree $\Sfr$ with three internal nodes
    that can be rewritten by $\Recr_\gamma$ into two different trees $\Tfr$
    and $\Tfr'$. The pair $(\Tfr, \Tfr')$ is a {\em critical pair} for
    $\Recr_\gamma$. Critical trees for $\Recr_\gamma$ are, for all
    $a, b, c \in [\gamma]$,
    \begin{equation} \label{equ:dup_gamma_arbres_critiques}
        \begin{split}
        \begin{tikzpicture}[xscale=.4,yscale=.3]
            \node(0)at(0.00,-3.50){};
            \node(2)at(2.00,-5.25){};
            \node(4)at(4.00,-5.25){};
            \node(6)at(6.00,-1.75){};
            \node(1)at(1.00,-1.75){\begin{math}\DDup_a\end{math}};
            \node(3)at(3.00,-3.50){\begin{math}\DDup_b\end{math}};
            \node(5)at(5.00,0.00){\begin{math}\GDup_c\end{math}};
            \draw(0)--(1);\draw(1)--(5);\draw(2)--(3);\draw(3)--(1);
            \draw(4)--(3);\draw(6)--(5);
            \node(r)at(5.00,1.75){};
            \draw(r)--(5);
        \end{tikzpicture}
        \end{split}\,,
        \quad
        \begin{split}
        \begin{tikzpicture}[xscale=.4,yscale=.3]
            \node(0)at(0.00,-5.25){};
            \node(2)at(2.00,-5.25){};
            \node(4)at(4.00,-3.50){};
            \node(6)at(6.00,-1.75){};
            \node(1)at(1.00,-3.50){\begin{math}\DDup_a\end{math}};
            \node(3)at(3.00,-1.75){\begin{math}\GDup_b\end{math}};
            \node(5)at(5.00,0.00){\begin{math}\GDup_c\end{math}};
            \draw(0)--(1);\draw(1)--(3);\draw(2)--(1);\draw(3)--(5);
            \draw(4)--(3);\draw(6)--(5);
            \node(r)at(5.00,1.75){};
            \draw(r)--(5);
        \end{tikzpicture}
        \end{split}\,,
        \quad
        \begin{split}
        \begin{tikzpicture}[xscale=.4,yscale=.3]
            \node(0)at(0.00,-5.25){};
            \node(2)at(2.00,-5.25){};
            \node(4)at(4.00,-3.50){};
            \node(6)at(6.00,-1.75){};
            \node(1)at(1.00,-3.50){\begin{math}\GDup_a\end{math}};
            \node(3)at(3.00,-1.75){\begin{math}\GDup_b\end{math}};
            \node(5)at(5.00,0.00){\begin{math}\GDup_c\end{math}};
            \draw(0)--(1);\draw(1)--(3);\draw(2)--(1);\draw(3)--(5);
            \draw(4)--(3);\draw(6)--(5);
            \node(r)at(5.00,1.75){};
            \draw(r)--(5);
        \end{tikzpicture}
        \end{split}\,,
        \quad
        \begin{split}
        \begin{tikzpicture}[xscale=.4,yscale=.3]
            \node(0)at(0.00,-1.75){};
            \node(2)at(2.00,-3.50){};
            \node(4)at(4.00,-5.25){};
            \node(6)at(6.00,-5.25){};
            \node(1)at(1.00,0.00){\begin{math}\DDup_a\end{math}};
            \node(3)at(3.00,-1.75){\begin{math}\DDup_b\end{math}};
            \node(5)at(5.00,-3.50){\begin{math}\DDup_c\end{math}};
            \draw(0)--(1);\draw(2)--(3);\draw(3)--(1);\draw(4)--(5);
            \draw(5)--(3);\draw(6)--(5);
            \node(r)at(1.00,1.75){};
            \draw(r)--(1);
        \end{tikzpicture}
        \end{split}\,.
    \end{equation}
    Since $\Recr_\gamma$ is terminating, by the diamond lemma~\cite{New42}
    (see also~\cite{BN98}), to prove that $\Recr_\gamma$ is confluent,
    it is enough to check that for any critical tree $\Sfr$, there is a
    normal form $\Rfr$ of $\Recr_\gamma$ such that
    $\Sfr \Recr_\gamma \Tfr \overset{*}{\Recr_\gamma} \Rfr$ and
    $\Sfr \Recr_\gamma \Tfr' \overset{*}{\Recr_\gamma} \Rfr$, where
    $(\Tfr, \Tfr')$ is a critical pair. This can be done by hand for each
    of the critical trees depicted in~\eqref{equ:dup_gamma_arbres_critiques}.
    \smallskip

    Let us finally prove that the generating series of the normal forms
    of $\Recr_\gamma$ is~\eqref{equ:dup_gamma_serie_gen_formes_normales}.
    Since $\Recr_\gamma$ is terminating, its normal forms are the
    syntax trees that have no partial subtree equal to
    $\GDup_a \circ_1 \DDup_{a'}$, $\GDup_a \circ_1 \GDup_{a'}$, or
    $\DDup_a \circ_2 \DDup_{a'}$ for all $a, a' \in [\gamma]$. Then,
    the normal forms of $\Recr_\gamma$ are the syntax trees wherein any
    internal node labeled by $\GDup_a$, $a \in [\gamma]$, has a leaf as
    left child and any internal node labeled by $\DDup_a$, $a \in [\gamma]$,
    has a leaf or an internal node labeled by $\GDup_{a'}$,
    $a' \in [\gamma]$, as right child. Therefore, by denoting by
    $\Gca'_\gamma(t)$ the generating series of the normal forms of
    $\Recr_\gamma$ equal to the leaf or with a root labeled by $\GDup_a$,
    $a \in [\gamma]$, we obtain
    \begin{equation}
        \Gca'_\gamma(t) = t + \gamma t \, \Gca_\gamma(t)
    \end{equation}
    and
    \begin{equation}
        \Gca_\gamma(t) =
        \Gca'_\gamma(t) + \gamma \, \Gca_\gamma(t) \Gca'_\gamma(t).
    \end{equation}
    An elementary computation shows that $\Gca(t)$
    satisfies~\eqref{equ:dup_gamma_serie_gen_formes_normales}.
\end{proof}
\medskip

\begin{Proposition} \label{prop:proprietes_dup_gamma}
    For any integer $\gamma \geq 0$, the operad $\Dup_\gamma$ is Koszul
    and for any integer $n \geq 1$, $\Dup_\gamma(n)$ is the vector space
    of $\gamma$-edge valued binary trees with $n$ internal nodes.
\end{Proposition}
\begin{proof}
    Since the space induced by the operad congruence induced by
    $\Recr_\gamma$ is $\RelDup$, and since by
    Lemma~\ref{lem:dup_gamma_reecriture}, $\Recr_\gamma$ is convergent,
    by the Koszulity criterion~\cite{Hof10,DK10,LV12} we have reformulated
    in Section~1.2.5 of~\cite{GirI}, $\Dup_\gamma$
    is a Koszul operad. Moreover, again because $\Recr_\gamma$ is
    convergent, as a vector space, $\Dup_\gamma(n)$ is isomorphic to the
    vector space of the normal forms of $\Recr_\gamma$ with $n \geq 1$
    internal nodes. Since the generating series $\Gca_\gamma(t)$ of the
    normal forms of $\Recr_\gamma$ is also the generating series of
    $\gamma$-edge valued binary trees (see
    Proposition~\ref{prop:serie_hilbert_dendr_gamma}), the second part
    of the statement of the proposition follows.
\end{proof}
\medskip

Since Proposition~\ref{prop:proprietes_dup_gamma} shows that the operads
$\Dup_\gamma$ and $\Dendr_\gamma$ have the same underlying vector space,
asking if these two operads are isomorphic is natural. Next result implies
that this is not the case.
\medskip

\begin{Proposition}%
\label{prop:description_operateurs_associatifs_dup_gamma}
    For any integer $\gamma \geq 0$, any associative element of
    $\Dup_\gamma$ is proportional to $\pi(\GDup_a)$ or $\pi(\DDup_a)$
    for an $a \in [\gamma]$, where
    $\pi : \OpLibre\left(\GenDup\right) \to \Dup_\gamma$ is the canonical
    surjection map.
\end{Proposition}
\begin{proof}
    Let $\pi : \OpLibre\left(\GenDup\right) \to \Dup_\gamma$ be the
    canonical surjection map. Consider the element
    \begin{equation}
        x := \sum_{a \in [\gamma]} \alpha_a \GDup_a + \beta_a \DDup_a
    \end{equation}
    of $\OpLibre\left(\GenDup\right)$, where $\alpha_a, \beta_a \in \K$
    for all $a \in [\gamma]$, such that $\pi(x)$ is associative in
    $\Dup_\gamma$. Since we have $\pi(r) = 0$ for all elements $r$
    of $\RelLibre_{\Dup_\gamma}$ (see~\eqref{equ:relation_dup_gamma_1},
    \eqref{equ:relation_dup_gamma_2}, and~\eqref{equ:relation_dup_gamma_3}),
    the fact that $\pi(x \circ_1 x - x \circ_2 x) = 0$ implies the
    constraints
    \begin{equation}\begin{split}
        \alpha_a \, \beta_{a'} - \beta_{a'} \, \alpha_a & = 0,
        \qquad a, a' \in [\gamma], \\
        \alpha_a \, \alpha_{a'} - \alpha_{a \Min a'} \, \alpha_a & = 0,
        \qquad a, a' \in [\gamma], \\
        \beta_a \, \beta_{a'} - \beta_{a \Min a'} \, \beta_a & = 0,
        \qquad a, a' \in [\gamma],
    \end{split}\end{equation}
    on the coefficients intervening in $x$. Moreover, since the syntax
    trees $\DDup_b \circ_1 \DDup_a$, $\DDup_a \circ_1 \GDup_{a'}$,
    $\GDup_b \circ_2 \GDup_a$, and $\GDup_a \circ_2 \DDup_{a'}$ do not
    appear in $\RelLibre_{\Dup_\gamma}$ for all $a < b \in [\gamma]$ and
    $a, a' \in [\gamma]$, we have the further constraints
    \begin{equation}\begin{split}
        \beta_b \, \beta_a & = 0, \qquad a < b \in [\gamma], \\
        \beta_a \, \alpha_{a'} & = 0, \qquad a, a' \in [\gamma], \\
        \alpha_b \, \alpha_a & = 0, \qquad a < b \in [\gamma], \\
        \alpha_a \, \beta_{a'} & = 0, \qquad a, a' \in [\gamma].
    \end{split}\end{equation}
    These relations imply that there are at most one $c \in [\gamma]$ and
    one $d \in [\gamma]$ such that $\alpha_c \ne 0$ and $\beta_d \ne 0$.
    In this case, the relations imply also that $\alpha_c = 0$ or
    $\beta_d = 0$, or both. Therefore, $x$ is of the form
    $x = \alpha_a \GDup_a$ or $x = \beta_a \DDup_a$ for an $a \in [\gamma]$,
    whence the statement of the proposition.
\end{proof}
\medskip

By Proposition~\ref{prop:description_operateurs_associatifs_dup_gamma}
there are exactly $2\gamma$ nonproportional associative operations in
$\Dup_\gamma$ while, by
Proposition~\ref{prop:description_operateurs_associatifs_dendr_gamma}
there are exactly $\gamma$ such operations in $\Dendr_\gamma$.
Therefore, $\Dup_\gamma$ and $\Dendr_\gamma$ are not isomorphic.
\medskip

%%%%%%%%%%%%%%%%%%%%%%%%%%%%%%%%%%%%%%%%%%%%%%%%%%%%%%%%%%%%%%%%%%%%%%%%
\subsubsection{Free multiplicial algebras}
We call {\em $\gamma$-multiplicial algebra} any $\Dup_\gamma$-algebra.
From the definition of $\Dup_\gamma$, any $\gamma$-multiplicial algebra
is a vector space endowed with linear operations $\GDup_a, \DDup_a$,
$a \in [\gamma]$, satisfying the relations encoded
by~\eqref{equ:relation_dup_gamma_1}---\eqref{equ:relation_dup_gamma_3}.
\medskip

In order the simplify and make uniform next definitions, we consider
that in any $\gamma$-edge valued binary tree $\Tfr$, all edges
connecting internal nodes of $\Tfr$ with leaves are labeled by $\infty$.
By convention, for all $a \in [\gamma]$, we have
$a \Min \infty = a = \infty \Min a$. Let us endow the vector space
$\AlgLibre_{\Dup_\gamma}$ of $\gamma$-edge valued binary trees with
linear operations
\begin{equation}
    \GDup_a, \DDup_a :
    \AlgLibre_{\Dup_\gamma} \otimes \AlgLibre_{\Dup_\gamma}
    \to \AlgLibre_{\Dup_\gamma},
    \qquad a \in [\gamma],
\end{equation}
recursively defined, for any $\gamma$-edge valued binary tree $\Sfr$
and any $\gamma$-edge valued binary trees or leaves $\Tfr_1$ and
$\Tfr_2$ by
\begin{equation}
    \Sfr \GDup_a \Feuille
    := \Sfr =:
    \Feuille \DDup_a \Sfr,
\end{equation}
\begin{equation}
    \Feuille \GDup_a \Sfr := 0 =: \Sfr \DDup_a \Feuille,
\end{equation}
\begin{equation}
    \ArbreBinValue{x}{y}{\Tfr_1}{\Tfr_2}
    \GDup_a \Sfr :=
    \ArbreBinValue{x}{z}{\Tfr_1}{\Tfr_2 \GDup_a \Sfr}\,,
    \qquad
    z := a \Min y,
\end{equation}
\begin{equation}
    \begin{split}\Sfr \DDup_a\end{split}
    \ArbreBinValue{x}{y}{\Tfr_1}{\Tfr_2}
    :=
    \ArbreBinValue{z}{y}{\Sfr \DDup_a \Tfr_1}{\Tfr_2}\,,
    \qquad
    z := a \Min x.
\end{equation}
Note that neither $\Feuille \GDendr_a \Feuille$ nor
$\Feuille \DDup_a \Feuille$ are defined.
\medskip

These recursive definitions for the operations $\GDup_a$, $\DDup_a$,
$a \in [\gamma]$, lead to the following direct reformulations. If $\Sfr$
and $\Tfr$ are two $\gamma$-edge valued binary trees, $\Tfr \GDup_a \Sfr$
(resp. $\Sfr \DDup_a \Tfr$) is obtained by replacing each label $y$
(resp. $x$) of any edge in the rightmost (resp. leftmost) path of $\Tfr$
by $a \Min y$ (resp. $a \Min x$) to obtain a tree $\Tfr'$, and by
grafting the root of $\Sfr$ on the rightmost (resp. leftmost) leaf of
$\Tfr'$. These two operations are respective generalizations of the
operations {\em under} and {\em over} on binary trees introduced by
Loday and Ronco~\cite{LR02}.
\medskip

For example, we have
\begin{equation}
    \begin{split}
    \begin{tikzpicture}[xscale=.25,yscale=.2]
        \node[Feuille](0)at(0.00,-4.50){};
        \node[Feuille](2)at(2.00,-4.50){};
        \node[Feuille](4)at(4.00,-4.50){};
        \node[Feuille](6)at(6.00,-6.75){};
        \node[Feuille](8)at(8.00,-6.75){};
        \node[Noeud](1)at(1.00,-2.25){};
        \node[Noeud](3)at(3.00,0.00){};
        \node[Noeud](5)at(5.00,-2.25){};
        \node[Noeud](7)at(7.00,-4.50){};
        \draw[Arete](0)--(1);
        \draw[Arete](1)edge[]node[EtiqArete]{\begin{math}1\end{math}}(3);
        \draw[Arete](2)--(1);
        \draw[Arete](4)--(5);
        \draw[Arete](5)edge[]node[EtiqArete]{\begin{math}3\end{math}}(3);
        \draw[Arete](6)--(7);
        \draw[Arete](7)edge[]node[EtiqArete]{\begin{math}1\end{math}}(5);
        \draw[Arete](8)--(7);
        \node(r)at(3.00,2){};
        \draw[Arete](r)--(3);
    \end{tikzpicture}
    \end{split}
    \GDup_2
    \begin{split}
    \begin{tikzpicture}[xscale=.25,yscale=.2]
        \node[Feuille](0)at(0.00,-4.67){};
        \node[Feuille](2)at(2.00,-4.67){};
        \node[Feuille](4)at(4.00,-4.67){};
        \node[Feuille](6)at(6.00,-4.67){};
        \node[Noeud](1)at(1.00,-2.33){};
        \node[Noeud](3)at(3.00,0.00){};
        \node[Noeud](5)at(5.00,-2.33){};
        \draw[Arete](0)--(1);
        \draw[Arete](1)edge[]node[EtiqArete]{\begin{math}1\end{math}}(3);
        \draw[Arete](2)--(1);
        \draw[Arete](4)--(5);
        \draw[Arete](5)edge[]node[EtiqArete]{\begin{math}2\end{math}}(3);
        \draw[Arete](6)--(5);
        \node(r)at(3.00,2){};
        \draw[Arete](r)--(3);
    \end{tikzpicture}
    \end{split}
    =
    \begin{split}
    \begin{tikzpicture}[xscale=.22,yscale=.15]
        \node[Feuille](0)at(0.00,-5.00){};
        \node[Feuille](10)at(10.00,-12.50){};
        \node[Feuille](12)at(12.00,-12.50){};
        \node[Feuille](14)at(14.00,-12.50){};
        \node[Feuille](2)at(2.00,-5.00){};
        \node[Feuille](4)at(4.00,-5.00){};
        \node[Feuille](6)at(6.00,-7.50){};
        \node[Feuille](8)at(8.00,-12.50){};
        \node[Noeud](1)at(1.00,-2.50){};
        \node[Noeud](11)at(11.00,-7.50){};
        \node[Noeud](13)at(13.00,-10.00){};
        \node[Noeud](3)at(3.00,0.00){};
        \node[Noeud](5)at(5.00,-2.50){};
        \node[Noeud](7)at(7.00,-5.00){};
        \node[Noeud](9)at(9.00,-10.00){};
        \draw[Arete](0)--(1);
        \draw[Arete](1)edge[]node[EtiqArete]{\begin{math}1\end{math}}(3);
        \draw[Arete](10)--(9);
        \draw[Arete](11)edge[]node[EtiqArete]{\begin{math}2\end{math}}(7);
        \draw[Arete](12)--(13);
        \draw[Arete](13)edge[]node[EtiqArete]{\begin{math}2\end{math}}(11);
        \draw[Arete](14)--(13);
        \draw[Arete](2)--(1);
        \draw[Arete](4)--(5);
        \draw[Arete](5)edge[]node[EtiqArete]{\begin{math}2\end{math}}(3);
        \draw[Arete](6)--(7);
        \draw[Arete](7)edge[]node[EtiqArete]{\begin{math}1\end{math}}(5);
        \draw[Arete](8)--(9);
        \draw[Arete](9)edge[]node[EtiqArete]{\begin{math}1\end{math}}(11);
        \node(r)at(3.00,2.50){};
        \draw[Arete](r)--(3);
    \end{tikzpicture}
    \end{split}\,,
\end{equation}
and
\begin{equation}
    \begin{split}
    \begin{tikzpicture}[xscale=.25,yscale=.2]
        \node[Feuille](0)at(0.00,-4.50){};
        \node[Feuille](2)at(2.00,-4.50){};
        \node[Feuille](4)at(4.00,-4.50){};
        \node[Feuille](6)at(6.00,-6.75){};
        \node[Feuille](8)at(8.00,-6.75){};
        \node[Noeud](1)at(1.00,-2.25){};
        \node[Noeud](3)at(3.00,0.00){};
        \node[Noeud](5)at(5.00,-2.25){};
        \node[Noeud](7)at(7.00,-4.50){};
        \draw[Arete](0)--(1);
        \draw[Arete](1)edge[]node[EtiqArete]{\begin{math}1\end{math}}(3);
        \draw[Arete](2)--(1);
        \draw[Arete](4)--(5);
        \draw[Arete](5)edge[]node[EtiqArete]{\begin{math}3\end{math}}(3);
        \draw[Arete](6)--(7);
        \draw[Arete](7)edge[]node[EtiqArete]{\begin{math}1\end{math}}(5);
        \draw[Arete](8)--(7);
        \node(r)at(3.00,2){};
        \draw[Arete](r)--(3);
    \end{tikzpicture}
    \end{split}
    \DDup_2
    \begin{split}
    \begin{tikzpicture}[xscale=.25,yscale=.2]
        \node[Feuille](0)at(0.00,-4.67){};
        \node[Feuille](2)at(2.00,-4.67){};
        \node[Feuille](4)at(4.00,-4.67){};
        \node[Feuille](6)at(6.00,-4.67){};
        \node[Noeud](1)at(1.00,-2.33){};
        \node[Noeud](3)at(3.00,0.00){};
        \node[Noeud](5)at(5.00,-2.33){};
        \draw[Arete](0)--(1);
        \draw[Arete](1)edge[]node[EtiqArete]{\begin{math}1\end{math}}(3);
        \draw[Arete](2)--(1);
        \draw[Arete](4)--(5);
        \draw[Arete](5)edge[]node[EtiqArete]{\begin{math}2\end{math}}(3);
        \draw[Arete](6)--(5);
        \node(r)at(3.00,2){};
        \draw[Arete](r)--(3);
    \end{tikzpicture}
    \end{split}
    =
    \begin{split}
    \begin{tikzpicture}[xscale=.22,yscale=.15]
        \node[Feuille](0)at(0.00,-10.00){};
        \node[Feuille](10)at(10.00,-5.00){};
        \node[Feuille](12)at(12.00,-5.00){};
        \node[Feuille](14)at(14.00,-5.00){};
        \node[Feuille](2)at(2.00,-10.00){};
        \node[Feuille](4)at(4.00,-10.00){};
        \node[Feuille](6)at(6.00,-12.50){};
        \node[Feuille](8)at(8.00,-12.50){};
        \node[Noeud](1)at(1.00,-7.50){};
        \node[Noeud](11)at(11.00,0.00){};
        \node[Noeud](13)at(13.00,-2.50){};
        \node[Noeud](3)at(3.00,-5.00){};
        \node[Noeud](5)at(5.00,-7.50){};
        \node[Noeud](7)at(7.00,-10.00){};
        \node[Noeud](9)at(9.00,-2.50){};
        \draw[Arete](0)--(1);
        \draw[Arete](1)edge[]node[EtiqArete]{\begin{math}1\end{math}}(3);
        \draw[Arete](10)--(9);
        \draw[Arete](12)--(13);
        \draw[Arete](13)edge[]node[EtiqArete]{\begin{math}2\end{math}}(11);
        \draw[Arete](14)--(13);
        \draw[Arete](2)--(1);
        \draw[Arete](3)edge[]node[EtiqArete]{\begin{math}2\end{math}}(9);
        \draw[Arete](4)--(5);
        \draw[Arete](5)edge[]node[EtiqArete]{\begin{math}3\end{math}}(3);
        \draw[Arete](6)--(7);
        \draw[Arete](7)edge[]node[EtiqArete]{\begin{math}1\end{math}}(5);
        \draw[Arete](8)--(7);
        \draw[Arete](9)edge[]node[EtiqArete]{\begin{math}1\end{math}}(11);
        \node(r)at(11.00,2.50){};
        \draw[Arete](r)--(11);
    \end{tikzpicture}
    \end{split}\,.
\end{equation}
\medskip

\begin{Lemme} \label{lem:produit_gamma_duplicial}
    For any integer $\gamma \geq 0$, the vector space
    $\AlgLibre_{\Dup_\gamma}$ of $\gamma$-edge valued binary trees
    endowed with the operations $\GDup_a$, $\DDup_a$, $a \in [\gamma]$,
    is a $\gamma$-multiplicial algebra.
\end{Lemme}
\begin{proof}
    We have to check that the operations $\GDup_a$, $\DDup_a$,
    $a \in [\gamma]$, of $\AlgLibre_{\Dup_\gamma}$ satisfy
    Relations~\eqref{equ:relation_dup_gamma_1},
    \eqref{equ:relation_dup_gamma_2},
    and~\eqref{equ:relation_dup_gamma_3} of
    $\gamma$-multiplicial algebras. Let $\Rfr$, $\Sfr$, and $\Tfr$ be
    three $\gamma$-edge valued binary trees and $a, a' \in [\gamma]$.
    \smallskip

    Denote by $\Sfr_1$ (resp. $\Sfr_2$) the left subtree (resp. right
    subtree) of $\Sfr$ and by $x$ (resp. $y$) the label of the left
    (resp. right) edge incident to the root of $\Sfr$. We have
    \begin{multline}
        (\Rfr \DDup_{a'} \Sfr) \GDup_a \Tfr  =
        \left(\Rfr \DDup_{a'} \ArbreBinValue{x}{y}{\Sfr_1}{\Sfr_2}\right)
        \GDup_a \Tfr
        = \left(\ArbreBinValue{z}{y}{\Rfr \DDup_{a'} \Sfr_1}{\Sfr_2}
        \right)
        \GDup_a \Tfr
        \displaybreak[0]
        \\
        = \ArbreBinValue{z}{t}{\Rfr \DDup_{a'} \Sfr_1}
                {\Sfr_2 \GDup_a \Tfr}
        \displaybreak[0]
        \\
        = \Rfr \DDup_{a'}
        \left(\ArbreBinValue{x}{t}{\Sfr_1}{\Sfr_2 \GDup_a \Tfr}\right)
        = \Rfr \DDup_{a'}
        \left(\ArbreBinValue{x}{y}{\Sfr_1}{\Sfr_2}
        \GDup_a \Tfr \right) =
        \Rfr \DDup_{a'} (\Sfr \GDup_a \Tfr),
    \end{multline}
    where $z := a' \Min x$ and $t := a \Min y$. This shows
    that~\eqref{equ:relation_dup_gamma_1} is satisfied in
    $\AlgLibre_{\Dup_\gamma}$.
    \medskip

    We now prove that Relations~\eqref{equ:relation_dup_gamma_2}
    and~\eqref{equ:relation_dup_gamma_3} hold by induction on the sum of
    the number of internal nodes of $\Rfr$, $\Sfr$, and $\Tfr$. Base
    case holds when all these trees have exactly one internal node, and
    since
    \begin{multline}
        \left(\Noeud \GDup_{a'} \Noeud\right) \GDup_a \Noeud
        - \Noeud \GDup_{a \Min a'} \left( \Noeud \GDup_a \Noeud \right)
        \displaybreak[0]
        \\
        =
        \begin{split}
        \begin{tikzpicture}[xscale=.3,yscale=.25]
            \node[Feuille](0)at(0.00,-1.67){};
            \node[Feuille](2)at(2.00,-3.33){};
            \node[Feuille](4)at(4.00,-3.33){};
            \node[Noeud](1)at(1.00,0.00){};
            \node[Noeud](3)at(3.00,-1.67){};
            \draw[Arete](0)--(1);
            \draw[Arete](2)--(3);
            \draw[Arete](3)edge[]node[EtiqArete]{\begin{math}a'\end{math}}(1);
            \draw[Arete](4)--(3);
            \node(r)at(1.00,1.67){};
            \draw[Arete](r)--(1);
        \end{tikzpicture}
        \end{split}
        \GDup_a \Noeud -
        \Noeud \GDup_{a \Min a'}
        \begin{split}
        \begin{tikzpicture}[xscale=.3,yscale=.25]
            \node[Feuille](0)at(0.00,-1.67){};
            \node[Feuille](2)at(2.00,-3.33){};
            \node[Feuille](4)at(4.00,-3.33){};
            \node[Noeud](1)at(1.00,0.00){};
            \node[Noeud](3)at(3.00,-1.67){};
            \draw[Arete](0)--(1);
            \draw[Arete](2)--(3);
            \draw[Arete](3)edge[]node[EtiqArete]{\begin{math}a\end{math}}(1);
            \draw[Arete](4)--(3);
            \node(r)at(1.00,1.67){};
            \draw[Arete](r)--(1);
        \end{tikzpicture}
        \end{split}
        \displaybreak[0]
        \\
        =
        \begin{split}
        \begin{tikzpicture}[xscale=.3,yscale=.25]
            \node[Feuille](0)at(0.00,-1.75){};
            \node[Feuille](2)at(2.00,-3.50){};
            \node[Feuille](4)at(4.00,-5.25){};
            \node[Feuille](6)at(6.00,-5.25){};
            \node[Noeud](1)at(1.00,0.00){};
            \node[Noeud](3)at(3.00,-1.75){};
            \node[Noeud](5)at(5.00,-3.50){};
            \draw[Arete](0)--(1);
            \draw[Arete](2)--(3);
            \draw[Arete](3)edge[]node[EtiqArete]{\begin{math}z\end{math}}(1);
            \draw[Arete](4)--(5);
            \draw[Arete](5)edge[]node[EtiqArete]{\begin{math}a\end{math}}(3);
            \draw[Arete](6)--(5);
            \node(r)at(1.00,1.75){};
            \draw[Arete](r)--(1);
        \end{tikzpicture}
        \end{split}
        -
        \begin{split}
        \begin{tikzpicture}[xscale=.3,yscale=.25]
            \node[Feuille](0)at(0.00,-1.75){};
            \node[Feuille](2)at(2.00,-3.50){};
            \node[Feuille](4)at(4.00,-5.25){};
            \node[Feuille](6)at(6.00,-5.25){};
            \node[Noeud](1)at(1.00,0.00){};
            \node[Noeud](3)at(3.00,-1.75){};
            \node[Noeud](5)at(5.00,-3.50){};
            \draw[Arete](0)--(1);
            \draw[Arete](2)--(3);
            \draw[Arete](3)edge[]node[EtiqArete]{\begin{math}z\end{math}}(1);
            \draw[Arete](4)--(5);
            \draw[Arete](5)edge[]node[EtiqArete]{\begin{math}a\end{math}}(3);
            \draw[Arete](6)--(5);
            \node(r)at(1.00,1.75){};
            \draw[Arete](r)--(1);
        \end{tikzpicture}
        \end{split}
        = 0,
    \end{multline}
    where $z := a \Min a'$, \eqref{equ:relation_dup_gamma_2}
    holds on trees with one internal node. For the same arguments,
    we can show that~\eqref{equ:relation_dup_gamma_3} holds
    on trees with exactly one internal node. Denote now by $\Rfr_1$
    (resp. $\Rfr_2$) the left subtree (resp. right subtree) of $\Rfr$
    and by $x$ (resp. $y$) the label of the left (resp. right) edge
    incident to the root of $\Rfr$. We have
    \begin{multline} \label{equ:produit_gamma_dup_expr}
        (\Rfr \GDup_{a'} \Sfr) \GDup_a \Tfr
        - \Rfr \GDup_{a \Min a'} (\Sfr \GDup_a \Tfr)
        \displaybreak[0]
        \\[1em]
        =
        \left(\ArbreBinValue{x}{y}{\Rfr_1}{\Rfr_2}
        \GDup_{a'} \Sfr \right) \GDup_a \Tfr
        -
        \ArbreBinValue{x}{y}{\Rfr_1}{\Rfr_2}
        \GDup_{a \Min a'} (\Sfr \GDup_a \Tfr)
        \displaybreak[0]
        \\
        =
        \left(\ArbreBinValue{x}{z}{\Rfr_1}{\Rfr_2 \GDup_{a'} \Sfr}
        \right) \GDup_a \Tfr
        -
        \ArbreBinValue{x}{y}{\Rfr_1}{\Rfr_2}
        \GDup_{a \Min a'} (\Sfr \GDup_a \Tfr)
        \displaybreak[0]
        \\
        =
        \ArbreBinValue{x}{t}{\Rfr_1}{(\Rfr_2 \GDup_{a'} \Sfr) \GDup_a \Tfr}
        -
        \ArbreBinValue{x}{t}{\Rfr_1}{\Rfr_2 \GDup_u (\Sfr \GDup_a \Tfr)}\,,
    \end{multline}
    where $z := y \Min a'$, $t := z \Min a = y \Min a' \Min a$,
    and $u := a \Min a'$. Now, since by induction hypothesis
    Relation~\eqref{equ:relation_dup_gamma_2} holds on $\Rfr_2$, $\Sfr$,
    and $\Tfr$, \eqref{equ:produit_gamma_dup_expr} is zero. Therefore,
    \eqref{equ:relation_dup_gamma_2} is satisfied
    in~$\AlgLibre_{\Dup_\gamma}$.
    \smallskip

    Finally, for the same arguments, we can show
    that~\eqref{equ:relation_dup_gamma_3} is satisfied in
    $\AlgLibre_{\Dup_\gamma}$, implying the statement of the lemma.
\end{proof}
\medskip

\begin{Lemme} \label{lem:produit_gamma_engendre_duplicial}
    For any integer $\gamma \geq 0$, the $\gamma$-multiplicial algebra
    $\AlgLibre_{\Dup_\gamma}$ of $\gamma$-edge valued binary trees
    endowed with the operations $\GDup_a$, $\DDup_a$, $a \in [\gamma]$,
    is generated by
    \begin{equation}
        \Noeud\,.
    \end{equation}
\end{Lemme}
\begin{proof}
    First, Lemma~\ref{lem:produit_gamma_duplicial} shows that
    $\AlgLibre_{\Dup_\gamma}$ is a $\gamma$-multiplicial algebra.
    Let $\Mca$ be the $\gamma$-multiplicial subalgebra of
    $\AlgLibre_{\Dup_\gamma}$ generated by $\NoeudTexte$. Let us show
    that any $\gamma$-edge valued binary tree $\Tfr$ is in $\Mca$
    by induction on the number $n$ of its internal nodes. When $n = 1$,
    $\Tfr = \NoeudTexte$ and hence the property is satisfied. Otherwise,
    let $\Tfr_1$ (resp. $\Tfr_2$) be the left (resp. right) subtree of
    the root of $\Tfr$ and denote by $x$ (resp. $y$) the label of the
    left (resp. right) edge incident to the root of $\Tfr$. Since $\Tfr_1$
    and~$\Tfr_2$ have less internal nodes than $\Tfr$, by induction
    hypothesis, $\Tfr_1$ and $\Tfr_2$ are in $\Mca$. Moreover, by
    definition of the operations $\GDup_a$, $\DDup_a$, $a \in [\gamma]$,
    of $\AlgLibre_{\Dup_\gamma}$, one has
    \begin{equation}
        \begin{split}\end{split}
        \left(\Tfr_1 \DDup_x \Noeud\right) \GDup_y \Tfr_2
        =
        \begin{split}
        \begin{tikzpicture}[xscale=.5,yscale=.4]
            \node(0)at(-.50,-1.50){\begin{math}\Tfr_1\end{math}};
            \node[Feuille](2)at(2.0,-1.00){};
            \node[Noeud](1)at(1.00,.50){};
            \draw[Arete](0)edge[]node[EtiqArete]
                {\begin{math}x\end{math}}(1);
            \draw[Arete](2)--(1);
            \node(r)at(1.00,1.5){};
            \draw[Arete](r)--(1);
        \end{tikzpicture}
        \end{split}
        \GDup_y \Tfr_2
        =
        \ArbreBinValue{x}{y}{\Tfr_1}{\Tfr_2} = \Tfr,
    \end{equation}
    showing that $\Tfr$ also is in $\Mca$. Therefore, $\Mca$ is
    $\AlgLibre_{\Dup_\gamma}$, showing that $\AlgLibre_{\Dup_\gamma}$ is
    generated by $\NoeudTexte$.
\end{proof}
\medskip

\begin{Theoreme} \label{thm:algebre_dup_gamma_libre}
    For any integer $\gamma \geq 0$, the vector space
    $\AlgLibre_{\Dup_\gamma}$ of $\gamma$-valued binary trees endowed
    with the operations $\GDup_a$, $\DDup_a$, $a \in [\gamma]$, is the
    free $\gamma$-multiplicial algebra over one generator.
\end{Theoreme}
\begin{proof}
    By Lemmas~\ref{lem:produit_gamma_duplicial}
    and~\ref{lem:produit_gamma_engendre_duplicial},
    $\AlgLibre_{\Dup_\gamma}$ is a $\gamma$-multiplicial algebra over
    one generator.
    \smallskip

    Moreover, since by Proposition~\ref{prop:proprietes_dup_gamma}, for
    any $n \geq 1$, the dimension of $\AlgLibre_{\Dup_\gamma}(n)$ is the
    same as the dimension of $\Dup_\gamma(n)$, there cannot be relations
    in $\AlgLibre_{\Dup_\gamma}(n)$ involving $\Gfr$ that are not
    $\gamma$-multiplicial relations (see~\eqref{equ:relation_dup_gamma_1},
    \eqref{equ:relation_dup_gamma_2}, and~\eqref{equ:relation_dup_gamma_3}).
    Hence, $\AlgLibre_{\Dup_\gamma}$ is free as a $\gamma$-multiplicial
    algebra over one generator.
\end{proof}
\medskip

%%%%%%%%%%%%%%%%%%%%%%%%%%%%%%%%%%%%%%%%%%%%%%%%%%%%%%%%%%%%%%%%%%%%%%%%
%%%%%%%%%%%%%%%%%%%%%%%%%%%%%%%%%%%%%%%%%%%%%%%%%%%%%%%%%%%%%%%%%%%%%%%%
\subsection{Polytridendriform operads}
We propose here a generalization $\TDendr_\gamma$ on a nonnegative
integer parameter $\gamma$ of the tridendriform operad~\cite{LR04}.
This last operad is the Koszul dual of the triassociative operad. We proceed
by using an analogous strategy as the one used to define the operads
$\Dendr_\gamma$ as Koszul duals of $\Dias_\gamma$. Indeed, we define
$\TDendr_\gamma$ as the Koszul dual of the operad $\Trias_\gamma$,
called $\gamma$-pluritriassociative operad, a generalization of the
triassociative operad defined in~\cite{GirI}.
\medskip

Since the proofs of the results contained in this section are very
similar to the ones of Section~\ref{sec:dendr_gamma}, we omit proofs here.
\medskip

Theorem~4.2.1 of~\cite{GirI}, by exhibiting
a presentation of $\Trias_\gamma$, shows that this operad is binary and
quadratic. It then admits a Koszul dual, denoted by $\TDendr_\gamma$ and
called {\em $\gamma$-polytridendriform operad}.
\medskip

\begin{Theoreme} \label{thm:presentation_tdendr_gamma}
    For any integer $\gamma \geq 0$, the operad $\TDendr_\gamma$ admits
    the following presentation. It is generated by
    $\GenTDendr := \GenTDendr(2) :=
    \{\GDendrA_a, \MTDendr, \DDendrA_a : a \in [\gamma]\}$ and its space
    of relations $\RelTDendr$ is generated by
    \begin{subequations}
    \begin{equation}\label{equ:relation_presentation_tdendr_gamma_1}
        \MTDendr \circ_1 \MTDendr
         -
        \MTDendr \circ_2 \MTDendr,
    \end{equation}
    \begin{equation}\label{equ:relation_presentation_tdendr_gamma_2}
        \GDendrA_a \circ_1 \MTDendr
         -
        \MTDendr \circ_2 \GDendrA_a,
        \qquad a \in [\gamma],
    \end{equation}
    \begin{equation}\label{equ:relation_presentation_tdendr_gamma_3}
        \MTDendr \circ_1 \DDendrA_a
         -
        \DDendrA_a \circ_2 \MTDendr,
        \qquad a \in [\gamma],
    \end{equation}
    \begin{equation}\label{equ:relation_presentation_tdendr_gamma_4}
        \MTDendr \circ_1 \GDendrA_a
         -
        \MTDendr \circ_2 \DDendrA_a,
        \qquad a \in [\gamma],
    \end{equation}
    \begin{equation}\label{equ:relation_presentation_tdendr_gamma_5}
        \GDendrA_a \circ_1 \DDendrA_{a'}
         -
        \DDendrA_{a'} \circ_2 \GDendrA_a,
        \qquad a, a' \in [\gamma],
    \end{equation}
    \begin{equation}\label{equ:relation_presentation_tdendr_gamma_6}
        \GDendrA_a \circ_1 \GDendrA_b
         -
        \GDendrA_a \circ_2 \DDendrA_b,
        \qquad a < b \in [\gamma],
    \end{equation}
    \begin{equation}\label{equ:relation_presentation_tdendr_gamma_7}
        \DDendrA_a \circ_1 \GDendrA_b
         -
        \DDendrA_a \circ_2 \DDendrA_b,
        \qquad a < b \in [\gamma],
    \end{equation}
    \begin{equation}\label{equ:relation_presentation_tdendr_gamma_8}
        \GDendrA_b \circ_1 \GDendrA_a
         -
        \GDendrA_a \circ_2 \GDendrA_b,
        \qquad a < b \in [\gamma],
    \end{equation}
    \begin{equation}\label{equ:relation_presentation_tdendr_gamma_9}
        \DDendrA_a \circ_1 \DDendrA_b
         -
        \DDendrA_b \circ_2 \DDendrA_a,
        \qquad a < b \in [\gamma],
    \end{equation}
    \begin{equation}\label{equ:relation_presentation_tdendr_gamma_10}
        \GDendrA_d \circ_1 \GDendrA_d -
        \GDendrA_d \circ_2 \MTDendr -
        \left(\sum_{c \in [d]}
        \GDendrA_d \circ_2 \GDendrA_c +
        \GDendrA_d \circ_2 \DDendrA_c
        \right),
        \qquad d \in [\gamma],
    \end{equation}
    \begin{equation}\label{equ:relation_presentation_tdendr_gamma_11}
        \left(\sum_{c \in [d]}
        \DDendrA_d \circ_1 \GDendrA_c +
        \DDendrA_d \circ_1 \DDendrA_c
        \right) +
        \DDendrA_d \circ_1 \MTDendr -
        \DDendrA_d \circ_2 \DDendrA_d,
        \qquad d \in [\gamma].
    \end{equation}
    \end{subequations}
\end{Theoreme}
\medskip

\begin{Proposition} \label{prop:serie_hilbert_tdendr_gamma}
    For any integer $\gamma \geq 0$, the Hilbert series
    $\Hca_{\TDendr_\gamma}(t)$ of the operad $\TDendr_\gamma$ satisfies
    \begin{equation}
        \Hca_{\TDendr_\gamma}(t) =
        t + (2\gamma + 1) t \, \Hca_{\TDendr_\gamma}(t) +
        \gamma (\gamma + 1) t \, \Hca_{\TDendr_\gamma}(t)^2.
    \end{equation}
\end{Proposition}
\medskip

By examining the expression for $\Hca_{\TDendr_\gamma}(t)$ of the
statement of Proposition~\ref{prop:serie_hilbert_tdendr_gamma}, we
observe that for any $n \geq 1$, $\TDendr(n)$ can be seen as the
vector space $\AlgLibre_{\TDendr_\gamma}(n)$ of Schröder trees
with $n + 1$ leaves wherein its edges connecting two internal nodes
are labeled on $[\gamma]$. We call these trees
{\em $\gamma$-edge valued Schröder trees}. For instance,
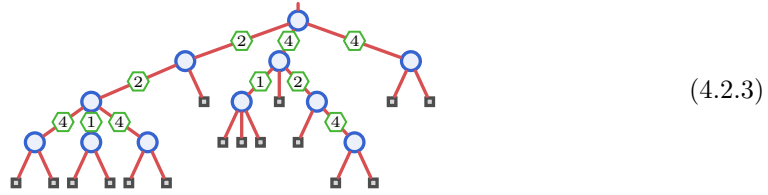
\begin{equation}
    \begin{split}
    \begin{tikzpicture}[xscale=.25,yscale=.10]
        \node[Feuille](0)at(0.00,-21.60){};
        \node[Feuille](11)at(10.00,-10.80){};
        \node[Feuille](13)at(11.00,-16.20){};
        \node[Feuille](15)at(12.00,-16.20){};
        \node[Feuille](16)at(13.00,-16.20){};
        \node[Feuille](18)at(14.00,-10.80){};
        \node[Feuille](19)at(15.00,-16.20){};
        \node[Feuille](2)at(2.00,-21.60){};
        \node[Feuille](21)at(17.00,-21.60){};
        \node[Feuille](23)at(19.00,-21.60){};
        \node[Feuille](24)at(20.00,-10.80){};
        \node[Feuille](26)at(22.00,-10.80){};
        \node[Feuille](4)at(3.00,-21.60){};
        \node[Feuille](6)at(5.00,-21.60){};
        \node[Feuille](7)at(6.00,-21.60){};
        \node[Feuille](9)at(8.00,-21.60){};
        \node[Noeud](1)at(1.00,-16.20){};
        \node[Noeud](10)at(9.00,-5.40){};
        \node[Noeud](12)at(15.00,0.00){};
        \node[Noeud](14)at(12.00,-10.80){};
        \node[Noeud](17)at(14.00,-5.40){};
        \node[Noeud](20)at(16.00,-10.80){};
        \node[Noeud](22)at(18.00,-16.20){};
        \node[Noeud](25)at(21.00,-5.40){};
        \node[Noeud](3)at(4.00,-10.80){};
        \node[Noeud](5)at(4.00,-16.20){};
        \node[Noeud](8)at(7.00,-16.20){};
        \draw[Arete](0)--(1);
        \draw[Arete](1)edge[]node[EtiqArete]{\begin{math}4\end{math}}(3);
        \draw[Arete](10)edge[]node[EtiqArete]{\begin{math}2\end{math}}(12);
        \draw[Arete](11)--(10);
        \draw[Arete](13)--(14);
        \draw[Arete](14)edge[]node[EtiqArete]{\begin{math}1\end{math}}(17);
        \draw[Arete](15)--(14);
        \draw[Arete](16)--(14);
        \draw[Arete](17)edge[]node[EtiqArete]{\begin{math}4\end{math}}(12);
        \draw[Arete](18)--(17);
        \draw[Arete](19)--(20);
        \draw[Arete](2)--(1);
        \draw[Arete](20)edge[]node[EtiqArete]{\begin{math}2\end{math}}(17);
        \draw[Arete](21)--(22);
        \draw[Arete](22)edge[]node[EtiqArete]{\begin{math}4\end{math}}(20);
        \draw[Arete](23)--(22);
        \draw[Arete](24)--(25);
        \draw[Arete](25)edge[]node[EtiqArete]{\begin{math}4\end{math}}(12);
        \draw[Arete](26)--(25);
        \draw[Arete](3)edge[]node[EtiqArete]{\begin{math}2\end{math}}(10);
        \draw[Arete](4)--(5);
        \draw[Arete](5)edge[]node[EtiqArete]{\begin{math}1\end{math}}(3);
        \draw[Arete](6)--(5);
        \draw[Arete](7)--(8);
        \draw[Arete](8)edge[]node[EtiqArete]{\begin{math}4\end{math}}(3);
        \draw[Arete](9)--(8);
        \node(r)at(15.00,3.5){};
        \draw[Arete](r)--(12);
    \end{tikzpicture}
    \end{split}
\end{equation}
is a $4$-edge valued Schröder tree and a basis element of $\TDendr_4(16)$.
\medskip

We deduce from Proposition~\ref{prop:serie_hilbert_tdendr_gamma} that
\begin{equation}
    \Hca_{\TDendr_\gamma}(t) =
    \frac{1 - \sqrt{1 - (4\gamma + 2)t + t^2} - (2\gamma + 1)t}
    {2(\gamma + \gamma^2)t}.
\end{equation}
Moreover, we obtain that for all $n \geq 1$,
\begin{equation}
    \dim \TDendr_\gamma(n) =
    \sum_{k = 0}^{n - 1} (\gamma + 1)^k \gamma^{n - k - 1} \, \Nar(n, k),
\end{equation}
where $\Nar(n, k)$ is defined in~\eqref{equ:definition_narayana}.
For instance, the first dimensions of $\TDendr_1$, $\TDendr_2$,
$\TDendr_3$, and $\TDendr_4$ are respectively
\begin{equation}
    1, 3, 11, 45, 197, 903, 4279, 20793, 103049, 518859, 2646723,
\end{equation}
\begin{equation}
    1, 5, 31, 215, 1597, 12425, 99955, 824675, 6939769, 59334605, 513972967,
\end{equation}
\begin{equation}
    1, 7, 61, 595, 6217, 68047, 770149, 8939707, 105843409,
    1273241431, 15517824973,
\end{equation}
\begin{equation}
    1, 9, 101, 1269, 17081, 240849, 3511741, 52515549, 801029681,
    12414177369, 194922521301.
\end{equation}
The first one is Sequence~\Sloane{A001003} of~\cite{Slo}. The others
sequences are not listed in~\cite{Slo} at this time.
\medskip

%%%%%%%%%%%%%%%%%%%%%%%%%%%%%%%%%%%%%%%%%%%%%%%%%%%%%%%%%%%%%%%%%%%%%%%%
%%%%%%%%%%%%%%%%%%%%%%%%%%%%%%%%%%%%%%%%%%%%%%%%%%%%%%%%%%%%%%%%%%%%%%%%
\subsection{Operads of the operadic butterfly}
The {\em operadic butterfly}~\cite{Lod01,Lod06} is a diagram gathering
seven famous operads. We have seen in
Section~\ref{subsec:diagramme_dias_as_dendr_gamma} that this diagram
gathers the diassociative, associative, and dendriform operads. It
involves also the {\em commutative operad} $\Com$, the {\em Lie operad}
$\Lie$, the {\em Zinbiel operad} $\Zin$~\cite{Lod95}, and
the {\em Leibniz operad} $\Leib$~\cite{Lod93}. It is of the form
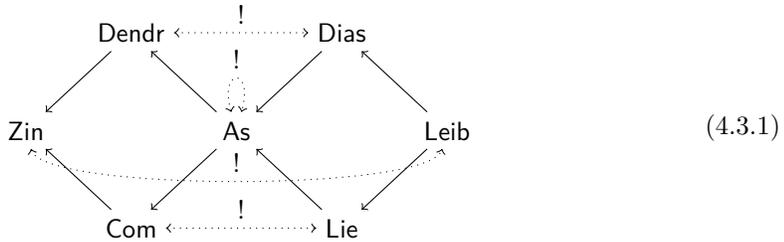
\begin{equation} \label{equ:diagramme_papillon}
    \begin{split}
    \begin{tikzpicture}[xscale=.7,yscale=.65]
        \node(Dendr)at(-2,2){\begin{math}\Dendr\end{math}};
        \node(As)at(0,0){\begin{math}\As\end{math}};
        \node(Dias)at(2,2){\begin{math}\Dias\end{math}};
        \node(Com)at(-2,-2){\begin{math}\Com\end{math}};
        \node(Lie)at(2,-2){\begin{math}\Lie\end{math}};
        \node(Zin)at(-4,0){\begin{math}\Zin\end{math}};
        \node(Leib)at(4,0){\begin{math}\Leib\end{math}};
        \draw[->](Dias)--(As);
        \draw[->](As)--(Dendr);
        \draw[->](Dendr)--(Zin);
        \draw[->](Leib)--(Dias);
        \draw[->](Com)--(Zin);
        \draw[->](As)--(Com);
        \draw[->](Lie)--(As);
        \draw[->](Leib)--(Lie);
        \draw[<->,dotted,loop above,looseness=13](As)
            edge node[anchor=south]{\begin{math} ! \end{math}}(As);
        \draw[<->,dotted](Dendr)
            edge node[anchor=south]{\begin{math} ! \end{math}}(Dias);
        \draw[<->,dotted](Com)
            edge node[anchor=south]{\begin{math} ! \end{math}}(Lie);
        \draw[<->,dotted,loop below,looseness=.3](Zin)
            edge node[anchor=south]{\begin{math} ! \end{math}}(Leib);
    \end{tikzpicture}
    \end{split}
\end{equation}
and as it shows, some operads are Koszul dual of some others
(in particular, $\Com^! = \Lie$ and $\Zin^! = \Leib$).
\medskip

We have to emphasize the fact the operads $\Com$, $\Lie$, $\Zin$, and
$\Leib$ of the operadic butterfly are symmetric operads. The computation
of the Koszul dual of a symmetric operad does not follows what we have
presented in Section~\ref{subsec:dual_de_Koszul}. We invite the
reader to consult~\cite{GK94} or~\cite{LV12} for a complete
description.
\medskip

For simplicity, in what follows, we shall consider algebras over
symmetric operads instead of symmetric operads.
\medskip

%%%%%%%%%%%%%%%%%%%%%%%%%%%%%%%%%%%%%%%%%%%%%%%%%%%%%%%%%%%%%%%%%%%%%%%%
\subsubsection{A generalization of the operadic butterfly}
A possible continuation to this work consists in constructing a diagram
\begin{equation} \label{equ:diagramme_papillon_gamma}
    \begin{split}
    \begin{tikzpicture}[xscale=.7,yscale=.65]
        \node(Dendr)at(-4,2){\begin{math}\Dendr_\gamma\end{math}};
        \node(As)at(1.5,0){\begin{math}\As_\gamma\end{math}};
        \node(DAs)at(-1.5,0){\begin{math}\DAs_\gamma\end{math}};
        \node(Dias)at(4,2){\begin{math}\Dias_\gamma\end{math}};
        \node(Com)at(-4,-2){\begin{math}\Com_\gamma\end{math}};
        \node(Lie)at(4,-2){\begin{math}\Lie_\gamma\end{math}};
        \node(Zin)at(-6,0){\begin{math}\Zin_\gamma\end{math}};
        \node(Leib)at(6,0){\begin{math}\Leib_\gamma\end{math}};
        \draw[->](Dias)--(As);
        \draw[->](DAs)--(Dendr);
        \draw[->](Dendr)--(Zin);
        \draw[->](Leib)--(Dias);
        \draw[->](Com)--(Zin);
        \draw[->](DAs)--(Com);
        \draw[->](Lie)--(As);
        \draw[->](Leib)--(Lie);
        \draw[<->,dotted](As)
            edge node[anchor=south]{\begin{math} ! \end{math}}(DAs);
        \draw[<->,dotted](Dendr)
            edge node[anchor=south]{\begin{math} ! \end{math}}(Dias);
        \draw[<->,dotted](Com)
            edge node[anchor=south]{\begin{math} ! \end{math}}(Lie);
        \draw[<->,dotted,loop below,looseness=.2](Zin)
            edge node[anchor=south]{\begin{math} ! \end{math}}(Leib);
    \end{tikzpicture}
    \end{split}
\end{equation}
where $\DAs_\gamma$ is the $\gamma$-dual multiassociative operad defined
in Section~\ref{subsubsec:das_gamma} and $\Com_\gamma$, $\Lie_\gamma$,
$\Zin_\gamma$, and $\Leib_\gamma$, respectively are generalizations on a
nonnegative integer parameter $\gamma$ of the operads $\Com$, $\Lie$,
$\Zin$, and $\Leib$. Let us now define these operads.
\medskip

%%%%%%%%%%%%%%%%%%%%%%%%%%%%%%%%%%%%%%%%%%%%%%%%%%%%%%%%%%%%%%%%%%%%%%%%
\subsubsection{Commutative and Lie operads} \label{subsubsec:com_gamma}
The symmetric operad $\Com$ is the symmetric operad describing the
category of algebras $\Cca$ with one binary operation $\MDAs$, subjected
for any elements $x$, $y$, and $z$ of $\Cca$ to the two relations
\begin{subequations}
\begin{equation}
    x \MDAs y = y \MDAs x,
\end{equation}
\begin{equation}
    (x \MDAs y) \MDAs z = x \MDAs (y \MDAs z).
\end{equation}
\end{subequations}
This operad has the property to be a commutative version of $\As = \DAs_1$.
\medskip

We define the symmetric operad $\Com_\gamma$ by using the same idea of
being a commutative version of $\DAs_\gamma$. Therefore, $\Com_\gamma$
is the symmetric operad describing the category of algebras $\Cca$ with
binary operations $\MDAs_a$, $a \in [\gamma]$, subjected for any elements
$x$, $y$, and $z$ of $\Cca$ to the two sorts of relations
\begin{subequations}
\begin{equation} \label{equ:relation_com_gamma_1}
    x \MDAs_a y = y \MDAs_a x,
    \qquad a \in [\gamma],
\end{equation}
\begin{equation} \label{equ:relation_com_gamma_2}
    (x \MDAs_a y) \MDAs_a z = x \MDAs_a (y \MDAs_a z),
    \qquad a \in [\gamma].
\end{equation}
\end{subequations}
Moreover, we define the symmetric operad $\Lie_\gamma$ as the Koszul dual
of $\Com_\gamma$.
\medskip

%%%%%%%%%%%%%%%%%%%%%%%%%%%%%%%%%%%%%%%%%%%%%%%%%%%%%%%%%%%%%%%%%%%%%%%%
\subsubsection{Zinbiel and Leibniz operads} \label{subsubsec:zin_gamma}
The symmetric operad $\Zin$ is the symmetric operad describing
the category of algebras $\Zca$ with one generating binary operation
$\ProdZin$, subjected for any elements $x$, $y$, and $z$ of $\Zca$ to
the relation
\begin{equation} \label{equ:relation_zinbiel}
    (x \ProdZin y) \ProdZin z =
    x \ProdZin (y \ProdZin z) + x \ProdZin (z \ProdZin y).
\end{equation}
This operad has the property to be a commutative version of
$\Dendr = \Dendr_1$. Indeed, Relation~\eqref{equ:relation_zinbiel} is
obtained from Relations~\eqref{equ:relation_dendr_1},
\eqref{equ:relation_dendr_2}, and~\eqref{equ:relation_dendr_3} of
dendriform algebras with the condition that for any elements $x$ and $y$,
$x \GDendr y = y \DDendr x$, and by setting $x \ProdZin y := x \GDendr y$.
\medskip

We define the symmetric operad $\Zin_\gamma$ by using the same idea of
having the property to be a commutative version of $\Dendr_\gamma$.
Therefore, $\Zin_\gamma$ is the symmetric operad describing the
category of algebras $\Zca$ with binary operations $\ProdZin_a$,
$a \in [\gamma]$, subjected for any elements $x$, $y$, and $z$ of $\Zca$
to the relation
\begin{equation} \label{equ:relation_zinbiel_gamma}
    (x \ProdZin_{a'} y) \ProdZin_a z =
    x \ProdZin_{a \Min a'} (y \ProdZin_a z) +
    x \ProdZin_{a \Min a'} (z \ProdZin_{a'} y),
    \qquad a, a' \in [\gamma].
\end{equation}
Relation~\eqref{equ:relation_zinbiel_gamma} is obtained
from Relations~\eqref{equ:relation_dendr_gamma_1_concise},
\eqref{equ:relation_dendr_gamma_2_concise},
and~\eqref{equ:relation_dendr_gamma_3_concise} of
$\gamma$-polydendriform algebras with the condition that for any elements
$x$ and $y$ and $a \in [\gamma]$, $x \GDendr_a y = y \DDendr_a x$, and
by setting $x \ProdZin_a y := x \GDendr_a y$. Moreover, we define the
symmetric operad $\Leib_\gamma$ as the Koszul dual of $\Zin_\gamma$.
\medskip

\begin{Proposition} \label{prop:morphism_com_zin_gamma}
    For any integer $\gamma \geq 0$ and any $\Zin_\gamma$-algebra
    $\Zca$, the binary operations $\MDAs_a$, $a \in [\gamma]$, defined
    for all elements $x$ and $y$ of $\Zca$ by
    \begin{equation}
        x \MDAs_a y := x \ProdZin_a y + y \ProdZin_a x,
        \qquad a \in [\gamma],
    \end{equation}
    endow $\Zca$ with a $\Com_\gamma$-algebra structure.
\end{Proposition}
\begin{proof}
    Since for all $a \in [\gamma]$ and all elements $x$ and $y$ of $\Zca$,
    by~\eqref{equ:relation_zinbiel_gamma}, we have
    \begin{equation}
        x \MDAs_a y - y \MDAs_a x =
        x \ProdZin_a y + y \ProdZin_a x - y \ProdZin_a x - x \ProdZin_a y
        = 0,
    \end{equation}
    the operations $\MDAs_a$ satisfy Relation~\eqref{equ:relation_com_gamma_1}
    of $\Com_\gamma$-algebras. Moreover, since for all $a \in [\gamma]$
    and all elements $x$, $y$, and $z$ of $\Zca$,
    by~\eqref{equ:relation_zinbiel_gamma}, we have
    \begin{equation}\begin{split}
        (x \MDAs_a y) \MDAs_a z
        & - x \MDAs_a (y \MDAs_a z) \\
        & =
        (x \ProdZin_a y + y \ProdZin_a x) \ProdZin_a z +
        z \ProdZin_a (x \ProdZin_a y + y \ProdZin_a x) \\
        & \qquad - x \ProdZin_a (y \ProdZin_a z + z \ProdZin_a y)
        - (y \ProdZin_a z + z \ProdZin_a y) \ProdZin_a x \\
        & = (x \ProdZin_a y) \ProdZin_a z + (y \ProdZin_a x) \ProdZin_a z
        + z \ProdZin_a (x \ProdZin_a y) + z \ProdZin_a (y \ProdZin_a x) \\
        & \qquad - x \ProdZin_a (y \ProdZin_a z) - x \ProdZin_a (z \ProdZin_a y)
        - (y \ProdZin_a z) \ProdZin_a x - (z \ProdZin_a y) \ProdZin_a x \\
        & = (y \ProdZin_a x) \ProdZin_a z - (y \ProdZin_a z) \ProdZin_a x \\
        & = y \ProdZin_a (x \ProdZin_a z) + y \ProdZin_a (z \ProdZin_a x)
        - y \ProdZin_a (z \ProdZin_a x) - y \ProdZin_a (x \ProdZin_a z) \\
        & = 0,
    \end{split}\end{equation}
    the operations $\MDAs_a$ satisfy Relation~\eqref{equ:relation_com_gamma_2}
    of $\Com_\gamma$-algebras.
    Hence, $\Zca$ is a $\Com_\gamma$-algebra.
\end{proof}
\medskip

\begin{Proposition} \label{prop:morphism_dendr_zin_gamma}
    For any integer $\gamma \geq 0$, and any $\Zin_\gamma$-algebra
    $\Zca$, the binary operations $\GDendr_a$, $\DDendr_a$,
    $a \in [\gamma]$ defined for all elements $x$ and $y$ of $\Zca$ by
    \begin{equation}
        x \GDendr_a y := x \ProdZin_a y,
        \qquad a \in [\gamma],
    \end{equation}
    and
    \begin{equation}
        x \DDendr_a y := y \ProdZin_a x,
        \qquad a \in [\gamma],
    \end{equation}
    endow $\Zca$ with a $\gamma$-polydendriform algebra structure.
\end{Proposition}
\begin{proof}
    Since, for all $a, a' \in [\gamma]$ and all elements $x$, $y$,
    and $z$ of $\Zca$, by~\eqref{equ:relation_zinbiel_gamma}, we have
    \begin{equation}\begin{split}
        (x \DDendr_{a'} y) \GDendr_a z
            & - x \DDendr_{a'} (y \GDendr_a z) \\
        & = (y \ProdZin_{a'} x) \ProdZin_a z
            - (y \ProdZin_a z) \ProdZin_{a'} x \\
        & = y \ProdZin_{a \Min a'} (x \ProdZin_a z)
            + y \ProdZin_{a \Min a'} (z \ProdZin_{a'} x)
            - y \ProdZin_{a \Min a'} (z \ProdZin_{a'} x)
            - y \ProdZin_{a \Min a'} (x \ProdZin_a z) \\
        & = 0,
    \end{split}\end{equation}
    the operations $\GDendr_a$ and $\DDendr_a$ satisfy
    Relation~\eqref{equ:relation_dendr_gamma_1_concise} of
    $\gamma$-polydendriform algebras. Moreover, since for all
    $a, a' \in [\gamma]$ and all elements $x$, $y$,
    and $z$ of $\Zca$, by~\eqref{equ:relation_zinbiel_gamma}, we have
    \begin{equation}\begin{split}
        (x \GDendr_{a'} y) \GDendr_a z
            & - x \GDendr_{a \Min a'} (y \GDendr_a z)
            - x \GDendr_{a \Min a'} (y \DDendr_{a'} z) \\
        & = (x \ProdZin_{a'} y) \ProdZin_a z
            - x \ProdZin_{a \Min a'} (y \ProdZin_a z)
            - x \ProdZin_{a \Min a'} (z \ProdZin_{a'} y) \\
        & = x \ProdZin_{a \Min a'} (y \ProdZin_a z)
            + x \ProdZin_{a \Min a'} (z \ProdZin_{a'} y)
            - x \ProdZin_{a \Min a'} (y \ProdZin_a z)
            - x \ProdZin_{a \Min a'} (z \ProdZin_{a'} y) \\
        & = 0,
    \end{split}\end{equation}
    the operations $\GDendr_a$ and $\DDendr_a$ satisfy
    Relation~\eqref{equ:relation_dendr_gamma_2_concise} of
    $\gamma$-polydendriform algebras. Finally, since for all
    $a, a' \in [\gamma]$ and all elements $x$, $y$,
    and $z$ of $\Zca$, we have
    \begin{equation}\begin{split}
        (x \GDendr_{a'} y) \DDendr_{a \Min a'} z
            & + (x \DDendr_a y) \DDendr_{a \Min a'} z
            - x \DDendr_a (y \DDendr_{a'} z) \\
        & = z \ProdZin_{a \Min a'} (x \ProdZin_{a'} y)
            + z \ProdZin_{a \Min a'} (y \ProdZin_a x)
            - (z \ProdZin_{a'} y) \ProdZin_a x \\
        & = z \ProdZin_{a \Min a'} (x \ProdZin_{a'} y)
            + z \ProdZin_{a \Min a'} (y \ProdZin_a x)
            - z \ProdZin_{a \Min a'} (y \ProdZin_a x)
            - z \ProdZin_{a \Min a'} (x \ProdZin_{a'} y) \\
        & = 0,
    \end{split}\end{equation}
    the operations $\GDendr_a$ and $\DDendr_a$ satisfy
    Relation~\eqref{equ:relation_dendr_gamma_3_concise} of
    $\gamma$-polydendriform algebras. Hence $\Zca$ is a
    $\gamma$-polydendriform algebra.
\end{proof}
\medskip

The constructions stated by Propositions~\ref{prop:morphism_com_zin_gamma}
and~\ref{prop:morphism_dendr_zin_gamma} producing from a
$\Zin_\gamma$-algebra respectively a $\Com_\gamma$-algebra and a
$\gamma$-polydendriform algebra are functors from the category of
$\Zin_\gamma$-algebras respectively to the category of
$\Com_\gamma$-algebras and the category of $\gamma$-polydendriform algebras.
These functors respectively translate into symmetric operad morphisms
from $\Com_\gamma$ to $\Zin_\gamma$ and from $\Dendr_\gamma$ to
$\Zin_\gamma$. These morphisms are generalizations of known morphisms
between $\Com$, $\Dendr$, and $\Zin$ of~\eqref{equ:diagramme_papillon}
(see~\cite{Lod01,Lod06,Zin12}).
\medskip

A complete study of the operads $\Com_\gamma$, $\Lie_\gamma$,
$\Zin_\gamma$, and $\Leib_\gamma$, and suitable definitions for all the
morphisms intervening in~\eqref{equ:diagramme_papillon_gamma} is worth
to interest for future works.
\medskip

%%%%%%%%%%%%%%%%%%%%%%%%%%%%%%%%%%%%%%%%%%%%%%%%%%%%%%%%%%%%%%%%%%%%%%%%
%%%%%%%%%%%%%%%%%%%%%%%%%%%%%%%%%%%%%%%%%%%%%%%%%%%%%%%%%%%%%%%%%%%%%%%%
%%%%%%%%%%%%%%%%%%%%%%%%%%%%%%%%%%%%%%%%%%%%%%%%%%%%%%%%%%%%%%%%%%%%%%%%
\bibliographystyle{alpha}
\bibliography{Bibliographie}

\end{document}